%% file: waves-with-Misha.tex
\documentclass[12pt]{amsart}

\usepackage[dvipsnames]{xcolor}

\usepackage{imakeidx}
\usepackage[backend=biber,style=alphabetic,eprint=true,url=true,doi=true, maxalphanames=4]{biblatex}
\input{macros/basic_macros}

\input{macros/annotation_macros}
\input{local_macros}

\input{misha_needed_macros}

\graphicspath{{figs/}}
\usepackage[normalem]{ulem}

\usepackage{mathtools}

\makeindex

\def\note#1
{\marginpar
{\nt $\leftarrow$
\par
\hfuzz=20pt \hbadness=9000 \hyphenpenalty=-100 \exhyphenpenalty=-100
\pretolerance=-1 \tolerance=9999 \doublehyphendemerits=-100000
\finalhyphendemerits=-100000 \baselineskip=6pt
#1}\hfuzz=1pt}

\title[Parabolically bounded primitive bounds]{MLC for parabolically bounded primitive renormalization}
\author{Jeremy Kahn, Alex Kapiamba, and Mikhail Lyubich}

\begin{document}
\maketitle

\begin{abstract}
    We prove {\em a priori} bounds and MLC  (local connectivity of the Mandelbrot set $\MM$)
    for a class of infinitely renormalizable parameters whose renormalization type is primitive
    but can approach the cusp of $\MM$. To this end we develop and refine a variety of tools 
    that allow us to control degeneration of renormalizations. They  
    include the Thin-Thick  Decomposition, the Value Calculus,
    the Wanderers Theorem, and the Wave Lemma.  
\end{abstract}

\tableofcontents

\section{Introduction} \label{sec:intro}

Since {it was conjectured} by Douady and Hubbard in \cite{Orsay}, the MLC conjecture on the local connectivity of the Mandelbrot set $\MM$ has been a fundamental open problem in Complex Dynamics. The first major breakthrough in the problem was made by Yoccoz, establishing MLC for any parameter in $\MM$ that is not \textit{infinitely quadratic-like renormalizable}   
\cite{Hubbard-Yoccoz, Milnor}.
Lyubich then showed that MLC holds for parameters in $\MM$ with \textit{a priori bounds} and satisfying a certain \textit{secondary limbs condition {(SLC)}},
{and established {\em a priori} bounds for {\em high type} maps of  the SLC class} \cite{lyubich1997dynamics}.

Quadratic-like renormalization can be divided into two types, \textit{primitive} and \textit{satellite}. 
The next move towards MLC was made by Kahn, establishing \textit{a priori} bounds at all parameters of \textit{bounded primitive} type \cite{K06}.
{This paper (along with \cite{KL:QA}, 
    which proved
  {\em Quasi-Additivity Law} and the  {\em Covering Lemma})
  introduced to the field the machinery of the {\em Near-Degenerate Regime};
  this has since become the dominant paradigm for further development.}
These results were then expanded in \cite{KL:decorations} and \cite{KL:molecules} to all parameters of ``{definitely primitive}" type satisfying an \textit{anti-molecule condition}. 
Recently, Dudko and Lyubich have produced \textit{a priori} bounds for all parameters of bounded satellite type \cite{dudko2023mlcfeigenbaumpoints}.
All the above parameters satisfy the secondary limbs condition, and hence MLC holds there.

In this article, we provide {\textit{a priori}}  bounds for some \textit{near-parabolic} primitive types, which do not satisfy the anti-molecule condition of \cite{KL:molecules}.
Notably, these are the first {controlled}%
\footnote
 {{Using methods of \cite{lyubich1997dynamics} one can produce a class of examples with uncontrolled
 (in terms of the limb rotation number) growth of renormalization combinatorics.}}
examples of {quadratic-like} \emph{a priori} bounds for parameters that do not satisfy {SLC};
we observe, {however},  that the argument in \cite{lyubich1997dynamics} can still be used to upgrade  \textit{a priori} bounds to MLC.

Let us also remark that  \textit{a priori} bounds
(in the traditional quadratic-like sense)
 {have} not been 
the only successful framework for approaching MLC. 
Some {satellite parameters}  
have been studied in \cite{CS_satellite} and \cite{DL_pacmen} in the framework of  
 \textit{Near-Parabolic Renormalization} developed in \cite{IS} 
and 
\textit{Pacman Renormalization} developed  in \cite{DLS_pacman} respectively.
They are based upon different kind of renormalization and notion of {\em a priori} bounds.
Unification of all these renormalization theories is one of the
challenges on the road to the MLC.
(Compare with \S \ref{sec:future perspective} outlining our future perspective.)

\subsection{Statement of the main result}

Let $M\subsetneq \MM$ be a small Mandelbrot set, we will call $M$ a \textit{(combinatorial) type}\index{Mandelbrot set $\MM$!combinatorial type $M\subset \MM$}. There is a unique $p\geq 2$, called the \textit{associated period}, such that every polynomial corresponding to a parameter in $M$ is quadratic-like renormalizable with period $p$. The type is \textit{prime} \index{Mandelbrot set $\MM$!combinatorial type $M\subset \MM$!prime} when there is no other type $M'\subsetneq \MM$ satisfying $M\subsetneq M'$.

Let $\CC$ be a family of types, and let $f_0:U_0\to V_0$ be a quadratic-like map with quadratic-like restrictions $(f_{n+1}:= f_n^{p_n}: U_n\to V_n)_{n=0}^N$. The sequence $(f_n)$ has \textit{combinatorics in $\CC$} if each $f_n$ is hybrid conjugate to a quadratic polynomial in some $M_n\in \CC$ with associated period $p_n$. When $N= \infty$, the sequence $(f_n)$ has \textit{a priori bounds} \index{\emph{a priori} bounds} if the restrictions can be selected so that 
$$\mmod(V_n\sm \overline{U_n})\geq \eps >0.$$
If the bounds depend only on $\CC$ when $n$ is large enough {(depending on $\mod (V_0 \sm \overline {U_0} )$),} 
then we will say that $\CC$ has \textit{beau bounds}\index{family $\CC$ of types!beau bounds}.

\renewcommand{\Card}{\dd \HC_0}
\newcommand{\HC}{\Delta}
{The \emph{main hyperbolic component} $\HC_0$ 
of $\MM$ 
is the component of $\Int \MM$
corresponding to quadratic polynomials with an attracting fixed point. 
For each $p/q\in \Q$, there is a unique parameter $c_{p/q}\in \Card$  corresponding to a polynomial with a parabolic fixed point with multiplier $e^{2\pi i p/q}$.}
There is a unique prime type $M_{p/q}$ attached to $\HC_0$ at $c_{p/q}$; these types are \emph{{satellite}} \index{Mandelbrot set $\MM$!combinatorial type $M\subset \MM$!prime!satellite}, and all other prime types are \emph{{primitive}} \index{Mandelbrot set $\MM$!combinatorial type $M\subset \MM$!prime!primitive}.
The {closure} $\hat M_{p/q}$ \index{Mandelbrot set $\MM$!limb $\hat M_{p/q}$} of the component  of $\MM\sm \{c_{p/q}\}$ intersecting $M_{p/q}$ is called the \textit{$p/q$-limb} of $\MM$.
In \cite{BD86}, 
Branner and Douady used a quasiconformal surgery,  inserting an extra limb into the filled Julia sets,  to define a continuous dynamical embedding $\Phi_2: \hat{M}_{1/2}\to \hat{M}_{1/3}$, and their argument can be easily generalized to construct embeddings $\Phi_{q}: \hat{M}_{1/q}\to \hat{M}_{1/{(q+1)}}$ for all $q\geq 2$.
Choosing the embeddings so that each new limb is inserted ``far away" from the limb containing the critical point, we define a \textit{parabolic tail}  to be a family of types of the form 
\begin{figure}
	\begin{center}
		\def\svgwidth{5.5in}
		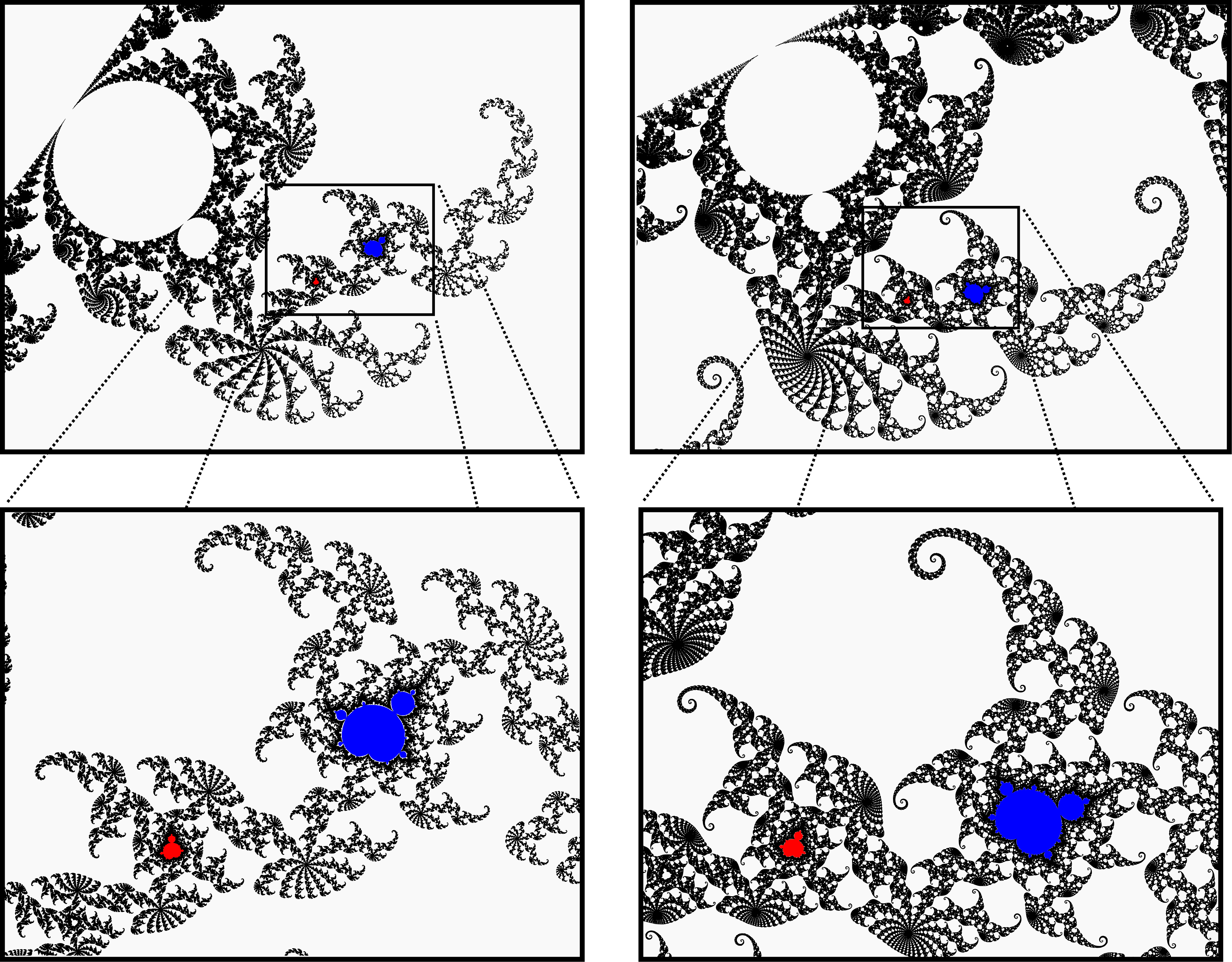
		\caption{The boundary of the $\MM$ zoomed in near the limbs $\hat{M}_{1/10}$ (left) and $\hat{M}_{1/20}$ (right). The small Mandelbrot sets 
         highlighted in red and blue belong to the same parabolic tails respectively.}
		\label{fig:parabolic tails}
	\end{center}
\end{figure}
$$\setofst{\Phi_{n}\circ \Phi_{n-1}\circ \dots \circ \Phi_{q}(M)}{n\geq q}$$
for some $q\geq 2$ and type $M\subset \hat{M}_{1/q}$ (see Figure \ref{fig:parabolic tails}).

We say that a family of types is \textit{parabolically bounded}  if 
it is a finite union of types and  parabolic tails. Our main result is:

\begin{thmA}\label{thm: A}
	Every parabolically bounded family of prime primitive types has beau bounds. 
\end{thmA}

While the beau bounds in Theorem A provide some control for an infinitely renormalizable quadratic polynomials with parabolically bounded prime primitive combinatorics, for example it follows from \cite{McMullen94} that  such polynomials support no invariant line fields on their Julia sets, they are not enough on their own to establish MLC. Let us say that a family of types $\CC$ is \textit{rigid} if MLC holds at any infinitely renormalizable quadratic polynomial in $\MM$ with combinatorics in $\CC$. The rigidity theorem in \cite{lyubich1997dynamics} asserts that any family of types $\CC$ with beau bounds and satisfying a certain \textit{secondary limbs condition} is rigid. 
However, any family of types containing small Mandelbrot sets in infinitely many distinct limbs does not satisfy the secondary limbs condition.
We observe that the argument in \cite{lyubich1997dynamics} actually provides a weaker condition for families of types, which we call \textit{bounded {puzzle} Teichm\"uller diameter}, that suffices for rigidity:
\begin{thmB}
	Every family of prime primitive types with beau bounds and bounded {puzzle} Teichm\"uller diameter is rigid.
\end{thmB}

Parabolic tails are shown to have bounded puzzle Teichm\"uller diameter in \cite{DKLP:elephants}, 
so Theorems A and B together yield:
\begin{thmC}
	Any parabolically bounded family of prime primitive types is rigid.
\end{thmC}

We remark that our results can also be combined with those in \cite{dudko2023mlcfeigenbaumpoints}, providing beau bounds and rigidity for families of types containing finitely many primitive or satellite types and finitely many primitive parabolic tails. 
However, such a combination requires transferring between quadratic-like maps and the $\psi^\bullet$-quadratic-like maps introduced in \cite{dudko2023mlcfeigenbaumpoints}, which is beyond the scope of this article.

\subsection{Outline of the proof}

For a quadratic-like map $f\from U'\to U$ with connected filled Julia set $\FK$, the \textit{width} of $f$ is defined to be $\WW(f) = \frac{1}{\mmod(U\sm \FK)}$.
If $f$ is renormalizable with renormalization $\RR f$,
{we  want to consider } the width $\WW(\RR f)$.
However, $\RR f$ is not uniquely defined as a quadratic-like map: 
we can for instance restrict $\RR f$ to make $\WW(\RR f)$ arbitrarily large. 
We want to choose the renormalization so that
the domain of $\RR f$ is ``near-maximal''
(and therefore $\WW(\RR f)$ is {near-minimal})
this is achieved by considering {\em pseudo-quadratic-like ($\psi$-ql) maps,} 
generalizations of quadratic-like maps introduced in \cite{K06} 
{that admit} canonical renormalizations. 

In \cite{K06}, beau bounds for bounded primitive combinatorics follow from the main technical theorem there, known as ``if it's Bad Now, it was Worse Earlier'' (BNWE):
if $f$ is an infinitely renormalizable \pql map,
with combinatorics in some bounded family of primitive types,
and the width is large at some renormalization level, 
then the width was twice as large at some recent earlier level.
  Indeed, \textit{a priori} bounds follow immediately by considering the {\em record} renormalization levels 
  with width larger than all previous levels.

As in \cite{KL:decorations} and \cite{KL:molecules},
we will produce beau bounds here by showing ``if it's Bad Now, it was Much Worse Earlier"(BNMWE).
Despite its sardonic nature
we will call a family with BNWE \emph{optimistic}%
\footnote{ BNWE can also be easily turned into ``if it's bad now, it will be better in the future''}, 
and one with BNMWE, \emph{super-optimistic}.
More precisely,
we say that a family of  types $\CC$ is super-optimistic
if for all $M>0$, 
there is some finite family $\CC'\subset \CC$ such that any renormalizable \pql map $f$ with combinatorics in $\CC\sm \CC'$ and $\WW(\RR f)\largegiven{\CC, M}$ satisfies $\WW(f)\geq M\WW(\RR f)$.
Using BNWE from \cite{K06}, 
we have the following fact which is proven  implicitly in \cite{KL:decorations}:
\begin{lemmaA}
    Any super-optimistic family of primitive types  has beau bounds.
\end{lemmaA}

Theorem A  then follows immediately from:
\begin{thmD}
	Any parabolically bounded family of prime primitive types is super-optimistic.
\end{thmD}

To prove Theorem D, our analysis starts off similarly to that in \cite{K06}. If $f: U'\to U$ is a renormalizable \pql map with period $p$, then its renormalization $\RR f$ defines a finite collection $\KK$ of small filled Julia sets in the dynamical plane of $f$. The set $U\sm \KK$ is a hyperbolic Riemann surface; its
 thick-thin decomposition induces a lamination on $U\sm \KK$ called the \textit{canonical lamination} $\FF$. 
For any proper arc $\gamma$ in $U\sm \KK$, its \textit{canonical width} is the extremal width of all the leaves of $\FF$ homotopic to $\gamma$; the function $X$ taking proper arcs to their canonical width is the \textit{canonical weighted arc diagram} of $U\sm \KK$. If $Y$ is a restriction of $X$, then the \textit{magnitude} of $Y$ is defined by $\|Y\|= \sum_\gamma Y(\gamma)$.

A proper arc $\gamma$ in $U\sm \KK$ is \textit{vertical}  if it connects $\dd\KK$ to $\dd U$,  and \textit{horizontal} if it connects $\dd \KK$ to
 itself; we denote 
 the restrictions of $X$ to vertical and horizontal curves by $X^\ver$ and $X^\hor$ respectively. 
 We can directly relate $X$ to the widths of $f$ and $\RR f$ by 
 \begin{equation}\label{eq:intro width f}
 	\WW(f)\geq \|X^\ver\|
 \end{equation}
and 
\begin{equation}\label{eq:intro width Rf}
	p\WW(\RR f) \asymp \|X^\ver\|+\|X^\hor\|+O(p).
\end{equation}
Thus{\color{Green},}  if \begin{equation}\label{eq:intro def vert}
	\|X^\ver\|\geq 2\eps\|X^\hor\|
\end{equation}
for some $0< \eps < 1/2$, then \eqref{eq:intro width f} and \eqref{eq:intro width Rf} imply that 
\begin{equation}\label{eq:intro optimism}
	\WW(f)\succeq p\eps \WW(\RR f)\geq 2\WW(\RR f) 
\end{equation}
when $\WW(\RR f)\largegiven{1}$ and $p\largegiven{\epsilon}.$
In \cite{K06}, \eqref{eq:intro def vert} is achieved for a universal $\eps$ when $\WW(\RR f)\largegiven{p}$. The central idea is that we can pull back $X$ by $f$ and apply the Gr\"otzsch inequality to see some loss of horizontal weight dictated by the \textit{core entropy} of $f$; that is the entropy of $f$ restricted to its Hubbard tree $\TT$. For primitive combinatorics the core entropy is always positive, so after applying several pullbacks the loss of horizontal weight can be made definite. The resulting definite vertical weight is then pushed forward by the Covering Lemma of \cite{KL:QA}.

With unbounded combinatorics, we immediately 
{encounter} a problem: if the geometry of $U\sm \KK$ does not respect the combinatorial structure of $\TT$, for example lots of weight in $X$ is supported on arcs that intersect $\TT$ many times, then we cannot relate the loss of horizontal weight to the core entropy of $f$. This possibility is excluded in \cite{K06}, where it is shown that after some restriction $X$ is ``aligned" with $\TT$, but the argument relies on the bounded topology of $U\sm \KK$. 
We split our analysis into two cases: the ``tame" case where the geometry of $U\sm \KK$ suitably respects $\TT$, and the ``wave" case where it does not. 
In both cases, the key to our analysis is that the parabolically bounded condition forces the unbounded part of the topology to occur where the dynamics is incredibly simple: away from a finite collection of components of $\KK$ the map $f$ looks like a translation.

In the tame case, our argument differs from that in  \cite{K06}. 
We suppose that \eqref{eq:intro def vert} does not hold for any definite $\eps$, so $\|X^\ver\|$ is almost negligible.
Pulling back $X^\hor$ by $g= f^p$, we therefore get a weighted arc diagram $g^*X^\hor$ that, up to some error, \textit{dominates} $X^\hor$.
Using the intersection arcs with the Hubbard tree $\TT$, we can assign \textit{values} $\val(X^\hor)$ and $\val(g^*X^\hor)$ to the weighted arc diagrams. The domination relation implies that 
{ $\val(g^*X^\hor)\geq \val(X^\hor)$,} and the positive entropy of $g$ acting on $\TT$ implies that $\val(g^*X^\hor)\leq \val(X^\hor)/2$; this is our contradiction. The technical difficulty is that we have error terms that depend on the topology of $U\sm \KK$. Our parabolically bounded condition allows us to replace $U\sm \KK$ with a bounded topology subsurface, completing the argument in this case.

In the wave case, the wild geometry of $U\sm \KK$ produces a particular type of lamination in the complement of $\TT$ called a \textit{wave}, which was first introduced in \cite{dudko2023mlcfeigenbaumpoints}. 
For bounded combinatorics, 
the \emph{Wave Lemma} in \cite{dudko2023mlcfeigenbaumpoints} asserts that when we push forward a wave by the dynamics of $f$, we produce either some definite vertical width or a wider wave; by considering the widest wave we must therefore have definite vertical width. 
In our setting we show that the same holds; the main difference is that we want to push forward by an unbounded iterate of our map. 
The parabolically bounded condition allows us to observe first, that most of a wave is localized to a single limb in $\TT$, and second, that for waves localized to a single limb we can use the near-translational part of the dynamics to treat large iterates similarly to a single iterate; this completes the argument in this case.

This article is organized as follows. In Section \ref{sec:prelim}, we recall the theory of quadratic-like and pseudo-quadratic-like maps; we also introduce the theory of \emph{value} for weighted arc diagrams, which is a tool that can rule out dynamically invariant weighted arc diagrams for quadratic-like maps.  In Section \ref{sec:life improvement}, we outline of the proof of Theorem D. 
In Section \ref{sec:trichotomy}, we show that the analysis can be divided into the ``tame" and ``wave" cases, which are studied in Sections \ref{sec:aligned} and \ref{sec:waves} respectively.
In Section \ref{sec:rigidity}, we prove Theorem B.
To streamline the organization of our argument, there are several self-contained discussions which we delay to the appendices.  
In Appendix \ref{app:markings} we recall the theory of ideal markings, canonical laminations and weighted arc diagrams, and a new tool: a {\em Uniform Thick-Thin Decomposition} for bordered Riemann surfaces
{(initially developed by Kahn and Lyubich in \cite{KL:eyes}).} 
In Appendix \ref{app:transl-mark} we introduce \textit{translation markings}, a marked disk with near-translational dynamics, {and bound the width of combinatorially long arcs in these markings}. 
{In Appendix \ref{app:clones}, we introduce the concept of \emph{cloning} a \emph{cross section} of a directed graph, and bound the leading eigenvalue and corresponding eigenvector of the result.}
In Appendix \ref{app:wanderers}, we state and prove the \textit{Wanderers Theorem}, a variant of the Covering Lemma.
The results from Appendices \ref{app:transl-mark}, \ref{app:clones}, and \ref{app:wanderers} are all used in Section \ref{sec:aligned}. {After the appendices, we include an index of the terms and notation used throughout this article for the readers convenience.}

\subsection{Future perspective}\label{sec:future perspective}
To fully resolve MLC, rigidity must be shown for all families of combinatorial types;
as a general {principle}    
this should follow from proving (properly defined) \emph{a priori} bounds.
Here we produce a brief ontology of renormalization types and then discuss methods of intermediate renormalization that we hope and expect to be sufficient for the general case.

For a prime satellite type $M_{p/q}$,
we let  $[0; a_{1}, \dots, a_{k}]$ be the continued fraction expansion of $p/q$; 
we set $$h(M) = h(p/q) = \sup_{n}a_{n}\text{ and }d(M) = d(p/q) = k.$$

Recalling that $\MM$ \index{Mandelbrot set $\MM$} is the set of all $c$ such that the orbit of $0$ under $z\mapsto z^2+c$ is bounded, the \textit{root} and \textit{principal tip} of $\MM$ are the parameters $c = 1/4$ and $c= -2$ respectively. For any small Mandelbrot set $M$, its root and principal tip are the images of the root and principal tip of $\MM$ under the tuning homeomorphism $\MM\to M$. 
For $r$ the root of $M$,  the \textit{limb} $\hat M$ of $M$ is defined to be the {closure of the} component of $\MM\sm \{r\}$ intersecting $M$. For $t$ the principal tip of $M$, the \textit{principal decorations} of $M$ are the {closures of the} components of $\MM\sm \{t\}$ avoiding $M$; we denote the union of the principal decorations by $\hat M^{1/2}$. 

For any prime and primitive type $M$, there is a unique infinite sequence $(M_n)_{n=0}^\infty$ such that $M_0 = \MM$, $M_n\subset M_{n-1}$ is a small Mandelbrot set mapped to some satellite $M_{p_n/q_n}$ under the straightening homeomorphism $M_{n-1}\to \MM$ for all $n\geq 1$, and $M\subset \hat M_n$ for all $n$.
When $n$ is sufficiently large, we  have $M\subset \hat M^{1/2}_N$ for all $N\geq n$;
we let $\l(M)$ be the minimal such value of $n$,
and then $p_n/q_n = 1/2$ for all $n > \l(M)$.
We let $h(M)$ and $d(M)$ be the maximum of $h(p_n/q_n)$ and $d(p_n/q_n)$ over all $n \le \l(M)$. 
To complete our definitions, we let $\l(M) = 0$ when $M$ is prime and satellite.

For any family of prime types $\CC$, we can define
$\l(\CC)$, $h(\CC)$, and $d(\CC)$ to be the supremum over all $M\in \CC$ of $\l(M)$,  $h(M)$, and $d(M)$ respectively.
We classify $\CC$ into four different cases:
\begin{enumerate}
	 \item \textit{Anti-molecule (and bounded satellite):} $l(\CC), h(\CC), d(\CC) < \infty$.
	\item \textit{Near-parabolic:} $h(\CC)= \infty$ and $l(\CC), d(\CC) < \infty$.
	\item \textit{Near-neutral:} $d(\CC) = \infty$ and $l(\CC) < \infty$.
	\item \textit{Virtually-satellite:} $l(\CC) = \infty$.
\end{enumerate}


Corresponding to these four cases are four renormalization theories 
(which we may think of as four fingers to be combined with the thumb of quadratic-like renormalization):
\begin{itemize}
\item
puzzle (or generalized) renormalization, 
\item
near-parabolic {renormalization}, 
\item
sectorial (neutral and satellite {unified as ``molecule"}) renormalization, 
\item
and finally:  virtual (near-molecule primitive) renormalization. 
\end{itemize}
The bounds in \cite{K06} were extended to the bounded satellite case in \cite{dudko2023mlcfeigenbaumpoints}
(preceded by \cite{DL_pacmen} on some special satellite combinatorics of  high type\footref{incompatibility})
yielding the  general quadratic-like case for bounded combinatorics.

The puzzle renormalization 
was introduced in \cite{lyubich1997dynamics} to prove bounds%
\footnote{\label{incompatibility} {This work gave an important insight into the problem.}
  However, it was performed in the {\em non-degenerate regime} 
that cannot be directly combined with results in the {\em near-degenerate regime.} 
{On the other hand, all the pieces handled in the near degenerate regime
can naturally be unified (so far).}}
for special types of high primitive  combinatorics,
using estimates of moduli 
that build on the work of Yoccoz showing MLC in the non-infinitely renormalizable case. 
It was then used {in} the near-degenerate regime in \cite{KL:decorations,KL:molecules} to show super-optimism for renormalization types bounded away from the main molecule of $\MM$,
thereby completing case (1) above. 
It will also appear as an endgame in the other three cases.


Near-parabolic renormalization began with the work of Lavaurs and reached maturity with the 
near-parabolic {\em a priori} bounds (once again assuming high type)
in the work of Inou and Shishikura \cite{IS}\smash{\footref{incompatibility}}
{that found numerous application (see e.g., \cite{BC12,AL22,CS_satellite}).}
The general neutral case,
combining near-parabolic and sectorial renormalization,
was recently handled in \cite{DLneutral} 
and the authors of that paper believe it can be promoted to all forms of unbounded satellite renormalization.

What remains are the primitive classes with $h(\CC)$, $d(\CC)$, or $l(\CC)$ infinite.
The current paper dealing with the primitive near-parabolic combinatorics of  
parabolically bounded type is the first case of this kind, 
resolving a subcase of the case where $h(\CC) = \infty$ and $d(\CC) = l(\CC)=1$.
We believe the techniques developed here will be the first step to the general near-parabolic case.

The reader may also wish to consult the excellent survey \cite{dudkoMLC} for a more detailed discussion of renormalization theory with applications to MLC. 
The sectorial renormalization of \cite{DLneutral} is discussed in greater detail there, 
and the virtual renormalization to appear in \cite{DKL:virtual} is briefly sketched as well.
{It is feasible that this strategy can bring the story to a completion.}


\subsection{Terminology and notation}
Let us collect here some of main terminology and notation used throughout this article. 
We denote:
\begin{itemize}
    \item $x\oplus y= (x^{-1}+y^{-1})^{-1}$ the harmonic sum; 
    \item $\WW(\FF)$ the extremal width of a path family $\FF$;
    \item $A\largergiven{x_1, \dots, x_n}B$ for $A\geq C B$ for a large constant $C$ depending only on the parameters $x_1, \dots, x_n$;
    \item the ``big-O" notation $A\largergiven{}O(A)$;
    \item $\ddi V$ the ideal boundary of a Riemann surface $V$;
    \item $\dcara X$ the Caratheodory boundary of a set $X\subset \C$.
\end{itemize}

By default, paths are considered up to reparameterization, and in a Riemann surface $V$ a path $\gamma: (0, 1)\to V$ is proper if it extends continuously to a path $\gamma: [0, 1]\to V\cup \ddi V$ with $\gamma(0), \gamma(1)\in \ddi V$; we will almost exclusively consider proper paths.

\subsection{{Long story of this paper}} 
This project was initially undertaken by Kahn and Lyubich, around 2007--2011, with some additional work up to 2015. 
These two authors were able to completely treat the case of one small Julia set per limb \cite{KL:eyes}, 
and, in doing so, developed most of what is now in Appendices A and B, along with some of its application in Section \ref{sec:aligned}. More than one sketch was produced for how to handle the general case. The project was resumed {and completed}  by Kahn and Kapiamba in 2024, based originally on one such sketch,
but then pivoting to use the theory of waves that had recently been developed by Dudko and Lyubich \cite{dudko2023mlcfeigenbaumpoints}. 

\subsection{Acknowledgment}
We thank the NSF for their continuous support of the MLC project. {The second author material was supported by the National Science Foundation under Grant No DMS 2303168}
{We also gratefully acknowledge the support of MSRI (now SLMath) and the organizers of the MSRI Spring 2022 Program in Complex Dynamics. In particular, the first and second authors began a collaboration that evolved into the current project (along with a complementary one on rigidity for parabolically bounded maps with \emph{a priori} bounds).}

\section{Pseudo-quadratic-like maps} \label{sec:prelim}
We recall some standard preliminaries for the ``modern theory of \emph{a priori} bounds'', such as pseudo-quadratic-like maps, renormalization, and Hubbard trees. For more details, we refer the reader to \cite{K06} or \cite{dudko2023mlcfeigenbaumpoints}.

\subsection{Pseudo-quadratic-like maps}\label{sec:pql maps}

A \textit{quadratic-like map} $f: U'\to U$ is a holomorphic double branched covering between two Jordan disks $U'\Subset U\subset \C$. 
The annulus $U\sm \overline{U'}$ is called the \textit{fundamental annulus} of $f$,
and the \textit{modulus} of $f$ is defined by $$\mmod f :=\mmod(U\sm \overline{U'}).$$
The \textit{filled Julia set} of $f$ is the set $\FK$ of nonescaping points; that is 
$$\FK:= \setofst{z}{f^n(z)\in U' \text{ for all }n\geq 0}.$$
The filled Julia set $\FK$ is connected if and only if the unique critical point of $f$ belongs to $\FK$.

More generally, a \textit{pseudo-quadratic-like ($\psi$-ql) map} \index{pseudo-quadratic-like ($\psi$-ql) map} is a pair of holomorphic maps $$F= (f, \iota): (\UU', \FK')\toto (\UU, \FK)$$ between two conformal disks $\UU', \UU$ such that:
\begin{enumerate}
	\item $f: \UU'\to \UU$ is a double branched covering;
	\item $\iota: \UU'\to \UU$ is an immersion;
	\item $\FK$ and $\FK'$ are compact connected full sets satisfying 
	$$\iota^{-1}(\FK) = f^{-1}(\FK)= \FK'.$$
\end{enumerate}
We will call $\FK$ \index{pseudo-quadratic-like ($\psi$-ql) map!filled Julia set $\FK$} the filled Julia set of $F$.
The \textit{width} of $F$ \index{pseudo-quadratic-like ($\psi$-ql) map!width $\WW(F)$} is defined by
$$\WW(F) := \WW(\UU\sm \FK) = \frac{1}{\mmod (\UU\sm \FK)}.$$

While \pql maps are more abstract than quadratic-like maps, they are closely related. Indeed, note that any quadratic-like map $f:U'\to U$ with connected filled Julia set $\FK$ is a \pql map $(f, \iota): (U', \FK)\toto(U, \FK)$ with $\iota$ the identity. The following fact from \cite[Theorem 7.3]{K06} ensures that any \pql map restricts to a quadratic-like one near its filled Julia set:

\begin{theorem}[Kahn]\label{thm:induced quadratic-like map}
	For any $M>0$ there exists $\mu>0$ such that: if  $$F= (f, \iota): (\UU', \FK')\toto (\UU, \FK)$$ is a \pql map with $\WW(F)\leq M$, then there exist neighborhoods $U'\subset U$ of $\FK$ such that the restriction $f\circ \iota^{-1}: U'\to U$ is a quadratic-like map with $\mmod f \geq  \mu$.
\end{theorem}

Theorem \ref{thm:induced quadratic-like map} allows us to iterate a \pql map $F= (f, \iota): (\UU', \FK')\toto (\UU, \FK)$ in a neighborhood of the filled Julia set $\FK$; by abuse of notation we will write $F = f\circ \iota^{-1}$ there.
To iterate $F$ globally, setting $\UU^0 = \UU$ and $\UU^1 = \UU'$, for all $m>1$ we can inductively define the \textit{restriction} 
$$F= (f, \iota): \UU^m\toto\UU^{m-1} \text{ with } \UU^m = \{(x, y)\in \UU^{m-1}\times \UU^{m-1}: f(x) = \iota(y)\},$$
where $f$ and $\iota$ are the component-wise projections. These restrictions induce the \textit{iterations} \index{pseudo-quadratic-like ($\psi$-ql) map!iterates} $$F^m = (f^m, \iota_m): \UU^m\to \UU^0$$
given by $f^m =  f^{m-1}\circ f$ and $\iota_m = \iota_{m-1} \circ \iota$. 

We will say that the \pql-map $F$ is \emph{periodically repelling} \index{pseudo-quadratic-like ($\psi$-ql) map!periodically repelling} if every periodic cycle in $\FK$ is repelling.

\subsection{Localization}\label{sec:localization}

Let us recall the idea of ``localization" from \cite[Section 7.4.1]{K06}. 
Let $\UU$ be a hyperbolic Riemann surface, let $\gamma$ be a homotopically non-trivial Jordan curve in $\UU$, and let $U$ be a component of $\UU\sm \gamma$. For $\pi: \tilde \UU\to \UU$ the universal cover with deck transformation group $\Delta$, let 
$\tilde \gamma$ be a lift of $\gamma$ by $\pi$ and let 
 $\Delta_\gamma$ be the stabilizer of $\tilde \gamma$.
 The \emph{subsurface cover} \index{subsurface cover}of $\UU$ corresponding to $U$ is the Riemann surface $\tilde U= \tilde \UU/ \Delta_\gamma$. 
 The curve $\tilde \gamma$ descends to a copy of $\gamma$ inside the subsurface cover.
 Cutting $\tilde U$ along this copy of $\gamma$ yields one component that is not mapped to $U$ by the covering $\tilde U\to U$; gluing this component to $U$ along $\gamma$ yields the \emph{localization} \index{localization} of $U$ inside $\UU$.

\subsection{Renormalization}\label{sec:renorm}

A quadratic-like map $f: U'\to U$ is called \textit{quadratic-like renormalizable}  with period $\per> 1$ if some restriction $f^\per: V'\to V$ is a quadratic-like map with connected filled Julia set $K$; we additionally require that $K$ contains the critical point of $f$.
The restriction is called a \textit{quadratic-like renormalization} of $f$ of period $\per$. 
The renormalization is called \textit{prime} if there 
is no renormalization of period $1< m < \per$, and \textit{primitive} if the sets $f^j(K)$ with $1\leq j < \per$ are pair-wise disjoint.
Note that for a given period $\per$, there is no canonically defined quadratic-like renormalization of $f$ of period $\per$, indeed there are infinitely many different possible choices of $U'$ and $U$. However, when the renormalization is prime the small filled Julia set $K$ is uniquely determined. One benefit of  working with \pql maps instead of quadratic-like maps is that
we can use the uniqueness of the small filled Julia set to construct a canonical renormalization.

Let $F= (f, \iota): (\UU', \FK')\toto (\UU, \FK)$ be a \pql map. We will say that $F$ is \textit{($\psi$-ql) renormalizable} with period $\per$ if it has a quadratic-like restriction (as in Theorem \ref{thm:induced quadratic-like map}) that is quadratic-like renormalizable with period $\per$, and the corresponding renormalization is prime. Let us additionally assume that the renormalization is primitive.
For $K$ the  corresponding small filled Julia set and $\KK = \bigcup_{j=0}^{\per-1}F^j(K)$ the cycle of small filled Julia sets, 
let $\VV$ be the localization of a small neighborhood of $K$ inside $\UU\sm (\KK\sm K)$. 
For the iterate $F^\per = (f^\per, \iota_\per): \UU^\per\toto \UU^0$ and $\KK^\per = (f^\per)^{-1}(\KK)$, let $\VV'$ be the localization of $K$ inside  $\UU^\per\sm (\KK^\per\sm K)$.
The pair of maps $(f^\per, \iota_\per): \UU^\per\to \UU$ lift to a \pql map $$\RR F= (g, \jmath): (\VV', K') \toto (\VV, K),$$ which we call the \textit{($\psi$-ql) renormalization} \index{pseudo-quadratic-like ($\psi$-ql) map!renormalization $\RR F$} of $F$.
With this construction, the renormalization $\RR F$ is uniquely defined by $F$. We will say that $F$ has \textit{$M$-amplification} \index{pseudo-quadratic-like ($\psi$-ql) map!amplification} for some $M\geq 1$ if 
$$\WW(F)\geq M \WW(\RR F).$$

\subsection{Limbs and rotation number}\label{sec:limbs}

Let  $F: (\UU', \FK')\toto (\UU, \FK)$ be a primitively renormalizable periodically repelling \pql map with cycle of small filled Julia sets $\KK$. \index{pseudo-quadratic-like ($\psi$-ql) map!renormalization $\RR F$!cycle of small filled Julia sets $\KK$}

There is a unique fixed point of $F$, called $\beta_F$, \index{pseudo-quadratic-like ($\psi$-ql) map!filled Julia set $\FK$!$\beta_F$} such that $\FK\sm \{\beta_F\}$ is connected.
Removing the other fixed point, which is called $\alpha_F$\index{pseudo-quadratic-like ($\psi$-ql) map!filled Julia set $\FK$!$\alpha_F$}, splits $\FK$ into $q\geq 2$ connected components. 
These components are the \textit{limbs} \index{pseudo-quadratic-like ($\psi$-ql) map!filled Julia set $\FK$!limbs $\FL_j$ and $\FL_I$} of $\FK$, and we will label them by $(\FL_j)_{j\in \Z/q\Z}$ according to their counter-clockwise circular ordering around $\alpha_F$, chosen so that $\FL_0$ contains the critical point. There is some $0 \leq p < q$ coprime to $q$ so that $F$ maps $\FL_j$ homeomorphically to $\FL_{j+p}$ for all $j \neq 0$; the fraction $p/q$ is called the \textit{combinatorial rotation number of $F$ at $\alpha_F$}\index{pseudo-quadratic-like ($\psi$-ql) map!combinatorial rotation number}. As $F$ maps $\FK$ to itself as a double branched cover, it follows that the unique critical value of $F$ is in $\FL_p$ and  $F(\FL_0)= \FK$. Note that the limbs do not contain $\alpha_F$, and $\overline{\FL_j} = \FL_j\cup \{\alpha_F\}$ for all $j$.

For any interval $I\subset \R$ we can associate a subset of $\Z/q\Z$ by first intersecting with $\Z$ and then projecting; by slight abuse of notation we will also call this subset $I$. For any such interval, we denote $\FL_I := \bigcup_{j\in I}\FL_j.$
Supposing now that $F$ is renormalizable with small filled Julia sets $\KK$, for any interval $I$ we set $\KK_I = \KK\cap \FL_I$. We let $U_I\subset \UU$ be an open Jordan neighborhood of $\FL_I$, chosen only up to isotopy rel $\KK$. \index{pseudo-quadratic-like ($\psi$-ql) map!subsurface $U_I$}

\subsection{Hubbard trees}\label{sec:hubbard trees}

Let  $F: (\UU', \FK')\toto (\UU, \FK)$ be a primitively renormalizable and periodically repelling \pql map with cycle of small filled Julia sets $\KK$.

We suppose for now that $\FK$ is locally connected. 
In this case, we define the \textit{Hubbard tree} \index{Hubbard tree $\TT$} of $F$ to be the smallest connected set $\TT\subset \FK$ containing $\KK$; it is  compact and forward invariant. 
We set $\LL_j = \FL_j\cap \TT$ for all $j$ and $\LL_I = \FL_I\cap \TT$ for any interval $I\subset \R$. \index{Hubbard tree $\TT$!limbs $\LL_j$ and $\LL_I$}
A point $x\in \TT\sm \KK$ is called a \textit{branch point} of $\TT$ if $\TT\sm \{x\}$ has at least three components. The branch points of $\TT$ are all repelling periodic points of $F$; if $F$ has combinatorial rotation number $p/q$ with $q>2$ then $\alpha_F$ is also a branch point of $\TT$.
We denote by $\hat\KK$ the union of $\KK$ with all the branch points of $\TT$ and $\alpha_F$. We will call the  components of $\hat\KK$ the \textit{vertices} of $\TT$ and denote the set of vertices by $V(\TT)$\index{Hubbard tree $\TT$!vertices $V(\TT)$}. The complement of the vertices in $\TT$ is a union of disjoint paths in $\UU\sm \hat\KK$; we will call each such path an \textit{edge} of $\TT$ and denote the set of edges by $E(\TT).$ \index{Hubbard tree $\TT$!edges $E(\TT)$} These sets of vertices and edges form an abstract tree, which we will call the \textit{combinatorial Hubbard tree}.
The \textit{combinatorial distance} between two vertices of $\TT$ is length of the shortest path in the combinatorial Hubbard tree connecting them.

We can similarly view $f^{-1}(\TT)$ as an abstract tree whose vertices are the components of $f^{-1}(\hat\KK)$ and whose edges are the components of $f^{-1}(E(\TT))$.
Thus we have a 2-to-1 map $f:f^{-1}(E(\TT))\to E(\TT)$.
 {Moreover, the tree $f^{-1} (\TT)$ contains a subtree $i^{-1} (\TT) $
whose edges map to the edges of $\TT$ under $i$.   
It induces a natural surjection $i:  E(i^{-1} (\TT)) \ra \TT$}.

The composition $f\circ \iota^{-1}$ induces a linear transformation $T_F: \R^{E(\TT)}\to \R^{E(\TT)}$. As the renormalization of $F$ is prime,  $T_F$ is primitive and has a unique Perron-Frobenius eigenvalue $\lambda_F\geq 1$; as the renormalization is primitive we have $\lambda_F>1$. We will call $\lambda_F$ the \emph{core eigenvalue} \index{Hubbard tree $\TT$!core eigenvalue $\lambda_F$ of $T_F$}of $F$ (the quantity $\log \lambda_F$ is called the \emph{core entropy} of $F$). 
The inverse composition $\iota\circ f^{-1}$ induces the transpose linear transformation $T_F^\top$, which has the same Perron-Frobenius eigenvalue $\lambda_F$.
For any $n> 1$ the compositions $f^n\circ \iota_n^{-1}$ and $\iota_n\circ f^{-n}$ similarly induces a linear transformation $\R^{E(\TT)}\to \R^{E(\TT)}$, these are exactly $T_F^n$ and $(T_F^\top)^n$.

Let us now suppose instead that $\FK$ is not locally connected. In this case, $\FK$ admits a locally connected \textit{combinatorial model} $\FK_\com$ obtained as the quotient of the unit disk by a rational lamination. In this case, there is a natural projection $\pi: \FK\to \FK_\com$ and $F$ descends to a map  $F_\com: \FK_\com\to \FK_\com$.
The combinatorial model has an analogously defined Hubbard tree $\TT_\com$ with vertices $V(\TT_\com)$, edges $E(\TT_\com)$, linear transformation $T_F$ and eigenvalue $\lambda_F$; we can lift by $\pi$ to get the corresponding objects for $F$.

\subsection{Hubbard arc-diagrams}

Let  $F: (\UU', \FK')\toto (\UU, \FK)$ be a primitively renormalizable and periodically repelling \pql map with cycle of small filled Julia sets $\KK$ and Hubbard tree $\TT$. 
We equip $\UU\sm \KK$ with the minimal ideal marking as in Appendix \ref{app:ideal markings and waves} and recall the definition weighted arc-diagrams from Appendix \ref{app:WAD}.

We will say that a proper path or arc  in $\UU\sm \KK$ is \textit{horizontal}, \textit{vertical}, or \textit{bivertical}
if both, one, or neither of its  endpoints lie in $\dcara \KK$ respectively.
We can define three arc-diagrams in $\UU\sm \KK$ associated to the Hubbard tree as follows (see Figure \ref{arc-diagrams}). We define the \textit{vertical Hubbard arc-diagram} \index{Hubbard tree $\TT$!vertical Hubbard arc-diagram $\HH^\ver$} $\HH^\ver$ to be the unique maximal AD in $\UU\sm \KK$ consisting of vertical arcs that can be realized in $\UU\sm \TT$. 
We define the \textit{horizontal Hubbard arc-diagram} \index{Hubbard tree $\TT$!horizontal Hubbard arc-diagram $\HH^\hor$} $\HH^\hor$ to be the unique maximal AD in $\UU\sm \KK$ consisting of horizontal arcs that can be realized in $\UU\sm \TT$ without intersecting $\HH^\ver$.
For any edge $\gamma\in E(\TT)$, there is a unique bivertical arc in $\UU\sm \KK$ that intersects $\TT$ only in $\gamma$;  we define the \textit{bivertical Hubbard arc-diagram} \index{Hubbard tree $\TT$!bivertical Hubbard arc-diagram $\HH^\biv$}
{$\HH^\biv\equiv \HH^\biv (\TT)$}
to be the set of all such arcs.  
These three arc-diagrams are closely related, for example we can relate the intersection numbers of any arc with $\hv$ and $\hperp$:

\begin{figure}
	\begin{center}
		\def\svgwidth{4in}
		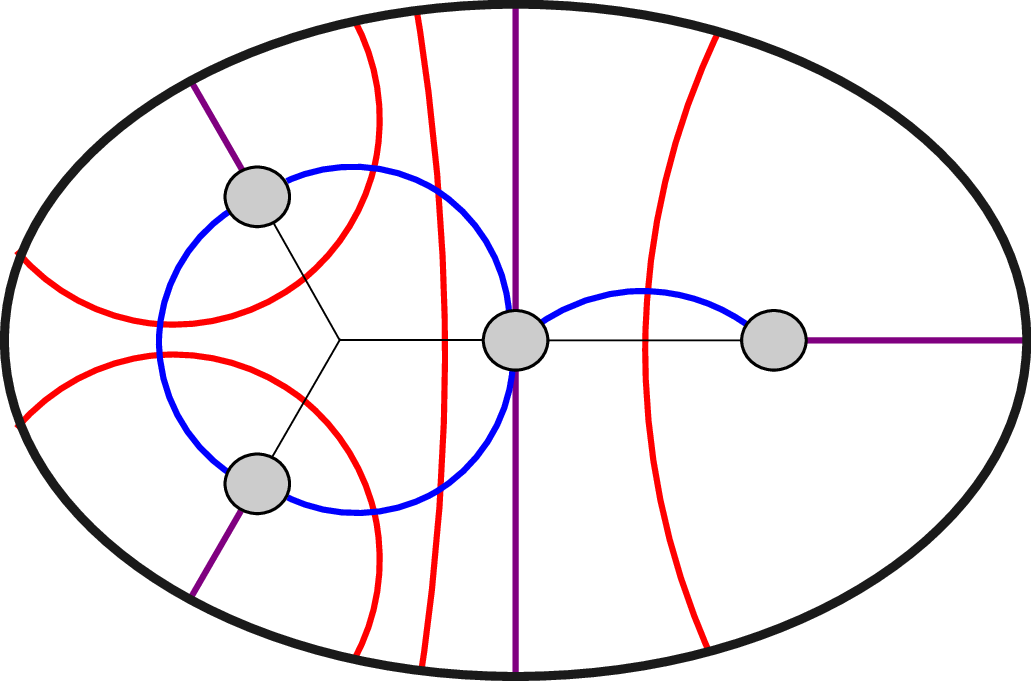
		\caption{The disk $\UU$ with small filled Julia sets represented as disks. The horizontal, vertical, and bivertical Hubbard arc-diagrams are drawn in blue, purple, and red respectively.}
		\label{arc-diagrams}
	\end{center}
\end{figure}
\begin{lemma} \label{lem:compare}
If each limb of $\TT$ contains at most $\bound$ many components of $\KK$, then 
we have
\begin{equation} \label{eq:compare}
\pair \alpha \hperp \le 2\bound (\pair \alpha \hv + 1)
\end{equation}
for any arc $\alpha$ in $\UU\sm \KK$.
\end{lemma}
\begin{proof}
We can use $\hv$ to cut $\alpha$ into $\pair \alpha \hv + 1$ segments $\alpha_i$. 
If an endpoint of some $\alpha_i$ lies on a small Julia set $K$, then we will call $K$ a  \emph{combinatorial endpoint} of $\alpha_i$. If an endpoint of some $\alpha_i$ lies on an arc of $\hv$ that connects a small Julia set $K$ to $\dcara \UU$, then we will also call $K$ a combinatorial endpoint. 
Each $\alpha_i$ has exactly two combinatorial endpoints, and by homotoping $\alpha_i$ close to $\TT$ we observe that  $\pair{\alpha_i}{\hperp}$ is equal to the combinatorial distance between its combinatorial endpoints.

As each limb of $\TT$ contains at most $\bound$ components of $\KK$, the combinatorial distance between any vertex of $\TT$ and $\alpha_F$ is at most $\bound$. Hence the combinatorial distance between any two vertices of $\TT$ is at most $2\bound$; \eqref{eq:compare} immediately follows. 
\end{proof}

{Notice in conclusion that the bivertical Hubbard arc-diagram $\HH^\biv \equiv \HH^\biv (\TT)$ lifts
  under $f$ to  the bivertical Hubbard arc diagram  $f^* (\HH^\biv)$ rel the tree $f^{-1} (\TT)$
and it lifts under $i$ to the bivertical Hubbard arc diagram  $i^* (\HH^\biv)$ rel the tree $i^{-1} (\TT)$.
Moreover, there are well defined push-forwards  
$$ f_*: \HH^\biv(f^{-1} (\TT)) \ra \HH^\biv (\TT) ,  \quad  i_*: \HH^\biv(i^{-1} (\TT)) \ra \HH^\biv (\TT) , $$
and we will use the following natural conventions:
$$f^* \beta = \{\alpha \in \HH^\biv(f^{-1} (\TT)) \ | \ f_* \alpha = \beta\}, \quad
 i^* \beta = \{\alpha \in \HH^\biv(i^{-1} (\TT)) \ |\ i_* \alpha = \beta \} . $$
}

\subsection{Values for weighted arc-diagrams} \label{sec:value}

\input{value1}

\section{Super-optimism}\label{sec:life improvement}

Let us say that a family of combinatorial types $\CC$ is \textit{$\eps$-optimistic} \index{family $\CC$ of types!optimistic (BNWE)}if any $n$-times renormalizable \pql map $F$
with combinatorics in $\CC$ and $\WW(\RR^n F)\largegiven{\CC, n}$ satisfies $\WW(F)\geq 2^n\eps\WW(\RR^nF).$
The optimism principle from  \cite{K06} can be precisely stated as:
\begin{theorem}[Kahn]\label{thm:bounded optimism}
	There is some universal $\eps>0$ such that any bounded family of primitive types is $\eps$-optimistic. 
\end{theorem}
Let us recall from the introduction that a family of types $\CC$ is \textit{super-optimistic} \index{family $\CC$ of types!super-optimistic (BNMWE)} if for all $M>0$, there is a finite set $\CC'\subset \CC$ such that any renormalizable \pql map $F$ with combinatorics in $\CC\sm \CC'$ and $\WW(\RR F)\largegiven{\CC, M}$ satisfies $\WW(F)\geq M\WW(\RR F)$. Said more plainly: we can achieve arbitrary amplification by removing a finite set of combinatorial types. 
Using Theorem \ref{thm:bounded optimism}, we can prove the following lemma from the introduction:
\begin{lem}\label{lem:asymptotic optimism}
	Every super-optimistic family of prime primitive types has beau bounds.
\end{lem}

\begin{proof}
	Let $0<\eps<1$ be the universal quantity in Theorem \ref{thm:bounded optimism}, and set $M = 2/\eps$. Let $\CC$ be a super-optimistic family of primitive types, so by definition there is some finite set $\CC'\subset \CC$ such that any renormalizable \pql map $F$ with combinatorics in $\CC\sm \CC'$ and $\WW(\RR F)\largegiven{\CC}$ satisfies $\WW(F)\geq M\WW(\RR F)$. 
	Let $N\geq 1$ be an integer large enough so that $2^N\eps \geq 2$.
	
	Let $F$ be an infinitely renormalizable \pql map whose renormalizations $(F_0 = F, \dots, F_n=\RR^nF, \dots)$ have combinatorics in $\CC$.
	For any $n\geq 0$, either  $(F_n, \dots, F_{n+N})$ has combinatorics in $\CC'$, or there is some $n\leq  m < n+N$
	so that $(F_n, \dots, F_{m})$ has combinatorics in 
	$\CC'$ and $(F_{m}, F_{m+1})$ has combinatorics in $\CC\sm \CC'$. 
	In the former case it follows from Theorem \ref{thm:bounded optimism} that $$\WW(F_n)\geq 2^N\eps\WW(F_{n+N})\geq 2\WW(F_{n+N})$$ 
	when $\WW(F_{n+N})\largegiven{\CC}$, and in the latter case it follows from Theorem \ref{thm:bounded optimism} and our choice of $M$ that $$\WW(F_n)\geq 2^{m-n}\eps\WW(F_m)\geq \eps M\WW(F_{m+1})= 2\WW(F_m)$$
	when $\WW(F_m)\largegiven{\CC}.$

	It follows from the above that there is some $C>0$ depending only on $\CC$ such that if $\WW(F_n)> C$ for some $n\geq N$, then $\WW(F_m)\geq 2\WW(F_n)$ for some $n-N\leq m < n$. Hence if $\WW(F_n)> C$ for some $n\geq kN$, then $$\sup_{0\leq j < N}\WW(F_j)>2^kC.$$
	As each $\WW(F_j)$ is finite, it follows that $\WW(F_n)\leq C$ for all $n\largegiven{\CC}$, so $\CC$ has beau bounds.
\end{proof}

\begin{rem}
    While it is not explicitly stated there, a version of Lemma \ref{lem:asymptotic optimism} first appeared in the proof of \cite[Corollary 5.17]{KL:decorations} in the context of combinatorial types satisfying the \emph{decoration condition} there; we include the slightly more general formulation here for completeness.
\end{rem}

\subsection{Parabolically bounded combinatorics}

For some $q\geq 2$, let $f'$ and $f$ be quadratic polynomials corresponding to  parameters in $\hat{M}_{1/q}$ and $\hat{M}_{1/(q+1)}$  respectively. Let $\FK'$ and $\FK$ be the filled Julia sets of $f$ and $\tilde f$ with limbs $\FL_j'$ and $\FL_j$ respectively. 
For  $1\leq k< q$, a \textit{$k$-limb insertion}  from $f'$ to $f$ is a continuous embedding 
$\phi: \FK'\to {\FK}$ that is holomorphic on $\Int \FK'$ and satisfies:
\begin{enumerate}
	\item For all $k-q<  j<   k$, $\phi(\FL_j') \subset \FL_j$ and $\phi \circ f' = f\circ \phi$ on $\FL_j'$.
	\item $\phi(\FL_{k}') \subset \phi(\FL_{k})$ and $\phi \circ f' = f^2\circ \phi$ on $\FL_k'$.
\end{enumerate}
A \textit{$k$-limb insertion} 
$\Phi: \hat{M}_{1/q}\to \hat{M}_{1/(q+1)}$
\index{Mandelbrot set $\MM$!limb-insertion}
is a continuous embedding 
such that for any quadratic polynomial $f'$ corresponding to a parameter $c\in\hat{M}_{1/q}$ and $f$ the quadratic polynomial corresponding to $\Phi(c)$,  there is a $k$-limb insertion from $f'$ to $f$.
A $1$-limb insertion from $\hat{M}_{1/2}$ to $\hat{M}_{1/3}$ was first  constructed in \cite{BD86} via quasiconformal surgery, and their argument can be easily generalized to produce $k$-limb insertions from $\hat{M}_{1/q}\to \hat{M}_{1/(q+1)}$ for all $q$ and $k$.

For all $q\geq 2$, let $\Phi_{q}: \hat{M}_{1/q}\to \hat{M}_{1/(q+1)}$ be a $k_q$-limb insertion. Let us say that the sequence of limb insertions $(\Phi_q)$ is \textit{balanced} if $k_q\to \infty$ and $q-k_q\to \infty$ when $q\to \infty$; for example we can take $k_q = \floor{q/2}.$ For any balanced sequence $(\Phi_q)$ of limb insertions, we define a \textit{parabolic tail} \index{family $\CC$ of types!parabolic tail} to be a combinatorial type $\CC$ of the form
$$\CC = \setofst{\Phi_n\circ \dots \Phi_q(M)}{n\geq q}$$
for some small Mandelbrot set $M\subset \hat{M}_{1/q}$. 

A family of {combinatorial types} $\CC$ is \textit{parabolically bounded} \index{family $\CC$ of types!parabolically bounded} if it is a union of a finite set of types and finitely many parabolic tails. This definition is a little hard to work with directly, so let us give a more combinatorial characterization. 

Let $F: (\UU', \FK')\to (\UU, \FK)$ be a periodically repelling renormalizable \pql map and let $\KK$ be the corresponding cycle of small filled Julia sets. 
For any $\bound\geq 1$, we will say that the renormalization of $F$  is \textit{parabolically $\bound$-bounded} if:
\begin{enumerate}
	\item $F$ has combinatorial rotation number $1/q$ for some $q\geq 2$; 
	\item each limb $\FL_j$ of $\FK$ contains at most $\bound$ many components of $\KK$; and
	\item any component of $\KK$ in $\FL_0$ is mapped into $\bigcup_{|j|\leq \bound}\FL_j$ by $F$.
\end{enumerate}

\begin{prop}\label{prop:parabolically bounded combinatorics}
	If $\CC$ is a parabolically bounded family of types, then there is a finite set $\CC'\subset \CC$ and $\bound\geq 1$ so that: if the renormalization of a periodically repelling  \pql map is in $\CC\sm \CC'$, then the renormalization is parabolically $\bound$-bounded.
\end{prop}

\begin{proof}
	Let $f'$  be renormalizable quadratic polynomial corresponding to a parameter in $\hat{M}_{1/q}$ 
	with filled Julia set $\FK'$, limbs $\FL_j'$, and cycle of small filled Julia sets $\KK'$. 
	Let $\phi$ be a $k$-limb insertion from $f'$ to a polynomial $f$ corresponding to a parameter in $\hat{M}_{1/(q+1)}$ with filled Julia set $\FK$ and limbs $\FL_j$. Setting $\KK = \phi(\KK')\cup f(\phi(\KK))$, it follows that $\KK$ is a cycle of small filled Julia sets for $f$. If every limb of $\FK'$ contains at most $\bound$ many components of $\KK'$, then every limb of $\FK$ also contains at most $\bound$ many components of $\KK.$ If $k, q-k> \bound$ and every component of $\KK'$ in $\FL_0'$ is mapped into $\bigcup_{|j|\leq \bound}\FL_j'$ by $f$, then every component of $\KK$ in $\FL_0$ is mapped into $\bigcup_{|j|\leq \bound}\FL_j'$ by $f$.
	 Hence if $k, q-k>\bound$, then $f$ is renormalizable with parabolically $\bound$-bounded combinatorics. 
	
	It follows from the above that, up to removing a finite set, any parabolic tail corresponds to parabolically $\bound$-bounded renormalizations; the proposition immediately follows.
\end{proof}

\begin{rem}
	Conversely, any $\bound$ the maximal family of types containing parabolically $\bound$-bounded renormalizations is parabolically bounded.
\end{rem}

\subsection{Parabolic optimism}\label{sec:parabolic optimism}

Let us now turn to the proof of Theorem D, which we restate here:
\begin{thmD}
	Any parabolically bounded family of prime primitive types is super-optimistic.
\end{thmD}

Let $F$ be a \pql map with parabolically $\bound$-bounded renormalization $\RR F$ and combinatorial rotation number $1/q$ at $\alpha_F$. For any quantities $x_1, \dots, x_n$, we will say that $F$ is  \textit{$[x_1, \dots, x_n]$-high-type-near-degenerate} \index{high-type-near-degenerate (HTND)}(abbreviated HTND) if $q, \WW(\RR F)\largegiven{x_1, \dots, x_n}.$ Let us also recall from Section \ref{sec:renorm} that $F$ has \emph{$M$-amplification} for some $M\geq 1$ when $\WW(\RR F)\geq M\WW(F)$.
Theorem D will be an immediate consequence of:
\begin{theorem}\label{thm:life improvement}
For any $M \geq 1$ and $\bound\geq 1$, if the renormalization of $F$  is {parabolically $\bound$-bounded} and $F$ is $[M, \bound]$-HTND, then $F$ has $M$-amplification.
\end{theorem} 

Indeed, assuming Theorem \ref{thm:life improvement} for now, let us prove Theorem D:
\begin{proof}[Proof of Theorem D]
	Let $\CC$ be a parabolically bounded family of prime primitive types.
  Proposition \ref{prop:parabolically bounded combinatorics} implies that there is some $\bound\geq 2$ and a finite set $\CC'$ such that any renormalizable periodically repelling \pql map with combinatorics in $\CC\sm \CC'$ is parabolically $\bound$-bounded. 
 For any $M> 0$, it then follows from Theorem \ref{thm:life improvement} that if $F$ is a renormalizable periodically repelling \pql map with combinatorics in $\CC\sm \CC'$, combinatorial rotation number $1/q$ at $\alpha_F$, and if $q, \WW(\RR F)\largegiven{\CC, M}$, then $\WW(F)\geq M\WW(\RR F).$ By replacing $\CC'$ with a larger finite subset of $\CC$, we can ensure that $q\largegiven{\CC, M}$ for all such $F$, hence $\CC$ is super-optimistic.	
\end{proof}

The proof of Theorem \ref{thm:life improvement} will occupy the next three sections. For the remainder of this section and the next three sections, we fix some  $\bound\geq 1$ and \pql map $F$ with parabolically $\bound$-bounded renormalization $\RR F$.
We will implicitly assume that all constants depend on $\bound$, for example 
we will just say that $F$ is $[x]$-HTND whenever it is $[x, \bound]$-HTND.

Our first step in proving Theorem \ref{thm:life improvement} is reducing the problem into two cases (with precise definitions appearing in \S\ref{sec:trichotomy}):

\begin{proposition}\label{prop:trichotomy}
    For any $M\geq 1$, if $\WW(\RR F)\largegiven{M}$, then either 
	$F$ has $M$-amplification, $\UU\sm \KK$ is $M$-tame, or 
	$\UU\sm \TT$ has $M$-wide controllable waves.
\end{proposition}

Theorem \ref{thm:life improvement} will then immediately follow by showing that all cases in Proposition \ref{prop:trichotomy} produce amplification:

\begin{prop}\label{prop:controlled WAD}
	For any $M$,  if 
	$\UU\sm \KK$ is $M$-tame and $F$ is $[M]$-HTND, 
	then $F$ has $M$-amplification.
\end{prop}

\begin{prop}\label{prop:very wide waves}
	For any $M$ there exists $N$ such that: if $\UU\sm \TT$ has $N$-wide controllable waves and $F$ is $[M]$-HTND,  
	then $F$ has $M$-amplification.
\end{prop}

Indeed, assuming Propositions  \ref{prop:trichotomy}, \ref{prop:controlled WAD}, and \ref{prop:very wide waves}, we can prove Theorem \ref{thm:life improvement}:

\begin{proof}[Proof of Theorem \ref{thm:life improvement}]
	For any $M$, let $N\geq M$ be the corresponding value in Proposition \ref{prop:very wide waves}. If $F$ is $[M]$-HTND, then it follows from Proposition \ref{prop:trichotomy} that either $\UU\sm \KK$ is $N$-tame,  $\UU\sm \TT$ has $N$-wide controllable waves, or $F$ has  $N$-amplification. In the first two cases, Propositions \ref{prop:controlled WAD} and \ref{prop:very wide waves} respectively imply that $F$ has $M$-amplification; as $N\geq M$ in the third case $F$ also has $M$-amplification.
\end{proof}

We will prove Propositions \ref{prop:trichotomy}, \ref{prop:controlled WAD}, and \ref{prop:very wide waves} in Sections \ref{sec:trichotomy}, \ref{sec:aligned}, and \ref{sec:waves} respectively.

\section{Dichotomy}\label{sec:trichotomy}

In this section we prove the dichotomy in Proposition \ref{prop:trichotomy}.
We will use the definitions of  ideal markings, canonical laminations and weighted arc diagrams, and waves detailed in Appendix \ref{app:markings}.

We equip the complement of the Hubbard tree $\UU\sm \TT$ with the ideal marking $\II$ induced by the sides of $\hat\KK$ in $\ddi \TT$. For any $M\geq 1$, we will say that $\UU\sm \TT$ has \emph{$M$-wide controllable waves} \index{dichotomy!wave case}if there is a wave $S$ in $\UU\sm \TT$
with 
$$\WW(S)\geq M \WW(\RR F)$$
that is either a wave over the Caratheodory boundary of $\LL_0$ 
or is localized to an interval in $\II$.

For any $M\geq 1$, we will say that the horizontal Hubbard arc-diagram $\HH^\hor$ has  \textit{$M$-bounded flux} \index{dichotomy!tame case!bounded flux} if 
$$\flux(\alpha)= O( M \WW(\RR F))$$
for every $\alpha\in \HH^\hor$.
Denoting the restriction of $\wcan(\UU\sm \KK)$ to horizontal and vertical arcs by $X$, we will say that $X$ is \textit{$M$-aligned with $\TT$} \index{dichotomy!tame case!aligned with $\TT$}if for all $n\geq 0$, 
$$n^2\|X|_\AA\|= O(M \|X\|)$$
for every set of arcs $\AA$ satisfying $\min_{\alpha\in \AA}\langle \alpha, \HH^\hor\rangle \geq n$.
We will say that $\UU \sm \KK$ is \textit{$M$-tame} \index{dichotomy!tame case} if $\HH^\hor$ has $M$-bounded flux, $X$ is $M$-aligned with $\TT$, and $\UU\sm \TT$ does not have $M$-wide controllable waves.

We will prove Proposition \ref{prop:trichotomy} in two steps. The first step is showing that we have either bounded flux, wide waves, or amplification:

\begin{prop}\label{prop:bounded flux}
    For any $M\geq 1$, if $\WW(\RR F)\largegiven{M}$, then either $\HH^\hor$ has $M$-bounded flux, $\UU\sm \TT$ has $M$-wide controllable waves, or $F$ has $M$-amplification.
\end{prop}

\begin{proof}
	We assume that there is an arc $\alpha\in \HH^\hor$ with 
	$$
	\flux(\alpha)> (4\bound +3)M\WW(\RR F);
	$$
	otherwise $\HH^\hor$ has $M$-bounded flux.
	We also assume that $\WW(\UU\sm \TT)\leq M \WW(\RR F)$; otherwise $F$ has $M$-amplification.

	It follows from the definition of $\flux(\alpha)$ that there is a lamination $\FF$ on $\UU\sm \KK$ such that every leaf of $\FF$ intersects $\alpha$ and $\flux(\alpha)\leq \WW(\FF)+2.$
	Let $I$ be one of the multi-intervals in $\II$ corresponding to $\alpha$; every leaf of $\FF$ restricts to a path in $\UU\sm \TT$ that has one endpoint in $I$. Such paths either have an endpoint in $\ddi \UU$ or are waves over one of the endpoints of $I$. As $\ell(I)\leq 2\bound$, it follows that there is a localized wave $S$ in $\UU\sm \TT$ satisfying
	$$\WW(\FF)\leq 4\bound \WW(S)+ \WW(\UU\sm \TT).$$
	Thus
	$$
	\WW(S)\geq \frac{\WW(\FF)- \WW(\UU\sm \TT)}{4\bound}\geq M\WW(\RR F),
	$$
	so $\UU\sm \TT$ has $M$-wide controllable waves.
\end{proof}

Our second step is  showing  that we have either an aligned canonical WAD or wide waves:
\begin{proposition}\label{prop:aligned wcan}
	For any $M\geq 1$, if $\WW(\RR F)\largegiven{M}$, then either $X$ is $M$-aligned with $\TT$,  $\UU\sm \TT$ has $M$-wide controllable waves, or $F$ has $M$-amplification.
\end{proposition}

Before proving Proposition \ref{prop:aligned wcan}, we need the a few facts. 
First, we need to compare $\WW(\RR F)$ with $\|X\|$:
\begin{lemma}\label{lem:width to weight}
	We have:
	\begin{equation}\label{eq:width to weight}
		\per \WW(\RR F)\leq 4\|X\|+ O(\per).
	\end{equation}
\end{lemma}

\begin{proof}
	For all $0\leq k < \per$, we set $\BJ_k= \dcara K_j$.
	It follows from Theorem \ref{main estimate}
	that 
	\begin{equation}\label{eq:width to weight 1}
		\sum \WW(\BJ_k)\leq 2\|X\|+O(\per).
	\end{equation}
	It follows from the definitions that $\WW(\RR F)= \WW(\BJ_0).$
	For all $k$ we have
	\begin{equation}\label{eq:width to weight 2}
		\WW(\BJ_0)/2\leq \WW(\BJ_k)\leq \WW(\BJ_0),
	\end{equation}
	see for example \cite[Theorem 9.3]{McMullen94} or \cite[(3.24)]{dudko2023mlcfeigenbaumpoints}. Combining \eqref{eq:width to weight 1} and \eqref{eq:width to weight 2} yields \eqref{eq:width to weight}.
\end{proof}

Next, we need an estimate how paths in $\UU\sm \KK$ restrict to $\UU\sm \TT$:
\begin{lemma}\label{lem:intersection with H}
	For any $n\geq 0$, if $\alpha$ is a path in $\UU\sm \KK$ with $\langle [\alpha], \HH^\hor\rangle\geq n$, then $\alpha$ restricts to  at least $n/2$  paths in $\UU\sm \TT$.
\end{lemma}

\begin{proof}
	Let $H$ be a representative of $\HH^\hor$ such that $\langle \alpha, H\rangle = \langle [\alpha], \HH^\hor\rangle = n$ and let $D$ be the union of the components of $\UU\sm (\KK\cup H)$ that do not touch $\dd \UU$. The paths in $H$ cut $\alpha$ into $n+1$ paths $\alpha_0, \dots, \alpha_n$. Each $\alpha_m$ not in $D$ corresponds to a distinct restriction of $\alpha$ to a path in $\UU\sm \TT$.
	As $\langle \alpha, H\rangle = \langle [\alpha], \HH^\hor\rangle$, two adjacent $\alpha_m$ cannot  both lie in $D$. The number of $\alpha_m$ not in $D$ is therefore at least 
	$\lfloor (n+1)/{2}\rfloor\geq n/2.$
\end{proof}

For each $j\in \Z/q\Z$, the Caratheodory boundary of $\LL_j$ is a multi-interval in $\II$ which we will call the \textit{$j$-th limb of $\II$} and denote by $\BL_j$. For any interval $I\subset \R$, we similarly denote by $\BL_I$ the smallest multi-interval containing $\dcara \LL_I$. The last fact we will need bounds the combinatorial lengths of limbs:

\begin{lemma}\label{lem:comb size of limbs}
	For all $j$, $|\BL_j| \leq 6\bound$.
\end{lemma}

The proof of Lemma \ref{lem:comb size of limbs} relies on the following easy  graph theory exercise which we leave to the reader:

\begin{lem}\label{lem:graph edges}
	Let $G$ be an abstract tree with vertices $V$ and edges $E$. 
	If  $V$ can be decomposed into $V_1\cup V_2$ such that
	\begin{enumerate}
		\item $|V_1|\leq  b+1$ for some $b\geq 1$, and
		\item every vertex in $V_2$ has valence at least 3,
	\end{enumerate}
	then $|E| \leq 2b-1$.
\end{lem} 

\begin{proof}[Proof Lemma \ref{lem:comb size of limbs}]
	Fixing some $j$, let $\GG$ be the hull in the combinatorial Hubbard tree of $\alpha_F$ and all the vertices of $\TT$ contained in $\LL_j$. 
	Let $V_1(\GG)$ denote the set of all vertices of $\GG$ corresponding to either $\alpha_F$ or a component of $\KK$, 
	 let $V_2(\GG)$ denote the remaining vertices, and let $E(\GG)$ denote the edges of $\GG$. As each vertex $v\in V_2(\GG)$ is a branch point of $\TT$, the valence in $\GG$ of any such $v$ is at least 3.
	 
	 Every interval of $\BL_j$ corresponds to a side in $\TT$ of some $v\in V_1(\GG)\sm \{\alpha_F\}$ or $e\in E(\GG)$. Each such vertex or edge has at most $2$ sides, hence 
	 \begin{equation}\label{eq:size of limbs}
	 	|\BL_j|\leq 2\left(|E(\GG)| +|V_1(\GG)|-1\right).
	 \end{equation}
	 As the renormalization of $F$ is parabolically $\bound$-bounded, $|V_1(\GG)|\leq \bound +1$. 
	 The desired bound then follows immediately from \eqref{eq:size of limbs} and Lemma \ref{lem:graph edges}.
	 \end{proof}

With the above facts, we are ready to prove Proposition \ref{prop:aligned wcan}.
\begin{proof}[Proof of Proposition \ref{prop:aligned wcan}]
	We denote by $\FF$ the set of all horizontal and vertical leaves in $\FF_{\can}(\UU\sm \KK)$, so $\WW(\FF) = \|X\|$. For all $n\geq 0$, let $\FF_n\subset \FF$ denote the sublamination consisting of leaves $\beta$ with  $\langle [\beta],\HH^\hor \rangle\geq n.$ Thus if $\AA$ is a set of arcs with $\min_{\alpha\in \AA}\langle \alpha, \HH^\hor\rangle\geq n$,  then 
	$$\|X|_{\AA}\|\leq \WW(\FF_n).$$
	We assume that \begin{equation}\label{eq:not aligned}
		{n^2\WW(\FF_{n})}\geq  {360 \bound M}{\WW(\FF)}
	\end{equation}
	for some $n\geq 0$; otherwise $X$ is $M$-aligned with $\TT$ by definition. We also assume that $\WW(\UU\sm \TT)\leq M\WW(\RR F)$; otherwise $F$ has $M$-amplification.

	Let $\GG$ be the restriction of $\FF$ by the inclusion $\UU\sm \TT\subset \UU\sm \KK$.
	Lemma \ref{lem:intersection with H} implies that every leaf of $\FF_n$ contributes at least $n/2$ leaves to $\GG$, so
	\begin{equation}\label{eq:hor_weight_totals}
		n^2\WW(\FF_n)\leq 4\WW(\GG)
	\end{equation}
    by Lemma \ref{lem:harmonic and geometric sum}.
	
	For every interval $\BI\in \II^0$, let $\GG(\BI)\subset \GG$ denote the sublamination consisting of leaves with an endpoint in $\BI$. 
	As $\II$ contains at most $6\bound q\leq 6\bound \per$ intervals by Lemma \ref{lem:comb size of limbs},
	there is an interval  $\BI\in \II$ with
	\begin{equation}\label{eq:large_wave}
		6\bound\per\WW({\GG}(\BI))\geq \WW(\GG).
	\end{equation}
	As 
	$$\per\WW(\RR F)\leq 4\WW(\FF)+ O(1)$$
	by Lemma \ref{lem:width to weight}, 
	 \eqref{eq:not aligned}, \eqref{eq:hor_weight_totals}, and \eqref{eq:large_wave} imply that 
	\begin{equation*}
		\frac{\WW({{\GG}}(\BI))}{\WW(\RR F)}\geq \frac{n^2\WW(\FF_{n})}{120\bound\WW(\FF)}\geq 3M
	\end{equation*}
	when $\WW(\RR F)\gg 1$. 
	There is a wave $S\subset \GG(\BI)$ with 
	$$2\WW(S)+ \WW(\UU\sm \TT)\geq \WW(\GG (\BI));$$
	hence $\WW(S)\geq M \WW(\RR F)$, so $\UU\sm \TT$ has $M$-wide controllable waves.
\end{proof}

Combining Propositions \ref{prop:bounded flux} and \ref{prop:aligned wcan} immediately yields Proposition \ref{prop:trichotomy}.

\begin{rem}
    Note that the only part of the definition of parabolically bounded used in the proof of Proposition \ref{prop:trichotomy} is the bound $\bound$ on the number of small filled Julia sets in each limb of $\FK$. Hence Proposition \ref{prop:trichotomy} holds more generally; this observation may be useful in producing \emph{a priori} bounds for other families of combinatorial types.
\end{rem}

\section{The tame case}\label{sec:aligned}

In this section, we will prove Proposition \ref{prop:controlled WAD}. 
Recall that we have fixed some parabolically $\bound$-bounded \pql map in Section \ref{sec:life improvement}, and that we implicitly assume all constants to depend on $\bound$.

We define the \textit{translation region} \index{translation region}
to be the subsurface $U_{(\bound, q-\bound)}.$
Any arc $\alpha$ in the translation region has a unique lift $F^*\alpha$ that is also an arc, and usually in the translation region. 
We can then form the equivalence relation on arcs in the translation region generated by $\alpha\sim F^*\alpha$; we call each resulting equivalence class a \emph{translation class}\index{translation region!translation class}. 
We will say that a set of arcs $\AA$ in the translation region is \textit{good} \index{translation region!good arcs} if it contains only horizontal arcs and for any $\alpha, \beta\in \AA$, no arc in the translation class of $\alpha$ intersects an arc in the translation class of $\beta$. 
When a certain good set of arcs is clear from the context, 
we'll say that an arc is good if it belongs to this good set.

For an arc $\alpha$ in $\UU\sm \KK$,  
we define its \textit{combinatorial length} \index{translation region!combinatorial length}to be 
$$\ell(\alpha) = 1+\langle \alpha, \HH^\hor\rangle.$$
For any set of arcs $\AA$, we denote $\ell(\AA) = \sup_{\alpha\in \AA}\ell(\alpha).$

We will arrange our whole discussion so as to consider only horizontal and vertical arcs (in particular, we ignore bivertical arcs);
we accordingly denote the horizontal and vertical part of the canonical WAD by $X = \wcan^{h+v}(\UU\sm \KK)$. 
Our first step towards proving Proposition \ref{prop:controlled WAD} will be observing that most of the canonical weight is supported on a good set of arcs with bounded combinatorial length:
\begin{lem} \label{lem:bounded-combinatorial}	
	For any $M$ and $\epsilon$ there exists $L$ such that: 
	if $F$ is $[M, \epsilon]$-HTND and $\UU\sm \KK$ is $M$-tame, then either $F$ has $M$-amplification or there is a good set of arcs $\AA$  with $\ell(\AA)\leq L$ and $$\|X|_\AA\|> (1-\eps)\|\xhv\|.$$
\end{lem}

Conversely, we will also see that a good set of arcs with bounded combinatorial length cannot have most of the canonical weight:

\begin{lem}\label{lem:hammer}
	For any $L$ there exists $\epsilon$ such that: if $F$ is $[L]$-HTND and $\AA$ is a good set of arcs with $\ell(\AA)\leq L$, then $$\|X|_\AA\|\leq (1-\eps)\|\xhv\|.$$
\end{lem}

Together Lemmas \ref{lem:bounded-combinatorial} and \ref{lem:hammer} are almost contradictory, however they have opposite dependence on quantifiers. To produce a contradiction, we first note that Lemma \ref{lem:bounded-combinatorial} implies that  translation classes with a definite proportion of weight have bounded combinatorial length:

\begin{cor} \label{cor:length-bound-from-weight}
	For any $M$ and $\eps$ there exists $L$ such that: if  $F$ is $[M, \eps]$-HTND, $\UU\sm \KK$ is $M$-tame,  and $\AA$ is a translation class with $\|X|_\AA\|\geq \eps \|\xhv\|$, then either $F$ has $M$-amplification or  $\ell(\AA)\leq L$. 
\end{cor}

The last observation we need is that a good set of arcs with bounded combinatorial length is contained in a bounded number of translation classes:

\begin{lem}\label{lem:bounded good}
	For any $L$, if $F$ is $[L]$-HTND and $\AA$ is a good set of arcs with $\ell(\AA)\leq L$, then $\AA$ is contained in at most $3\bound$ translation classes.
\end{lem}

By playing Lemma \ref{lem:bounded-combinatorial} and Corollary \ref{cor:length-bound-from-weight} back and forth against each other to control the heaviest $n$ translation classes, and using Lemma \ref{lem:bounded good} to bound $n$, 
we can prove Proposition \ref{prop:controlled WAD}:

\begin{proof}[Proof of Proposition \ref{prop:controlled WAD}]
	Setting $N = 3\bound$,  $\eps_0 = 1$, and fixing some $M$, there exist sequences $(L_n)_{n=1}^N$ and $(\eps_n)_{n=1}^N$ such that each $L_n$ depends on $$\tilde\eps_n:=\frac{\eps_{n-1}}{2(N-n+1)}$$ as in Corollary \ref{cor:length-bound-from-weight}, and each $\eps_n$ depends on $\tilde{L}_n:=\sup_{i\leq n}L_i$ as in Lemma \ref{lem:hammer}. We set $\delta= \min \eps_n/2$ 
	and let $L$ depend on $\delta$ as in Lemma \ref{lem:bounded-combinatorial}. Note that all the above constants depend only on $\bound$ and $M$, thus when $F$ is $[M]$-HTND it follows that $F$ is HTND relative to each $\tilde \epsilon_n$ and $\tilde L_n$.
	We assume for the sake of contradiction that $F$ does not have $M$-amplification.
	
	It follows from Lemma \ref{lem:bounded-combinatorial} that there is a good set of arcs $\AA$  with $\ell(\AA)\leq L$ and
	\begin{equation}\label{eq:good arcs}
		\|X|_\AA\|> (1-\delta) \|\xhv\|.
	\end{equation} It then follows from Lemma \ref{lem:bounded good} that we can write $\AA=\AA_1\cup \dots \cup\AA_N$, where each $\AA_n$ is contained in a single translation class and $\|X|_{\AA_n}\|\geq \|X|_{\AA_{n+1}}\|.$ 
	
	Let us suppose that for some $0\leq n < N$ we have 
	\begin{equation}\label{eq:good arcs 2}
		\sum_{i = 1}^{n}\|X|_{\AA_i}\|\leq (1-\eps_n)\|\xhv\|
	\end{equation}
	and  $\ell(\AA_i)\leq L_i$ for all $1\leq i \leq n$; both of these facts hold trivially for $n = 0$.  
	As we have ordered the $\AA_i$ in decreasing order by weight, it then follows from \eqref{eq:good arcs} and \eqref{eq:good arcs 2} that 
	$$\frac{\|X|_{\AA_{n+1}}\|}{\|\xhv\|}\geq \frac{(1-\delta)- (1-\eps_n)}{N-n}= \frac{\eps_n-\delta}{N-n}\geq \frac{\eps_n}{2(N-n)}=\tilde\eps_{n+1}.$$
	Corollary \ref{cor:length-bound-from-weight} therefore implies that $\ell(\AA_{n+1})\leq L_{n+1}$. 
	As $\AA_1\cup \cdots \cup \AA_{n+1}$ is a good set of  arcs with combinatorial length at most $\tilde L_{n+1}$, it then follows from Lemma \ref{lem:hammer} that 
	$$\sum_{i= 1}^{n+1}\|X|_{\AA_n}\|\leq(1-\eps_{n+1})\|\xhv\|.$$
	By induction on $n$, we can conclude that 
	$$\|X|_\AA\| = \sum_{i=1}^N\|X|_{\AA_i}\|\leq (1-\eps_N)\|\xhv\|\leq  (1-2\delta)\|\xhv\|,$$
	which contradicts \eqref{eq:good arcs}.
\end{proof}

We will prove Lemmas \ref{lem:bounded-combinatorial}, \ref{lem:bounded good}  and Corollary \ref{cor:length-bound-from-weight} in Section \ref{sec:good arcs} after some preparation in Section \ref{sec:translation markings}. We will outline the proof of Lemma \ref{lem:hammer} in Section \ref{align:ggc}; the details of the argument will then occupy the rest of this section.
\subsection{Translation markings} \label{sec:translation markings}
In this section we apply the theory of translation markings introduced in Appendix \ref{app:transl-mark} to our setting. For simplicity we will assume that $\FK$ is locally connected; the general case can be handled similarly.

Let us think of $\D$ as the universal cover of $\UU^0 \sm \KK^0$, with $0\in \D$ mapping to the fixed point $\alpha_F$.
Simultaneously viewing $\D$ as the universal cover of $\UU^1 \sm \KK^1$, with $0\in \D$ mapping to the fixed point $\alpha_f$, the maps $f$ and $\iota$ lift to a holomorphic isomorphism $f_\inbr: \D\to \D$ and a holomorphic immersion $\iota_\inbr: \D\to \D$ respectively; we denote $F_\inbr = (f_\inbr, \iota_\inbr)$.

We consider the components of $\KK$ that are ``closest" to the fixed point $\alpha_F$, that is those that are connected to $\alpha_F$ by a path in $\TT$ that avoids $\KK$. 
Each such path 
lifts to a path in $\D$ connecting $0$ to an interval in $\dD$ that covers the ideal boundary of a component of $\KK$.  
We define the \textit{inner translation marking} \index{translation marking!inner $(\FF_\inbr, \II_\inbr)$} $\II_\inbr$ to be the ideal marking of $\D$ induced 
by all intervals that can be obtained  this way.
\begin{prop}\label{prop:inner translation marking}
	The pair $(F_\inbr, \II_\inbr)$ is an $(n, p)$-translation marking with $n \leq q$ and $p\leq \bound$. If $\UU\sm \KK$ is $M$-tame, then   
	$$\WW(F_\inbr, \II_\inbr)\leq O(M\WW(\RR f)).$$
\end{prop}

\begin{proof}
	We can write $\II_\inbr = (\BI_j)_{j\in \Z/N\Z}$ for some $N\geq 1$, labeled so that $\BI_j$ covers a component of $\KK$ for all odd $j$ and so that the counter-clockwise circular ordering of the intervals is the ordering on $\Z/N\Z$.
    We set $\TT_j:= \TT\cap \FL_j$ and $\TT_j':= F^{-1}(\TT)\cap \FL_j$. 
	As $\TT_j'= \TT_j$ for all $\bound \leq j < q-\bound -1$, it follows that $\iota_\inbr$ extends to a proper map fixing $\BI_j$ for all odd $j$. 
	As $f$ maps $\TT_j'$ homeomorphically to $\TT_{j+1}$ for all $\bound \leq j < q-\bound -1$, we can choose the labeling  so that  $f_\inbr$ extends to a homeomorphism from $\BI_j$ to $\BI_{j+2p}$ for some integer $p$ and all $0< j < N-2p$.
	Thus  $(F_\inbr, \II_\inbr)$ is a translation marking. 
	As each limb of $\FK$ contains at most $\bound$ components of $\KK$, it follows that $p\leq \bound$ and $2(n+1)p = N$ for $n \leq q-2\bound$.
	
	Let $\II_\inbr'$ be the ideal marking of $\D$ induced by all the intervals corresponding to the components of $\KK$ closest to $\alpha_F$, not just the ones in the translation region.
	Thus $\II_\inbr'$ is a refinement of $\II_\inbr$: if  $j\neq 0$ then $\BI_j$ is an interval in $\II_\inbr'$, and $\BI_0$ is a multi-interval in $\II_\inbr'$ of length at most $2(2\bound +1)+1$.
	If  $ \BI'\in  \II_\inbr'$  covers the Caratheodory  boundary of a component of $\KK$, then we have $$\WW( \BI')\leq \WW( \RR F).$$
	Otherwise, $\WW(
	 \BI')$ is exactly the flux through some arc $\HH^\hor$; as $\UU\sm \KK$ is  $M$-tame it then follows that 
	$$\WW( \BI')\leq O(M\WW(\RR F)).$$
	Thus 
	\begin{align*}
		\WW(F_\inbr, \II_\inbr)& = \sum_{|j|\leq 2p}\WW(\BI_j)\\
		& \leq 
		(4p+2(2\bound+1)+1)(\WW(\RR F)+O(M\WW(\RR F)))\\
		& = O(M\WW(\RR f))
	\end{align*}
	as desired.
\end{proof}

Let us now view $\D$ as the universal cover of $\UU^0\sm \KK^0$, with $0\in \D$ mapping to a point $*\in \UU^0 \sm \FK$ in the translation region. 
Simultaneously viewing $\D$ as the universal cover of $\UU^1\sm \KK^1$ with $0\in \D$ mapping to the same base point, the maps $f$ and $\iota$ lift to a holomorphic isomorphism $f_\outbr: \D\to \D$ and a holomorphic immersion $\iota_\outbr: \D\to \D$ respectively; we denote $F_\outbr = (f_\outbr, \iota_\outbr)$.

For every side of a component of $\KK$, any path in $\UU\sm \FK$ connecting $*$ to that side lifts to a path in $\D$ connecting $0$ to an interval in $\dD$ that covers the ideal boundary of a component of $\KK$.  We define the \textit{outer translation marking} $\II_\outbr$ \index{translation marking!outer $(\FF_\outbr, \II_\outbr)$} to be the ideal marking of $\D$ induced by all the intervals that can be obtained in this way from paths inside the translation region.

\begin{prop}\label{prop:outer translation marking}
	The pair $(F_\outbr, \II_\outbr)$ is a cardinality $\leq 4\bound q$ translation marking. If $\UU\sm \KK$ is $M$-tame, then 
	\begin{equation}\label{eq:outer translation marking}
		\WW(F_\outbr, \II_\outbr)\leq O(M\WW(\RR f)).
	\end{equation}	
\end{prop}
\begin{proof}
	The argument that $(F_\outbr, \II_\outbr)$ is a translation marking is identical to the proof in Proposition \ref{prop:inner translation marking}. The computation of the standard width is similar; the only difference is now the standard width is bounded by the width to $O(\bound)$ many intervals that cover a component of $\KK$, the flux through $O(\bound)$ many arcs in $\HH^\hor$, and the width of two waves over the $\dcara\LL_0$. 
	When $\UU\sm \KK$ is $M$-tame, so $\UU\sm \TT$ does not have $M$-wide controllable waves, we similarly recover \eqref{eq:outer translation marking}.
\end{proof}

\subsection{Good arcs}\label{sec:good arcs}

In this subsection we study good sets of arcs and prove Lemma \ref{lem:bounded-combinatorial}, Corollary \ref{cor:length-bound-from-weight}, and Lemma \ref{lem:bounded good}. We recall the definition of a long arc for a translation marking from Appendix \ref{app:transl-mark}.

Let us say that an arc in the translation region  \textit{has a long segment} \index{translation region!long segment} if its lift into either $(\D, \II_\inbr)$ or $(\D, \II_\outbr)$ contains a long arc.
Note that any arc with a long segment must intersect an arc in its translation class, and hence cannot be good. 
We observe that arcs with bounded combinatorial length and without long segments stay in a bounded part of the translation region:

\begin{prop}\label{prop:comb-length-bound}
	If $\alpha$ is a horizontal arc in the translation region with $\ell(\alpha)\leq L$ and $\alpha$ has no long segments, then there is an interval $I\subset \R$ with $|I|\leq 2L$ such that $\alpha$ is contained in $U_I$.
\end{prop}

\begin{proof}
	For all intervals $I\subset \R$, let $\HH_I^\hor\subset \HH^\hor$ denote the subset consisting of arcs in $U_I$. 
	The arcs in $\HH^\hor$ cut $\alpha$ into segments $\alpha_0, \dots, \alpha_N$, where $$N = \langle \alpha, \HH^\hor\rangle= \ell(\alpha)-1.$$ 
	For each $n$, let $I_n$ be the minimal interval so that $\alpha_n$ is contained in $U_{I_n}$. 
	Thus $\alpha$ is contained in $U_I$ for some interval $I$ with 
	\begin{equation}
		|I| \leq \sum_{n=0}^N|I_n|.
	\end{equation}

	Let $D$ be the component of $\UU\sm (\KK\cup \HH^\hor)$ containing $\alpha_F$ when $\HH^\hor$ is chosen in $\UU\sm \TT$. Any segment $\alpha_n$ in $D$ corresponds to a lift of $\alpha$ to an arc in $(\D, \II_\inbr)$; if $|I_n|> 3$ then then the lifted arc must be long. Any segment $\alpha_n$ not in $D$ either satisfies $I_n = 1$ or corresponds to a lift of $\alpha$ to an arc in $(\D, \II_\outbr)$; again if $|I_n|> 3$ then the lifted arc is long. 
	As $\alpha$ has no long segments, it follows that $|I_n|\leq 2$ for all $n$. Hence $|I|\leq 2 (N+1)=2 \ell(\alpha)$.
\end{proof}

With Proposition \ref{prop:comb-length-bound}, we can prove Lemma \ref{lem:bounded good}:
\begin{proof}[Proof of Lemma \ref{lem:bounded good}]
	Let $\AA$ be a good set of arcs with $\ell(\AA)\leq L$ and let $k$ be the number of distinct translation classes that intersect $\AA$.
	Proposition \ref{prop:comb-length-bound} implies that the union of all these translation classes form an arc diagram in $\UU\sm \KK$  that contains at least $(q-2\bound-3L)k$ arcs. The total number of arcs in an arc-diagram in $\UU\sm \KK$ is at most $3\bound q$, so
	$$k \leq \frac{3\bound q}{q-2\bound -3L}\leq 3\bound+1/2$$
	when $q\largegiven{L}$. As $k$ is an integer, the proof is complete.
\end{proof}

We can now prove Lemma \ref{lem:bounded-combinatorial} and Corollary \ref{cor:length-bound-from-weight}. 

\begin{proof}[Proof of Lemma \ref{lem:bounded-combinatorial}]
	We assume that $F$ is $M$-tame, and that $F$ does not have $M$-amplification.	
	Let $\Gamma$ denote the set of all arcs in the support of $X$. For all $L\geq 1$, we denote $P(L)= U_{(\bound+3L, q-\bound-3L)}$. Let $\Gamma_L\subset \Gamma$ be the subset consisting of horizontal arcs in $P(L)$, and let  $G_L\subset \Gamma_L$ be the subset consisting of arcs with combinatorial length at most $L$
	and without long segments. Theorem \ref{main estimate} implies that 
	$$\|\xhv\|\geq \per (\WW(\RR F)/2-O(1))\geq (q/3)\WW(\RR F)$$
	when $\WW(\RR F)\gg 1$.

	Any arc in $\Gamma\sm \Gamma_L$ must contribute its weight in $X$ to either $\WW(F)$, the flux through some arc in  $\HH^\hor$ outside $P(L)$, the total width of some component of $\KK$ outside $P(L)$, or the width of a wave over $\dcara\LL_0$ in $\UU\sm \TT$. As we assume that $\UU\sm \KK$ is $M$-tame and $F$ does not have $M$-amplification, it follows that
	\begin{equation}\label{eq:arcs outside translation region}
		\|X|_{\Gamma\sm \Gamma_L}\|  = 
		O(L M \WW(\RR F))\leq (\eps/2)\|X\|
	\end{equation}
	when $F$ is $[M, \eps, L]$-HTND.

	For any $L\geq 1$ and $\gamma\in \Gamma_L\sm G_L$, 
	then either $\ell(\gamma)> L$ or $\gamma$ has a long segment. 
	As $\UU\sm \KK$ is $M$-tame, so $\HH^\hor$ is $M$-aligned with $\TT$, it therefore follows from 
	Propositions \ref{prop:inner translation marking}, \ref{prop:outer translation marking} and Theorem \ref{thm:bounded short} 
	that 
	\begin{equation}\label{eq:arcs inside translation region}
		\|X|_{\Gamma_L\sm G_L}\|\leq \frac{O(M\|X\|)}{(L-1)^2}+ O(M\WW(\RR F)+q)\leq (\eps/2)\|X\|
	\end{equation}
	when $L\largegiven{\eps}$ and $F$ is $[M, \eps]$-HTND.
	Combining \eqref{eq:arcs outside translation region} and \eqref{eq:arcs inside translation region} yields 
	$$\|X|_{\Gamma\sm G_L}\|= \|X|_{\Gamma\sm \Gamma_L}\|+ \|X|_{\Gamma_L\sm G_L}\|\leq \eps\|X\|.$$
	
	Now we observe that $X(F^*\alpha)\geq X(\alpha)$ for any arc $\alpha$ in the translation region with $F^*\alpha$ defined.
	Thus for any $\alpha, \beta\in G_L$, we have $\langle (F^*)^k\alpha, \beta\rangle = 0$ for all $0\leq k < 3L$. 
	As $\alpha, \beta$ have combinatorial length $\leq L$ and no long segments, it then follows from Proposition \ref{prop:comb-length-bound} 
	that the translation classes of $\alpha$ and $\beta$ are disjoint. Hence $G_L$ is good.
\end{proof}

\begin{proof}[Proof of Corollary \ref{cor:length-bound-from-weight}]
	Fixing some $M$ and $\eps$, let $L$ be the corresponding value from Lemma \ref{lem:bounded-combinatorial}. If $F$ is $[M, \eps]$-HTND, $\UU\sm \KK$ is $M$-tame, $F$ does not have $M$-amplification, and $\AA$ is a translation class with $\|X|_\AA\|\geq  \eps\|X\|$, then Lemma \ref{lem:bounded-combinatorial} implies that some arc in $\AA$ has combinatorial length at most $L$. As every arc in $\AA$ has the same combinatorial length, it follows that $\ell(\AA)\leq L$.
\end{proof}

\input{not-all-good2}

\section{The wave case}\label{sec:waves}

In this section, we will prove Proposition \ref{prop:very wide waves}.
We recall from Section \ref{sec:parabolic optimism} that $F: \UU'\to \UU$ is a parabolically $\bound$-bounded \pql map, with Hubbard tree $\TT$ and vertices $\hat\KK$ as in Section \ref{sec:hubbard trees}. 
We also recall the definition of ideal markings, waves from  Appendix \ref{app:ideal markings and waves}.
We equip $\UU\sm \TT$ with the ideal marking  $\II$  induced by the sides of $\hat\KK$ in $\TT$.
We will say a wave $A$ in $\UU\sm \TT$ is \textit{$[x_1, \dots, x_n]$-wide} \index{waves!wide} if $$\WW(A)\largergiven{x_1, \dots, x_n}\WW(\RR F).$$ 
We will implicitly assume that all constants in this section depend on $\bound$, for example if a wave is $[\bound, x]$-wide then we will usually just say that it is $[x]$-wide.

Recall from Appendix \ref{app:ideal markings and waves} that a wave is \emph{localized} if it ends at a single interval.
For any $k\geq 1$, we will say that a wave $A$ in $\UU\sm \TT$ is \emph{$k$-controllable} \index{waves!controllable} if we can write it as
$A = A_0\cup \dots\cup A_k$, where $A_0$ crests over $\BL_0$ and each $A_j$ with $1\leq j \leq k$ is a localized wave. 
For any $0< \delta< 1$ and $k\geq 0$, let us say that a wave $A$ in $\UU\sm \TT$ is \textit{$(\delta, \eta)$-terminal} \index{waves!terminal} if either  
$$\WW(\UU\sm \TT)> \delta\WW(A) \text{ or }\WW(B)> (1+\delta)\WW(A)$$ for some $\eta$-controllable wave $B$. 
{Thus, a terminal wave produces either definite vertical width or a definitely wider controllable wave.}  

Our goal in this section is prove that all wide controllable waves are terminal (this is the analogue in our setting to the \emph{Wave Lemma} from \cite[Lemma 5.1]{dudko2023mlcfeigenbaumpoints}):
\begin{lem}\label{lem:terminal}
	There exists some universal $\delta$ and $\eta$ such that: for all $k$,  any $[k]$-wide $k$-controllable wave is $(\delta, \eta)$-terminal.
\end{lem}

Proposition \ref{prop:very wide waves} then follows immediately from 
Lemma \ref{lem:terminal}:
\begin{proof}[Proof of Proposition \ref{prop:very wide waves}]
	Let $\delta$ and $\eta$ be as in Lemma \ref{lem:terminal}.
	Let $A$ be a $\eta$-controllable wide wave; we choose $A$ to be nearly maximal in the sense that sny $\eta$-controllable wave has width less than $(1+\delta)\WW(A)$.
    As any $1$-controllable wave is $\eta$-controllable, if $\UU\sm \TT$ has $N$-wide controllable waves, then we can choose $A$ so that
	$\WW(A)\geq N \WW(\RR f)$.
	
	By Lemma \ref{lem:terminal}, if $N\geq M/\delta$ is sufficiently large, then $A$ is $(\delta, \eta)$-terminal. 
	The maximality of $A$ then implies that 
	$$\WW(f) \geq \WW(\UU\sm \TT)> \delta\WW(A)\geq M\WW(\RR f),$$
	so $f$ has $M$-amplification.
\end{proof}

Our proof of Lemma \ref{lem:terminal} comes in three parts. 
Let us say that a wave  in $\UU\sm \TT$ is \textit{short} if no leaf crests over a limb of $\II$. 
By Lemma \ref{lem:comb size of limbs}, the combinatorial length of each limb of $\II$ is at most $6\bound$, so any short wave is automatically $12\bound$-controllable: if an end of a wave $A$ has combinatorial length at least $12\bound$, then that end must contain a limb of $\II$, hence $A$ is not short.
We define the \textit{short part} $A^\sh$ of a wave $A$ to be the unique maximal short subwave.
For any $0< \lambda< 1$, we will say that $A$ is \textit{$\lambda$-short} \index{waves!short} if $\WW(A^\sh)> \lambda\WW(A).$
Our first 
step towards Lemma \ref{lem:terminal} is showing that wide controllable waves are either short or terminal (see Section \ref{sec:breaking} for the proof):
\begin{prop}\label{prop:mostly short}
	For any $\lambda$ there exists $\delta$ such that: for all $k$, any $[k, \lambda]$-wide $k$-controllable wave is either $\lambda$-short or $(\delta, \eta)$-terminal for some universal $\eta$.
\end{prop}

For any $0< \lambda< 1$ and wave $A$ in $\UU\sm \TT$ with crest $I$, we will say that a wave $B$ in $\UU\sm \TT$ is a 
\textit{critical $\lambda$-visit} \index{waves!critical visit} of $A$ if $\WW(B)> \lambda\WW(A)$ and 
$B$ is a short wave over $F_*^nI$ for some $n\geq 1$ such that $F_*^nI$ intersects $\BL_0$.

Our second step in proving Lemma \ref{lem:terminal} is showing that short waves either have a critical visit or are terminal (see Section \ref{sec:visits} for the proof):
\begin{prop}\label{prop:critical visit}
	For any $\lambda$ there exists $\delta$ such that: any $[\lambda]$-wide short wave either has a critical $\lambda$-visit or is $(\delta, \eta)$-terminal for some universal $\eta$.
\end{prop}

The final step of our proof is observing that waves can have only a bounded number of critical visits (see Section \ref{sec:intervals} for the proof):
\begin{lemma}\label{lem:bounded visits}
	For some universal $N> 1$, there is no sequence of  waves $(A_n)_{n=1}^N$  such that each $A_{n+1}$ is a critical visit of $A_n$. 
\end{lemma}

Assuming the Propositions \ref{prop:mostly short}, \ref{prop:critical visit} and Lemma \ref{lem:bounded visits} for now, we can prove Lemma \ref{lem:terminal}:

\begin{proof}[Proof of Lemma \ref{lem:terminal}]
    Let $N$ be the universal constant in Lemma \ref{lem:bounded visits}, and let $\eta$ be the maximum of the universal constants in Propositions \ref{prop:mostly short} and \ref{prop:critical visit}.
    Let us fix a sequence $$0< \lambda_N< \lambda_{N-1}< \cdots <\lambda_0 < \lambda_{-1} = 1$$
    so that for $\delta_n$ depending on $\lambda_n$ as in Propositions \ref{prop:mostly short} or \ref{prop:critical visit}, we have
    $$(1+\delta_n)\lambda_{n-1}^n> (1+\delta)$$
    for some $\delta>0$ and all $n$. Note that we can choose $\delta$ and the sequence $(\lambda_n)_{n=0}^{N-1}$ universally; in particular, we may assume that any $[k]$-wide wave is $[\lambda_n]$-wide for all $n$.

    Now let $A=A_0$ be a $k$-controllable $[k]$-wide wave. Proposition \ref{prop:mostly short} implies that  $A_0$ is either $\lambda_0$-short or $(\delta_0, \eta)$-terminal.
	We set $A_1= A^\sh$, so if $A_0$ is $\lambda_0$-short then $A_1$ is $[k]$-wide.
	
    For some $1\leq n <  N$, suppose $A_n$ is a short $[k]$-wide wave. Proposition \ref{prop:critical visit} implies that $A_n$ either has a critical $\lambda_{n}$-visit $A_{n+1}$ or is $(\delta_n, \eta)$-terminal. In the former case, the wave $A_{n+1}$ is automatically short and $[k]$-wide.
	By induction on $n$, it follows that either there is a  $(\delta_n, \eta)$-terminal wave $A_n$ satisfying
	$$\WW(A_n)> \WW(A)\prod_{m=0}^{n-1}\lambda_m>  \lambda_{n-1}^n\WW(A),$$
	for some $0\leq n < N$, 
	or there is a sequence of waves $(A_n)_{n=1}^{N}$ such that each $A_{n+1}$ is a critical visit of $A_n$. 
	The latter case is impossible by Lemma \ref{lem:bounded visits}, and in the former case our choice of the sequence $(\lambda_n)_{n=0}^N$ implies that $A$ is $(\delta, \eta)$-terminal.
\end{proof}

We will prove the above Propositions by studying the dynamics of $F$. More precisely, setting $\TT' = f^{-1}(\TT)$, so $f^{-1}(\UU\sm \TT)= \UU' \sm \TT'$, the maps $f$ and $\iota$ induce a covering and an immersion from $\UU'\sm \TT'$ to $\UU\sm \TT$; by abuse of notation we will also call these induced maps $f$ and $\iota$ respectively and denote the pair by $F = (f, \iota)$.
In Section \ref{sec:intervals}, we will study the action of this $F$ on intervals in $\II$ and prove Lemma \ref{lem:bounded visits}. 
In Section \ref{sec:breaking}, we will study the action of $F$ on waves in $\UU\sm \TT$ and prove Proposition \ref{prop:mostly short}.
In Section \ref{sec:visits}, we will study the action of large iterates of $F$ on waves and prove Proposition \ref{prop:critical visit}.

\subsection{Interval dynamics}\label{sec:intervals}

In this subsection we will study the action of $F_*$ on $\II$; {we recall the definition of this action from Appendix \ref{app:dynamics}}.  We first prove several facts that will be used in Sections \ref{sec:breaking} and \ref{sec:visits}, and then prove Lemma \ref{lem:bounded visits}.
{Let us first briefly outline the proof of Lemma \ref{lem:bounded visits} now for the reader's convenience. If $A_1, \dots, A_N$ is a sequence of critical visits and $I_n$ is the crest of each $A_n$, then the expanding action of $F$ on the Hubbard tree forces the combinatorial length of $I_n$ to grow after a bounded number of visits; eventually the crest becomes long enough to contradict the requirement that $A_N$ is short.}

First, we need to understand exactly how $F$ acts on $\II$. 
As the immersion $\UU'\to \UU$ is injective in a neighborhood of $\FK$, we can identify $\TT'$ with a subset of $\FK$; by abuse of notation we will also denote this subset by $\TT'$.
We will implicitly use this identification throughout this section. 
In particular, it provides us with the following description of the action of $\iota^*$ on intervals:

\begin{figure}
	\begin{center}
		\def\svgwidth{5.5in}
		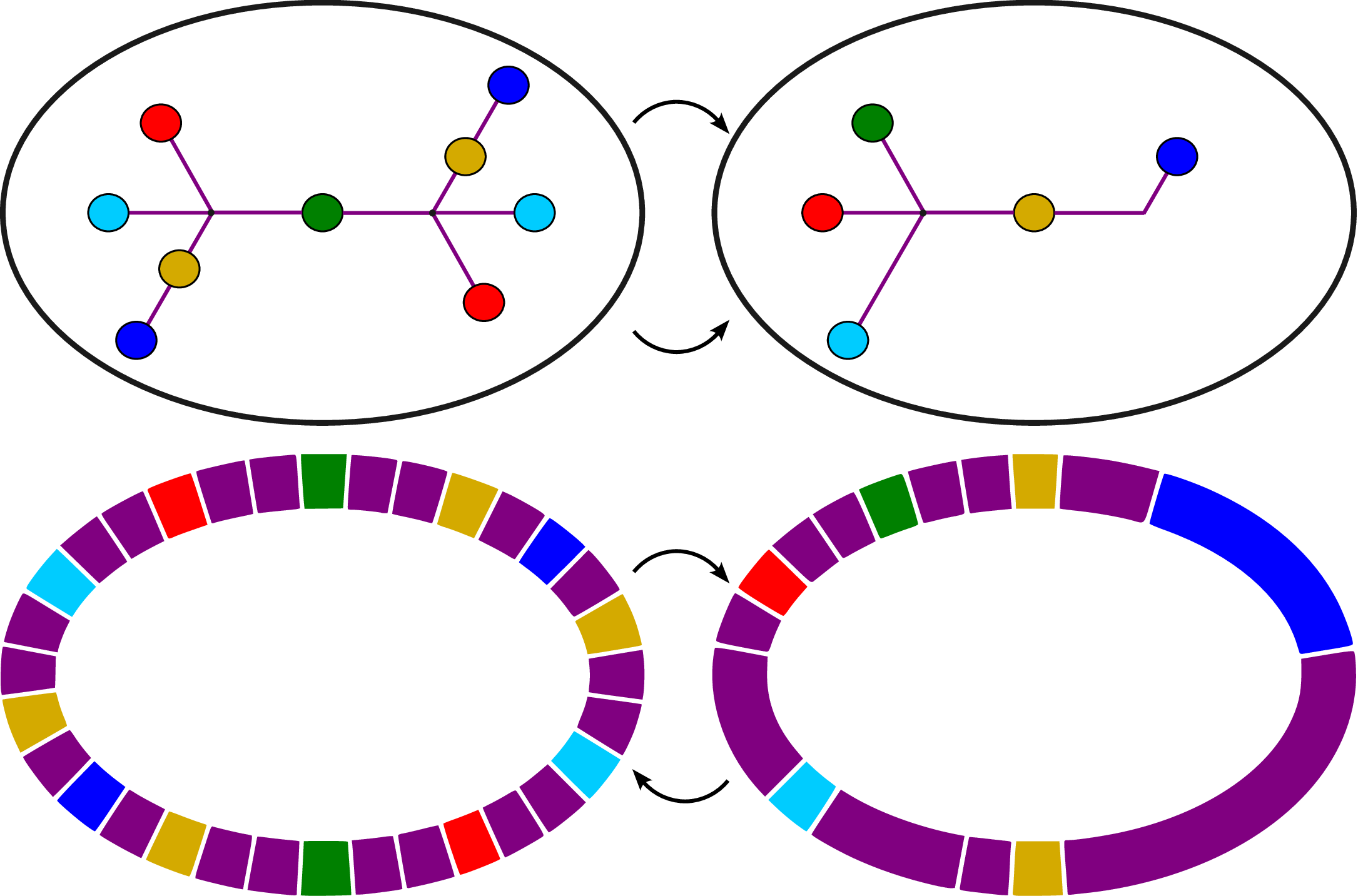
		\caption{The action of $f$ and $\iota$ from $\UU'\sm \TT'$ to $\UU\sm \TT$ and the induced actions of $f_*$ and $\iota^*$ on the ideal boundaries of $\TT'$ and $\TT$. Small filled Julia sets and intervals in $\TT'$ and $\ddi \TT'$ are mapped by $f$ to those of the same color in $\TT$ and $\ddi\TT$ respectively. Intervals in $\ddi \TT$ are lifted by $\iota$ to those in the same position.}
		\label{fig:dynamic-intervals}
	\end{center}
\end{figure}

\begin{prop}\label{prop:pullback description}
    If $I\subset \dcara \TT$ is either a multi-interval in $\II$ or the closure of a multi-interval in $\II$, then $\iota^*I$ is the union of $\iota^{-1}(I)$ and the Caratheodory boundaries of the components of $\TT'\sm \TT$ that meet $\TT$ at $I$.
\end{prop}

\begin{proof}
	We have that $\iota^{-1}(I)\subset \iota^*I$ by definition.
	If $\XX$ is a component of $\TT'\sm \TT$ that meets $\TT$ at $I$, then there is a proper path in $\UU\sm \TT'$ that is subordinate to $\overline{\dcara \XX}$ and is also a proper path in $\UU\sm \TT$ subordinate to $I$; hence $\overline{\dcara\XX}\subset \iota^*I$ by definition. The union of $\iota^{-1}(I)$ and the Caratheodory boundaries of all such $\XX$ is an interval $I'\subset \dcara\TT'$. To see that $I'=\iota^*I$, we observe that if $\gamma$ is a proper path in $\UU\sm \TT'$ that is also a proper path in $\UU\sm \TT$, and $\gamma$ is subordinate to an interval strictly larger than $I'$, then $I$ is not subordinate to $I$.
\end{proof}
Recalling the definition of combinatorial action for $F_*$ from Appendix \ref{sec:dynamics}, using Proposition \ref{prop:pullback description} we observe:
\begin{prop}\label{prop:comb action}
	$F_*$ acts combinatorially on $\dcara \TT$.
\end{prop}

\begin{proof}
	It follows immediately from Proposition \ref{prop:pullback description} that the collection $\{\iota^*\BI: \BI\in \II\}$ is a partial ideal marking of $\UU'\sm \TT'$; indeed the intervals are all defined and pair-wise disjoint. Thus $\iota^*\II$ is defined. 
	Moreover, for any interval $\BI\in \II$, each component of $\iota^{-1}(\BI)$ is an interval in $f^*\II$ and the Carathéodory boundary of each component of $\TT'\sm \TT$ meeting $\TT$ at $X$ is a multi-interval in $\II$, so $\II'$ is a refinement of $\iota^*\II$. 
	Thus $F$ acts on $(\UU\sm \TT, \II)$. 
	Note that the end-intervals (in $f^*\II$) of any interval in $\iota^*\II$ are automatically contained in $\iota^{-1}(\dcara \TT)$, so the action is combinatorial.
\end{proof}

We denote $\II' = f^*\II$ and $\tilde \II = \iota^*\II$. It also immediately follows from Proposition \ref{prop:pullback description} that the action of $F_*$ on $\II$ respects the dynamics of $F$ on $\TT$. More precisely, we have that $F_*\BL_j= \BL_{j+1}$ for $j\neq 0$, $F_*\BL_0 = \dcara \TT$, and if $\BJ$ is a side of some component of $K$ of $\KK$, then the end-intervals of  $F_*\BJ$ are sides of $F(K)$.
Recalling the definition of $F_*$ acting properly from Appendix \ref{sec:dynamics},
as $F$ is parabolically bounded we can conclude that $F_*$ acts properly on intervals outside a bounded neighborhood of $\BL_0$:

\begin{prop}\label{prop:improper intervals}
	$F_*$ is proper on $\BL_{[\bound, q-\bound-2]}$.
\end{prop}

\begin{proof}
	As $F$ is parabolically bounded, the components of $F^{-1}(\KK)$ in $\LL_{[\bound, q-\bound-2]}$ are all components of $\KK$. Thus $$\TT' \cap \LL_{[\bound, q-\bound-2]}= \TT\cap \LL_{[\bound, q-\bound-2]};$$
	the proof then immediately follows from Proposition \ref{prop:pullback description}.
\end{proof}

Proposition \ref{prop:improper intervals} implies that for most intervals $\BI\in \II$, $\ell(F_*\overline\BI)=1$. 
In fact, 
there is a unique multi-interval of length at most two, which we call the \emph{stretching multi-interval} $I_\str$, that contains the intervals that push forward to long multi-intervals:

\begin{prop}\label{prop:stretching interval}
	There is a multi-interval $I_\str$ in $\II$ with  $\BL_{(\bound, q-\bound)}\subset F_*I_\str$ and $\ell(I_\str)\leq 2$. Moreover, there is some universal $N$ so that 
	$\ell(F_* \BI)\leq N$ for any  interval $\BI\not\subset I_\str$ in $\II$.
\end{prop}

\begin{proof}
	Let $\LL_{[\bound, q-\bound-2]}'$ be the subset of $F^{-1}(\LL_{(\bound, q-\bound)})$ contained in $\FL_0$, and let
	$\BL_{[\bound, q-\bound-2]}'$ 
	be the minimal multi-interval in $\II'$ containing $\dcara \LL_{[\bound, q-\bound-2]}'$; thus $f_*\BL_{[\bound, q-\bound-2]}'=\BL_{(\bound, q-\bound)}.$ The components of $\TT'\sm \TT$ all meet $\TT$ at a single point $x$, so $\BL_{[\bound, q-\bound-2]}'\subset \Int\iota^*x$. If $x$ is contained in an interval $\BI\in \II$, then we set $\tilde \BI_\str = \iota^*\BI\in \tilde \II$. 
	If instead $x$ is an endpoint of $\II$, then we set $\tilde \BI_\str = \Int\iota^*x \in \tilde \II$. In both cases, we have 
	$$f_*\tilde \BI_\str\supset f_*\Int \iota^*x \supset f_*\BL_{[\bound, q-\bound-2]}' = \BL_{(\bound, q-\bound)}.$$
	We define $I_\str$ to be the minimal multi-interval satisfying $\tilde \BI_\str\subset \iota^*I_\str$.

	If $\BI\in \II$ lies in $\BL_j$ for some $j\neq 0$, then 
	$$\ell(F_* \BI)\leq \ell(\BL_{j+1})\leq 6\bound.$$
	If  instead $\BI\subset \BL_0\sm I_\str$, then $F_* \BI$ avoids $\BL_{(\bound, q-\bound)}$.  Hence 
	$$\ell(F_* \BI)\leq (2\bound +1)6\bound,$$
	which completes the proof.
\end{proof}

We now turn our attention to the proof of Lemma \ref{lem:bounded visits}. The key observation we need is that the length of any interval in $\II$ grows after a bounded number of {returns to $\BL_0$}:

\begin{prop}\label{prop:combinatorial growth}
	For any multi-interval $I$ and $n\geq 1$, if there are $6\bound+\ell(I)$  distinct integers $0\leq m \leq n$ such that $F_*^mI$ intersects $\BL_0$, then $\ell(F_*^nI)> \ell(I).$
\end{prop}

\begin{proof}
	For all $0\leq m \leq n$ we set $I_m = F_*^mI$.
	We assume for the sake of contradiction that $\ell(I_n)= \ell(I)$, so $\ell(I_m)= \ell(I)=k$ for all $m$.
	Let $(m_j)_{j=1}^{6\bound+k}$ be an increasing sequence of integers in $[0, n]$ such that $I_{m_j}$ intersects $\BL_0$ for all $j$.
	As the number of multi-intervals in $\II$ with length $\ell(I)$ that intersect $\BL_0$ is bounded above by $6\bound+k $, 
	it follows that there is some $1\leq s < r\leq n$ such that $I_{s}= I_{r}.$ This contradicts the fact that $\lambda_F> 1$.
\end{proof}

\begin{proof}[Proof of Lemma  \ref{lem:bounded visits}]
	Setting $N = 216\bound^2$, we assume that there is a sequence $(A_n)_{n=1}^N$ such that $A_n$ is a critical visit of $A_{n-1}$ for all $1< n \leq N$. Let $I_n$ 
	be the  crest of $A_n$. 
	As $A_{n}$ is a critical visit of $A_{n-1}$, and hence short, 
	it follows that $I_n$ intersects $\BL_0$, $\ell(I_n)< 12\bound$,  and there is an integer $m_n\geq 1$ such that $I_n\supset F_*^{m_n}I_{n-1}$. 
	For any $0\leq n \leq N-18\bound$,
	Proposition \ref{prop:combinatorial growth} implies that $\ell(I_{n+18\bound}) >\ell(I_n)$. Hence
	$\ell(I_{N})> 12\bound$, which is a contradiction.
\end{proof}

\subsection{Breakable waves}\label{sec:breaking}

In this section we prove Proposition \ref{prop:mostly short}. Recalling from Appendix \ref{app:dynamics} the definition of ``breaking" a wave, we will first show that breakable waves are terminal, restricting our analysis to unbreakable waves. For long unbreakable waves we have three cases: waves over the critical limb, waves in the translation region, and waves connecting the critical limb to the translation region. 
In these three cases, pushing forward yields respectively either self-intersection, which automatically implies that the width is small, an intersection with a shift of the wave, which implies the width is small by the Shift Lemma, or a lamination connecting components of $\KK$, which implies that the width is bounded by the renormalization width; 
in each case we can conclude that the wave cannot be very wide.

Let us now be more precise. 
We recall from Appendix \ref{app:dynamics} that a path $\gamma$ in $(\UU\sm \TT, \II)$ \textit{breaks} under $F$ if $F_*\gamma$ is not a single path. 
For any wave $A$ in $\UU\sm \TT$ and $n\geq 0$, we define the \textit{$n$-unbreakable part} \index{waves!unbreakable part} $A^{n\mathhyphen\nbr}\subset A$ to be the subwave consisting of leaves
$\alpha$ such that $F_*^m\alpha$ does not break under $F$ for all $0\leq m < n$. We define the \textit{$n$-breakable part} \index{waves!breakable part} of $A$ to be $A\sm A^{n\mathhyphen\nbr}$. 
For any $0< \lambda < 1$, we will say that $A$ is  \textit{$(n, \lambda)$-unbreakable} if $\WW(A^{n\mathhyphen\nbr})\geq \lambda\WW(A)$. 

We split our proof of Proposition \ref{prop:mostly short} into two parts. The first part is showing that controllable waves are either unbreakable or terminal, and the second is showing that unbreakable waves are short. More precisely, we have the following two propositions:
\begin{prop}\label{prop:small breaking}
	For any $n$, $k$, and $\lambda$  there exists $\eta$ and $\delta$ such that: if $A$ is a $[n, \lambda]$-wide wave and $A^{n\mathhyphen\nbr}$ is $k$-controllable, then $A$ is  $(n, \lambda)$-unbreakable or $(\delta, \eta)$-terminal. Moreover, $\eta$ does not depend on  $\lambda$, and $\delta$ does not depend on $k$.
\end{prop}

\begin{prop}\label{prop:nonbreaking implies short}
	There is a universal $N$ such that for any $\lambda$ there exists $\lambda'$ satisfying: for all $k$, any $[k, \lambda]$-wide $(N, \lambda')$-unbreakable $k$-controllable wave is $\lambda$-short.
\end{prop}

Together, Propositions \ref{prop:small breaking} and \ref{prop:nonbreaking implies short} imply Proposition  \ref{prop:mostly short}:
\begin{proof}[Proof of Proposition  \ref{prop:mostly short}]
	Let $N$ be the universal constant in Proposition \ref{prop:nonbreaking implies short}, and let $A$ be a $[k, \lambda]$-wide $k$-controllable wave.

	We set $B= A^{N\mathhyphen\br}\cup A^{\sh}$. 	
	For any $\lambda_1$, if $\WW(A^{N\mathhyphen\nbr})< (1-\lambda_1)\WW(A)$ then 
	\begin{equation}\label{eq:mostly short 1}
		\WW(B)\geq \WW(A^{N\mathhyphen\br})\geq \lambda_1\WW(A).
	\end{equation}
	Choosing $\lambda_1\gg\lambda_2$ for some $\lambda_2$, if $\WW(A^{N\mathhyphen\nbr})\geq (1-\lambda_1)\WW(A)$, then
	Proposition \ref{prop:nonbreaking implies short} implies that $A^{N\mathhyphen\nbr}$ is $\lambda_2$-short, so 
	\begin{equation}\label{eq:mostly short 2}
		\WW(B)\geq \WW(A^{N\mathhyphen\br})+ \lambda_2\WW(A^{N\mathhyphen\nbr})\geq \lambda_2\WW(A).
	\end{equation}
	Combining \eqref{eq:mostly short 1} and \eqref{eq:mostly short 2}, we are guaranteed to have $\WW(B)\geq \lambda_2\WW(A)$.
	
	As $B^{N\mathhyphen\nbr}$ is short, and hence  $12\bound$-controllable, for any $\lambda_3$ it follows from Proposition \ref{prop:small breaking} that $B$ is either $(N, \lambda_3)$-unbreakable or $(\delta_3, \eta_3)$-terminal for some $\delta_3$ and universal $\eta_3$.
	Choosing $\lambda_3\gg\lambda_4$ for some $\lambda_4$, Proposition \ref{prop:nonbreaking implies short} implies that if $B$ is $(N, \lambda_3)$-unbreakable, then it is $\lambda_4$-short. 
	
	By the above,  either $B$ is $(\delta_3, \eta_3)$-terminal or 
	$$\WW(A^\sh) = \WW(B^\sh)\geq \lambda_4\WW(B)\geq \lambda_2\lambda_4\WW(A).$$
	Choosing $\lambda_2\lambda_4\gg_{\delta_3}\lambda$, it follows that $A$ is either $\lambda$-short or  $(\delta, \eta_3)$-terminal for some $\delta$.
\end{proof}

Our first step in proving  Propositions \ref{prop:small breaking} and \ref{prop:nonbreaking implies short} is observing that controllable waves push forward to controllable waves:

\begin{lem}\label{lem:pushing controllable}
	For any $n$ and $k$ there exists $\eta$ such that: if $A$ is a wave and $A^{n\mathhyphen\nbr}$ is  $k$-controllable, then any wave in $F_*^nA$ is  $\eta$-controllable.
\end{lem}

\begin{proof}
	First we consider the case where $n = 1$.
	Let $B\subset F_*A$ be a wave.
	Let us assume that there is a leaf $\beta\in B$  subordinate to $\BL_{(\bound, q-\bound)}$. Let $\tilde{\alpha}$ be the leaf of $\iota^*A$ satisfying $f_*\tilde\alpha = \beta$, so $\tilde\alpha$ is subordinate to a component of $f^{-1}(\BL_{(\bound, q-\bound)})$.
	It follows from the proof of 
	Proposition \ref{prop:stretching interval} that one such component is contained in an interval in $\tilde \II$; recall that no paths in $\iota^*A$ are subordinate to an interval in $\tilde \II$ by definition. Thus $\tilde{\alpha}$ must be subordinate to $\iota^*\BL_{[\bound, q-\bound-2]}$. 
	As $F_*$ is proper on $\BL_{[\bound, q-\bound-2]}$ by Proposition \ref{prop:improper intervals}, it follows that there is a path $\alpha$ subordinate to $\BL_{[\bound, q-\bound-2]}$ so that $\tilde \alpha=\iota^*\alpha$. Hence $\beta = F_*\alpha$ and $\alpha\in A^\nbr$. 
	
	Let $B'\subset B$ be the subwave consisting of the leaves subordinate to $\BL_{(\bound, q-\bound)}$. By the above, there is a subwave $A'\subset A^\nbr$ subordinate to $\BL_{{[\bound, q-\bound-2]}}$ such that $B' = F_*A'$. As $A^\nbr$ is $k$-controllable, $A'$ is the union of $k$ localized waves. As $F_*$ is proper on $\BL_{[\bound, q-\bound-2]}$ it follows that $B'$ is also the union of $k$  localized waves. As any leaf of $B\sm B'$ must either have one endpoint in $\BL_{[-\bound, \bound]}$ or be a wave over $\BL_0$, it follows that $B$ is the union of a wave over $\BL_0$ and at most $\eta = k+4\bound(2\bound+1)$ many localized waves. This completes the proof when $n = 1$.
	
	Now let us suppose that $n > 1$ and that the Proposition holds for smaller values of $n$. Fixing some wave $B\subset F_*^nA$, there is a wave $X\subset F_*^{n-1}A$ satisfying $F_*X = B$. For any $\alpha\in A^{(n-1)\mathhyphen\br}$, if $F_*^{(n-1)}\alpha\in X^\nbr$ then $\alpha\in A^{n\mathhyphen\nbr}$ by definition. 
	Thus 
	$$X^\nbr\subset F_*^{n-1}(A^{n\mathhyphen\nbr}\cup A^{(n-1)\mathhyphen\br}).$$
	As $A^{n\mathhyphen\nbr}$ is  $k$ controllable, it follows from the inductive hypothesis that $X^\nbr$ is $\eta'$ controllable for some $\eta'$ depending only on $n$ and $k$. As $B\subset F_*X$, it then follows from the inductive hypothesis that $B$ is  $\eta$ controllable for some $\eta$ depending only on $\eta'$.
\end{proof}

With Lemma \ref{lem:pushing controllable}, we can now prove Proposition \ref{prop:small breaking}:

\begin{proof}[Proof of Proposition \ref{prop:small breaking}]
	We will only consider the case where $n = 1$; the general case follows by an easy induction. 
	Let $A$ be a $[\lambda]$-wide wave such that  $A^\nbr$ is $k$-controllable, and let 
	$\eta$ depend on $k$ and $n =1$ as in Lemma \ref{lem:pushing controllable}.
	Setting $\eps= 1-\lambda$, we suppose that $\WW(A^\br)\geq \eps\WW(A)$ and choose some $\delta$ to be specified later. 
	
	Let $I$ be the crest of $A$.  
	As $F_*$ is end-proper on $I$, the restriction in $\UU\sm \TT'$ of any leaf of $A^\br$ must  contain either a path connecting the two components of $\UU\sm \TT'$, a wave over $\iota^*I$, or two waves, one over each endpoint of $\iota^*I$ (see Figure \ref{fig:breaking}). We will consider these three possibilities separately. 
	
	\begin{figure}
		\begin{center}
			\def\svgwidth{5in}
			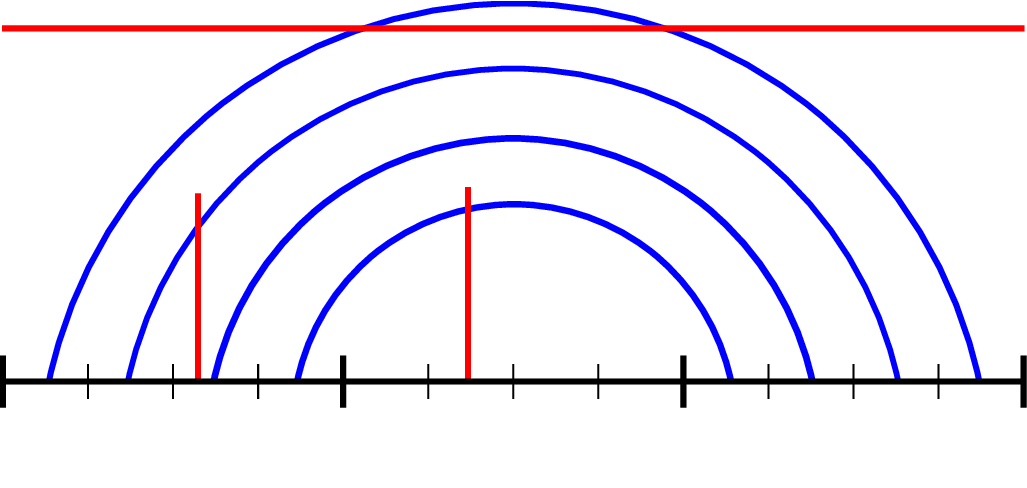
			\caption{A wave $A$ in $\UU\sm \TT$ shown in blue with crest $I$ and ends $J_1$, $J_2$. For simplicity we assume that the immersion $\UU'\to \UU$ is an inclusion; the boundary of $\UU'\sm \TT'$ in $\UU\sm \TT$ is shown in red. The leaves $\gamma_1, \gamma_2, \gamma_3,$ and $\gamma_4$ of $A$ belong to $A^\ver$, $A^\outbr$, $A^\nbr$, and $A^\inbr$ respectively.}
			\label{fig:breaking}
		\end{center}
	\end{figure}

	Let $A^\ver\subset A^\br$ be the subwave consisting of leaves whose restriction by $\iota$ contains a path connecting the two components of $\ddi (\UU'\sm \TT')$.  
	As 
	$$\WW(A^\ver)\leq \WW(\UU'\sm \TT')\leq 2\WW(\UU\sm \TT),$$
	if $\WW(A^\ver)\geq 2\delta \WW(A)$ then $A$ is $(\delta, \eta)$-terminal. 
	So we can  assume that  
	\begin{equation}\label{eq:vert}
		\WW(A^\ver)< 2\delta \WW(A).
	\end{equation}

	Let $A^\inbr\subset A^\br$ be the subwave consisting of leaves whose restriction by $\iota$  contains two waves, one over each endpoint of $\iota^*I$. 
	Thus there is a wave $B^\inbr\subset \iota^*A^\br$ over one endpoint of 
	$\iota^*I$ with 
	\begin{equation}\label{eq:inbr}
		\WW(B^{\inbr})\geq 2\WW(A^\inbr).
	\end{equation}

	We set $A^{\out} = A^\br\sm (A^\inbr\cup A^\ver)$, so the restriction by $\iota$ of any leaf of $A^{\out}$ contains a wave over $\iota^*I$. 
	Let $\BI_1, \dots, \BI_N$ be the intervals in $\II$ on which $F_*$ is improper; it follows from Proposition \ref{prop:improper intervals} that $N\leq 4\bound(2\bound+2)$.
	As $A^{\out}$ is disjoint from $A^\nbr$ and $A^\ver$, the restriction by $\iota$ of any leaf of $A^{\out}$ must also contain a wave with one end in some $\iota^*\BI_n$. Thus there are waves $B^{\out}$ and $C^{\out}$ in $\iota^*A^{\out}$ such that $B^{\out}$ is a wave over $\iota^*I$ and
	\begin{equation}\label{eq:out 1}
		\WW(A^{\out})\leq \WW(B^{\out})\oplus \left(2N\WW(C^{\out})\right).
	\end{equation}
	If $\WW(C^{\out})\geq 3\WW(A)$, 
	then it follows from the Lemmas \ref{lem:pushing controllable} and \ref{lem:wave} that there is a $\eta$-controllable wave with width at least $2\WW(A)$, hence $A$ is $(\delta, \eta)$-terminal. Thus we may assume that $\WW(C^{\out})<  3\WW(A)$. 
	It therefore follows from \eqref{eq:out 1} that 
	\begin{equation}\label{eq:out}
		\WW(B^{\out})\geq  \frac{\WW(A^{\out})}{1- \frac{\WW(A^{\out})}{2N\WW(C^{\out})}}\geq \WW(A^{\out})+ \frac{\WW(A^{\out})^2}{6N\WW(A)}.
	\end{equation}

	We set   $B^\nbr= \iota^*A^\nbr$, so $B^\nbr$ is a wave over $\iota^*I$ with $\WW(B^\nbr)\geq \WW(A^\nbr).$
	It follows from the constructions that $B^\nbr$, $B^\inbr$, and $B^{\out}$ are all disjoint and together form a wave; we set $B = B^\nbr\cup B^\inbr\cup B^{\out}$.
	As $A^\br = A^\ver\cup A^\inbr\cup A^{\out}$ and $\WW(A^\br)\geq \eps\WW(A)$,  \eqref{eq:vert} implies 
	\begin{equation}\label{eq:br minus ver}
		\WW(A^\inbr)+ \WW(A^{\out})\geq (\eps-2\delta)\WW(A).
	\end{equation}
	Combining 
	\eqref{eq:inbr}, \eqref{eq:out}, \eqref{eq:vert}, and \eqref{eq:br minus ver} yields
	\begin{align*}
		\WW(B)&\geq \WW(A)- \WW(A^\ver)+ \WW(A^\inbr)+ \frac{\WW(A^{\out})^2}{6N\WW(A)}\\
		& \geq \left(1- 2\delta+\frac{\WW(A^\inbr)}{\WW(A)}+ \frac{1}{6N}\left(\frac{\WW(A^{\out})}{\WW(A)}\right)^2\right)\WW(A)\\
		&\geq \left(1- 2\delta+ \frac{(\eps-2\delta)^2}{6N}\right)\WW(A).
	\end{align*}
	Choosing $\delta\ll\epsilon$, 
	Lemmas \ref{lem:pushing controllable} and \ref{lem:wave} then imply  
	that there is a $\eta$-controllable wave $C\subset f_*B$ with $\WW(B)\geq (1+\delta)\WW(A)$, so $A$ is $(\delta, \eta)$-terminal.

\end{proof}

Proposition \ref{prop:nonbreaking implies short} will be an immediate consequence of the following observation:

\begin{prop}\label{prop:mostly short in L}
	There is some universal $N$ such that: for any $0<\lambda< 1$, if $A$ is a $[k, \lambda]$-wide $k$-controllable wave, then $$\lambda\WW(A)\leq \WW(A^\sh)+\WW(A^{N\mathhyphen\br}).$$
\end{prop}

\begin{proof}[Proof of Proposition \ref{prop:nonbreaking implies short}]
	We set $\lambda' = (1+\lambda)/2$ and let $N$ be as in Proposition \ref{prop:mostly short in L}.
	If $A$ is a $[k, \lambda]$-wide $(N, \lambda')$-unbreakable $k$-controllable wave, then  Proposition \ref{prop:mostly short in L} implies that
	\begin{align*}
		\lambda'\WW(A)\leq \WW(A^\sh)+ \WW(A^{N\mathhyphen\br})\leq \WW(A^\sh)+ (1-\lambda')\WW(A),
	\end{align*}
	so $\WW(A^\sh)\geq \lambda\WW(A).$
\end{proof}

We will prove Proposition \ref{prop:mostly short in L} by decomposing waves according to their position relative to the critical limb.
First we observe that unbreakable  waves over $\BL_0$  have bounded width:
\begin{prop}\label{prop:over critical}
	If $A$ is a wave over $\BL_0$, then 
	$$\WW(A^{\nbr})= O(1).$$
\end{prop}
\begin{proof}
	It follows from Lemma \ref{lem:wave} and Proposition \ref{prop:nonbreaking} that there is a wave $B\subset F_* A^{\nbr}$ over $F_*\BL_0$ with $$\WW(A^{\nbr})\leq \WW(B)+O(1).$$
	As $F_*\BL_0= \dcara \TT$, $B$ must be empty.
\end{proof}

Next we observe that unbreakable waves away from $\BL_0$ have bounded width:
\begin{prop}\label{prop:away from critical}
	If $A$ is a localized wave in $\UU\sm \TT$ subordinate to $\dcara\TT\sm \BL_0$, then 
	\begin{equation*}
		\WW(A^\nbr)\leq \WW(A^\sh)+ O(1).
	\end{equation*}
\end{prop}
\begin{proof}
	As $\dcara \TT \sm \BL_0$ shifts under $F_*$, the proposition follows immediately from Lemma \ref{lem:shift}.
\end{proof}

Finally, we observe that that unbreakable waves ending at $\BL_0$ have bounded width:
\begin{prop}\label{prop:into the critical limb}
	There is some universal $N$ such that if $A$ is a wave in $\UU\sm \TT$  ending at $\BL_0$, then $$\WW(A^{N\mathhyphen\nbr})\leq \WW(A^\sh)+ O(\WW(\RR F)).$$
\end{prop}

\begin{proof}
    Let us first outline the argument.
    For a wave $A$ ending at $\BL_0$, 
    we can decompose its pushforward into a wave also ending at $\BL_0$ and a wave that does not; the latter case is controlled by the previous two propositions. Thus we need only consider a wave that pushes forward repeatedly to waves ending at $\BL_0$. We then consider how such a wave intersects $\TT^n$ for some large $n$. If the wave is unbreakable then it must be ``aligned" with $\TT^n$ near $\beta_F$, and so passes between two components of $\KK^n$; this gives a bound on the width of the wave in terms of $\WW(\RR F)$.

	We set $N = \bound+3.$
	Setting $A_0= A^{N\mathhyphen\nbr}\sm A^\sh$, for all $0< n < N$ we inductively define $A_n\subset A_{n-1}$ to be the maximal subwave so that $F_*^nA_n$ is a wave ending at $\BL_0$. 
	Let us first observe:
    \begin{lemma}
        We have
        \begin{equation}\label{eq:into critical bound}
		\WW(A_{n}\sm A_{n+1})= O(1)
	\end{equation}
	for all $0\leq n < N-1$.
    \end{lemma}

    \begin{proof}
        We fix some $n$ and set $S_n = F_*^nA_n$, so $S_n$ is a wave ending at $\BL_0$ and $S_n = S_n^{(N-n)\mathhyphen\nbr}.$
	It follows from Lemma \ref{lem:wave} that there is a wave $X_n\subset F_*S_n$ with no short leaves and satisfying
	\begin{equation}\label{eq:into stationary X_n}
		\WW(S_{n})\leq \WW(X_{n})+O(1).
	\end{equation}
	We can decompose $X_n$ into waves $X_{n+1}$, $Y_{n+1}$, and $Z_{n+1}$ so that 
	$X_{n+1}$ is localized to $\BL_0$, $Y_{n+1}$ is a wave over $\BL_0$, and $Z_{n+1}$ is subordinate to $\dcara\TT\sm\BL_0$. As $S_{n+1} = S_{n+1}^{(N-n-1)\mathhyphen\nbr}$ and $N> n+1$, it follows from Propositions \ref{prop:over critical} and \ref{prop:away from critical} that 
	\begin{equation}\label{eq:into stationary extraneous}
		\WW(Y_{n+1})+ \WW(Z_{n+1})= O(1).
	\end{equation}
	Let $A_{n+1}'\subset A_n$ be the subwave satisfying $F_*^{n+1}A_{n+1} = X_{n+1}$, 
	so 
	\begin{equation}\label{eq:into stationary A'}
		\WW(A_{n+1})\geq \WW(A_{n+1}'),
	\end{equation}
	and \eqref{eq:into stationary X_n} implies
	\begin{equation}\label{eq:into stationary S_n}
		\WW(A_n\sm A_{n+1}')\leq \WW(S_{n+1}	
		\sm X_{n+1})+O(1)
	\end{equation}
	Combining \eqref{eq:into stationary extraneous}, \eqref{eq:into stationary A'}, and \eqref{eq:into stationary S_n} yields \eqref{eq:into critical bound}.
    \end{proof}

	It follows from \eqref{eq:into critical bound}  that 
    \begin{equation}\label{eq:stationary A_N}
        \WW(A^{N\mathhyphen\nbr})\leq \WW(A^\sh)+\WW(A_{N-1})+O(\bound).
    \end{equation}
    We next observe:
    \begin{lemma}\label{lem:stationary wave alignment}
        There is a component $K$ of $\KK^{N-1}$ such that the restriction to $\UU^{N-1}\sm \KK^{N-1}$ of each leaf of $A_{N-1}$ contains a path connecting $K$ to $\KK^{N-1}\sm K$.
        Moreover, $f^{N-1}$ has degree one on $K$. 
    \end{lemma}

    \begin{proof}
        Setting $\alpha^0= \alpha_F$, $K_0 = K_0^0$, and  $\LL_0^0 = \LL_0$, for all $n> 0$ we inductively define $K_0^n$, $\alpha^n$, and $\LL_0^n$ to be the components of $F^{-1}(K_0^{n-1})$, $F^{-1}(\alpha^{n-1})$, and $F^{-1}(\LL_0^{n-1})$ respectively contained in $\FL_0$. 
	For any $n>0$, if $\LL_0^n$ contains a component $K$ of $\KK$, then $\LL_0^{n-1}$ contains $F(K)$, which is also a component of $\KK$. As $\FL_0$ contains at most $\bound$ many components of $\KK$, it therefore follows that $\LL_{0}^{\bound}$ is disjoint from $\TT$.

    For any $n$, every leaf of $A_n$ must have one endpoint in $\LL_0^0$ and restrict to a path in $\UU^n\sm F^{-n}(\TT)$ that has one endpoint in $\LL_0^n$. 
	In particular, for $n\geq \bound$ any leaf of $A_n$ must start at $\LL_0^0$ and pass through $\LL_0^n$ before terminating somewhere on $\dcara\TT\sm \BL_0$.
	Note that $K_0^0$ separates $\alpha_0\in \overline{\LL_0^0}$ from $\LL_0^1$ in $\FK$, so $K_0^n$ separates $\alpha_n\in \overline {\LL_0^n}$ from $\LL_0^{n+1}$ for all $n$. In particular, $K_0^n$ separates $\TT$ from $\LL_0^{n+1}$ for all $n\geq \bound$. 
	Moreover, any leaf $\alpha$ of $A_{n+1}$ must pass through $K_0^n$ when $n\geq \bound$; otherwise the portions of $\alpha$ connecting $\BL_0$ to $\LL_0^{n+1}$ and connecting $\LL_0^{n+1}$ to $\dcara\TT\sm \BL_0$ would both pushforward to waves, so $\alpha$ would break under $F_*^{n+1}$. 

    As $N-1 = \bound+2$, the above implies that every leaf of $A_{N}$ passes through $K_0^{\bound}$ and $K_0^{\bound+1}$; taking $K = K_0^\bound\subset \KK^{N-1}$ completes the proof.
    \end{proof}

	It follows from Lemma \ref{lem:stationary wave alignment} that the restriction of $A_{N-1}$ to $\UU^{N-1}\sm \KK^{N-1}$ contains a lamination $B$ connecting a component $K$ of $\KK^{N-1}$ to $\KK^n\sm K$ and satisfying 
	\begin{equation}\label{eq:into stationary B}
		\WW(A_{N-1})\leq \WW(B).
	\end{equation}
	As $F^{N-1}$ has degree $1$ on $K$,  it follows from Lemma \ref{lem:push forward laminations} that there is a lamination $C$ in $\UU\sm \KK$ connecting $F^{N-1}(K)$ to $\KK$ with 
	\begin{equation}\label{eq:into stationary C}
		\WW(B)\leq \WW(C)+O(1).
	\end{equation}
	Moreover $\WW(C)\leq 2\WW(\RR F)$; 
    so combining \eqref{eq:into stationary B} and \eqref{eq:into stationary C} yields 
    \begin{equation}\label{eq:into stationary goal}
		\WW(A_{N-1})= O(\WW(\RR F)).
	\end{equation}
    Combining \eqref{eq:stationary A_N} and \eqref{eq:into stationary goal} completes the proof.
\end{proof}

\begin{proof}[Proof of Proposition \ref{prop:mostly short in L}]
	Let $N$ be the constant in Proposition \ref{prop:into the critical limb}. 
	As $A$ is $k$-controllable, we can write it as the union $$A= A^\sh \cup X\cup Y \cup Z_1\cup \cdots \cup Z_k,$$
	where $X$ is a wave over $\BL_0$, $Y$ is a wave ending at $\BL_0$, and each $Z_j$ is a localized wave subordinate to $\dcara\TT \sm \BL_0$; moreover $X, Y$, and each $Z_j$ have no short leaves. 
	Propositions \ref{prop:over critical}, \ref{prop:away from critical}, and \ref{prop:into the critical limb} then imply that
	\begin{align*}
		\WW(A)&\leq \WW(A^{N\mathhyphen\br})+\WW(A^\sh)+ \WW(X^{\nbr})+ \WW(Y^{N\mathhyphen\nbr})+ \sum_{j=1}^k \WW(Z^\nbr)\\
		&\leq \WW(A^{N\mathhyphen\br})+\WW(A^\sh)+O(k)+ O(\WW(\RR f)).
	\end{align*}
	Thus for any $\lambda$, 
	$$\lambda \WW(A)\leq  \WW(A^{N\mathhyphen\br})+\WW(A^\sh)$$
	when $A$ is $[k, \lambda]$-wide.
\end{proof}

We end this section by observing  two consequences our Proposition \ref{prop:mostly short}. The first is that any terminal wave is $(\delta, \eta)$-terminal for some universal $\eta$:

\begin{cor}\label{cor:universal terminal}
	For any $\delta'$ there exists $\delta$ such that: for all $k'$, any $[\delta', k']$-wide $(\delta', k')$-terminal wave is $(\delta, \eta)$-terminal for some universal $\eta$. 
\end{cor}

\begin{proof}
	Let $A$ be a $[\delta', k']$-wide $(\delta', k')$-terminal wave.
	If $\WW(\UU\sm \TT)\geq \delta'\WW(A)$, then $A$ is $(\delta, \eta)$-controllable for any $\delta\leq \delta'$ and any $\eta$. For any $\lambda_0$, if there is a $k'$-controllable wave $B$ with $\WW(B)\geq (1+\delta')\WW(A)$, then  Proposition \ref{prop:mostly short} implies that $B$ is either $\lambda_0$-short or $(\delta_0, \eta_0)$-terminal for some $\delta_0$ and universal $\eta_0$. 
	Choosing $\lambda_0 \gg_{\delta'} 0$, if $B$ is $\lambda_0$-short, then as $B^\sh$ is $12\bound$-controllable and  $$\WW(B^\sh)\geq \lambda_0(1+\delta')\WW(A),$$
	it follows that $A$ is $(\delta, 12\bound)$-terminal for some $\delta$.
	As $\WW(B)\geq \WW(A)$, if $B$ is $(\delta_0, \eta_0)$-terminal then so is $A$.
\end{proof}

For any  $0< \lambda< 1$, $n\geq 1$, and wave $A$ in $\UU\sm \TT$  with  crest $I$, we will say that a short wave $B$ over $F_*^nI$ is a \textit{short $(n, \lambda)$-pushforward} \index{waves!short pushforward} if $\WW(B)\geq \lambda\WW(A)$.
Our second consequence of Proposition \ref{prop:mostly short} is that non-terminal short waves have short pushforwards:

\begin{cor}\label{cor:short pushforward}
	For any $n$ and $\lambda$ there exists $\delta$ such that: any $[n, \lambda]$-wide short wave either has a short $(n, \lambda)$-pushforward or is $(\delta, \eta)$-terminal for some universal $\eta$.
\end{cor}

\begin{proof}
	Let $A$ be a $[n, \lambda]$-wide short wave with crest $I$. 
	For any $\lambda_0$, Proposition \ref{prop:small breaking}  implies that 
	$A$ is either  $(n, \lambda_0)$-unbreakable or $(\delta_0, \eta_0)$-terminal for some $\delta_0$ and $\eta_0$.  Moreover, $\eta_0$ depends only on $n$.
	If $A$ is $(\delta_0, \eta_0)$-terminal, then Corollary \ref{cor:universal terminal} implies that $A$ is $(\delta, \eta)$-terminal for some $\delta$ and some universal $\eta$. Thus we may assume that $A$ is $(n, \lambda_0)$-unbreakable.
	
	For any $\lambda_1$,  Lemmas \ref{lem:pushing controllable} and \ref{lem:wave} imply that  there is a  $k_1$-controllable wave $B$ over $F_*^nI$ with  $$\WW(B)\geq \lambda_1\lambda_0\WW(A)$$
	for some $k_1$ depending only on $n$.
	For any $\lambda_2$, Proposition \ref{prop:mostly short} implies that $B$ is either $\lambda_2$-short
	or $(\delta_2, \eta_2)$-terminal for some $\eta_2$ depending only on $n$. 
	Choosing $\lambda_0 \lambda_1\lambda_2\geq \lambda$, if $B$ is $\lambda_2$-short then $B^\sh$ is a short $(n, \lambda)$-pushforward of $A$. 
	Choosing $\lambda_0\lambda_1\gg_{\delta_2} 0$, if  $B$ is $(\delta_2, \eta_2)$-terminal, then $A$ is $(\delta_2', \eta_2)$-terminal for some $\delta_2'$. 
	It then follows from Corollary \ref{cor:universal terminal} that $A$ is $(\delta, \eta)$-terminal for some $\delta$ and some universal $\eta$.	
\end{proof}

\subsection{Critical visits}\label{sec:visits}

We will now prove Proposition \ref{prop:critical visit}; let us briefly outline the idea first.
Recall from Proposition \ref{prop:small breaking} that breakable waves are terminal, so we may restrict ourselves  to unbreakable waves. 
Pushing forward a wave by $F$ yields either a critical visit or a short wave in the translation region.
We can use the simple dynamics of $F$ in the translation region to show that, even for large iterates,  breakable waves are terminal; an unbreakable wave can then be pushed forward to a critical visit with small loss in width.

Let us now be more precise. We will show that non-terminal waves away from the critical limb eventually return:
\begin{prop}\label{prop:fast forward}
	For any $\lambda$ there exists $\delta$ such that: any short $[\lambda]$-wide wave whose crest avoids $\BL_0$ either has a critical $\lambda$-visit or is $(\delta, \eta)$-terminal for some universal $\eta$.
\end{prop}

Indeed, Proposition \ref{prop:critical visit} immediately follows from Proposition \ref{prop:fast forward}:

\begin{proof}[Proof of Proposition \ref{prop:critical visit}]
	Let $A$ be a $[\lambda]$-wide short wave with crest $I$. For any $\lambda_0$, Corollary \ref{cor:short pushforward} implies that either there is a short wave $B$ over $F_*I$ with $\WW(B)\geq \lambda_0\WW(A)$, or $A$ is $(\delta_0, \eta_0)$ for some $\delta_0$ and universal $\eta_0$. If $F_*I$ intersects $\BL_0$, then $B$ is a critical $\lambda_0$-visit of $A$. 
	If $F_*I$ avoids $\BL_0$, then for any $\lambda_1$ it follows from Proposition \ref{prop:fast forward} that  $B$ either has a critical $\lambda_1$-visit $C$ or is $(\delta_1, \eta_1)$-terminal for some $\delta_1$ and universal $\eta_1$. Note that $C$ is a critical $\lambda_0\lambda_1$-visit of $A$.
	
	Choosing $\lambda_0\gg_{\delta_1} 0$,  if $B$ is $(\delta_1, \eta_1)$-terminal then $A$ is $(\delta_1', \eta_1)$-terminal for some $\delta_1'$. 
	Choosing $\lambda_0\lambda_1\geq \lambda$,  $\delta< \min(\delta_0, \delta_1')$, and $\eta\geq \max(\eta_0, \eta_1)$, it follows from the above that $A$ either has a critical $\lambda$-visit or is $(\delta, \eta)$-terminal.
\end{proof}

For any short wide wave $A$ with crest $I$, if there is some $n\leq O(\bound)$ such that $F_*^nI$ intersects $\BL_0$, then we can use Corollary \ref{cor:short pushforward} to produce a critical visit. Thus to prove Proposition \ref{prop:fast forward}, we need to consider unbounded iterates of $F_*$. 
For any $n\geq 0$, we set $\UU^n = f^{-n}(\UU)$ and $\TT^n = f^{-n}(\TT)$ and define $\iota_n: \UU^n\sm \TT^n\to \UU\sm \TT$ to be the associated inclusion. We equip $\UU^n\sm \TT^n$ with the ideal marking $\II^n = (f^n)^*\II$;  the pair $(f^n, \iota_n)$ maps $\UU^n\sm \TT^n$ to $\UU\sm \TT.$
The action of the pair $(f^n, \iota_n)$ is closely related to $F_*^n$, indeed for any multi-interval $I$ in $\II$ we have $f_*^n\iota_n^*I = F_*^nI$. 
An identical argument to that of Proposition \ref{prop:comb action} implies that $(f^n, \iota_n)$ acts combinatorially on $(\UU\sm \TT, \II)$.

For the pair $(f^n, \iota_n)$, we have an analog of Lemma \ref{lem:pushing controllable} in the translation region:

\begin{lem}\label{lem:fast pushing controllable}
	For any $x< 0 <  y$ there exists $\eta$ such that: if $A$ is a wave in $\UU\sm \TT$ and $B\subset \iota_n^*A$ is a wave localized to $\iota_n^*\BL_{[x-n, y-n]}$ for some $y\leq n < x+q$, then any wave in $f_*^nB$ is $\eta$-controllable.
\end{lem}

\begin{proof}
	If $q\leq 2\bound +y-x$ then any wave in $\UU\sm \TT$ is $6\bound(2\bound+y-x)$-controllable, so we may assume that $q> 2\bound +y-x$.

	Let $D\subset f_*^nB$ be a wave, so there is some wave $C\subset f_*^{n-y}B$ satisfying $D\subset f_*^yC$. Thus $C$ is localized to $\iota_y^*\BL_{[x-y, 0]}.$
	Setting $I = \BL_{[-q+\bound+y, -\bound+x-1]}$ and $I' = \BL_{[-q+\bound, -\bound+x-y-1]}$, we suppose for the sake of contradiction that $D$ contains a leaf $\delta$ subordinate to $\iota_y^*I.$ Let $\gamma$ be the leaf of $C$ satisfying $f_*^y\gamma= \delta$, so $\gamma$ is subordinate to a component of $f^{-y}_*(I)$.  Proposition \ref{prop:stretching interval} implies that
	$\iota_y^*I'$ is the unique component of $f^{-y}_*I$ that is not contained in the restriction by $\iota_y^*$ of a single interval in $\II$, thus $\gamma$ is subordinate to $\iota_y^*I'$. But this contradicts the fact that $C$ is localized to  $\iota_y^*\BL_{[x-y, 0]}$, hence $D$ has no leaves subordinate to $I$. 
	Thus every leaf of $D$ either is a wave over $\BL_0$ or has an endpoint in $\BL_{[-\bound+x, \bound+y+1]}$, so $D$ is $6\bound(2\bound+y-x+2)$-controllable.
\end{proof}

We can now prove Proposition \ref{prop:fast forward}:

\begin{proof}[Proof of Proposition \ref{prop:fast forward}]
	Choosing some $\lambda_0$, we set $\eps = (1-\lambda_0)/4$ and fix an  integer $M\geq 2/\eps$.
	Let $A$ be a short $[\lambda]$-wide wave whose crest $I$ avoids $\BL_0$. 
	We can restrict to the case where $\bound< j < q-\bound - M-1$: if $0< j\leq \bound$ then we can use Corollary \ref{cor:short pushforward} to replace $A$ with a wave whose crest intersects $\BL_{\bound+1}$, and if $q-\bound - M-1\leq j < q$ then we can use Corollary \ref{cor:short pushforward}  to produce a critical visit of $A$.

	We set $N = q-\bound - j -2$ and $J = \BL_{[j, q+M-N]}$. 
	As in the proof of Proposition \ref{prop:small breaking}, we will decompose $A$ into subwaves based on their restriction by $\iota_N^*$. The main difference is that here we will keep track of their positions relative to $\iota_N^*J$. 
	It follows from Proposition  \ref{prop:improper intervals} that $\iota_N$ is proper on $\BL_{\bound, j+1}$ and that $f_*^N$ has degree $1$ on $\BL_{[1, q-N]}$ and degree $2^m$ on $\BL_{q+m-N}$ for all $1\leq m \leq M$. 
	Thus $f_*^N$ has degree $\leq 2^{M+1}$ on $\iota_N^*J$. 		
	As $A$ is short, $I\subset \BL_{[j-1, j+1]}$ and one end of $A$ is contained in $\BL_{[j, j+1]}$.

	Let $B^\ver\subset \iota_{N}^*A$ be the maximal sublamination consisting of leaves connecting 
	$\iota_N^*J$ to $\ddi\UU^{N}$.
	It follows from Lemma \ref{lem:push forward laminations} that there is a lamination $C^\ver\subset f_*^{N}A^\ver$ connecting  the two boundary components of $\UU\sm \TT$ and satisfying
	$$\WW(B^\ver)\leq 2^{M+1}(\WW(C^\ver)+2).$$
	Let $A^\ver\subset A$ be the subwave consisting of leaves whose restriction by $\iota_N^*$ contribute at least one path to $B^\ver$, 
	so $\WW(A^\ver)\leq \WW(B^\ver)$. If 
	$\WW(A^\ver)\geq \eps \WW(A)$ then it follows that 
	$$\WW(V)\geq \WW(C^\ver)\geq \WW(A^\ver)/2^{M+1}-2\geq (\eps/2^{M+2})\WW(A),$$
	hence $A$ is $(\eps /2^{M+2}, \eta)$-terminal 
	for any $\eta$.

	Let $B^{\out}\subset \iota_{N}^*A$ be the maximal subwave over $\iota_{N}^*I$ with one end in $J.$
	Let $A^{\out}\subset A$ be the  subwave consisting of leaves whose restriction by $\iota_{N}$ contributes a path to $B^{\out}$, so $\WW(B^{\out})\geq \WW(A^{\out}).$

	For all $1\leq m \leq M$, let $B^\lo_m$ be the maximal wave over $\iota_N^*\BL_{m-N}$ with one end in $J$. For each $m$, Lemma \ref{lem:wave} implies that there is a wave $C^\lo_m$ over $f^N_*\iota_N^*\BL_{m-N}$ with 
	$$\WW(B^\lo_m)\leq \WW(C^\lo_m)+O(M).$$
	As $f^N_*\iota_N^*\BL_{m-N}= F_*^N\BL_{m-N} = \ddi\TT$, each $C^\lo_m$ must be empty. 
	Let $A^\lo\subset A$ be the subwave consisting of leaves whose restriction by $\iota_N^*$ contributes a path to some $B_m^\lo$, so the above implies 
	$$\WW(A^\lo)\leq \sum_{m=1}^M\WW(B_m^\lo)\leq \eps\WW(A).$$

	For all $1\leq m\leq M$, let $B_m^\amp\subset \iota_{N}^*A$
	be the maximal wave with one end in $\iota_N^*\BL_{m-N}$ and one end in $\iota_N^*\BL_{j, m-N-1}$. Let $A^\amp\subset A$ be the subwave consisting of leaves whose restriction by $\iota_{N}^*$ contributes a different path to each $B_m^\amp$. 
	Thus there is some $m$ so that 
	$$M\WW(A^\amp)\leq \WW(B_m^\amp).$$
	For some $k$ depending only on $M$, it follows from Lemmas \ref{lem:fast pushing controllable} and \ref{lem:wave} that there is a $k$-controllable wave $C^\amp\subset f^{N}_*B_m^\amp$ satisfying 
	$$\WW(B_m^\amp)\leq \WW(C^\amp)+O(M)\leq \WW(C^\amp)+ \WW(A)/2.$$
	Thus if $\WW(A^\amp)\geq \eps\WW(A)$, then $\WW(C^\amp)\geq (3/2)\WW(A)$. This implies that $A$ is $(1/2, k)$-terminal, so Corollary \ref{cor:universal terminal} implies that $A$ is $(\delta, \eta)$-terminal for some $\delta$ and universal $\eta$.

	Let us now observe that $A= A^\ver\cup A^{\out}\cup A^\lo\cup A^\amp$. 
	Indeed, as $F_*^N$ is proper on $I$ and end-proper on each limb of $\II$,  for each $1\leq m \leq M$ it follows from Proposition \ref{prop:itineraries} that the restriction of any  leaf $\alpha\in A\sm (A^\ver\cup A^{\out})$ must  contain either a wave over $\iota_N^*\BL_{m-N}$  with one end in $\iota_N^*J$ or a wave with one endpoint in $\iota_N^*\BL_{m-N}$ and the other in   $\iota_N^*\BL_{j, j+m-N-1}$.
	
	It follows from the above that either $A$ is $(\delta, \eta)$-terminal for some $\delta$ and universal $\eta$, or $$\WW(A^\outbr)\geq (1-3\eps)\WW(A).$$
	It follows from Lemmas \ref{lem:fast pushing controllable} and \ref{lem:wave} that there is a $k_0$-controllable wave $C\subset f_*^n(B^\outbr)$ with 
	$$\WW(C)\geq \lambda_0\WW(A)$$
	for some $k_0$ depending only on $M$. 
	For any $\lambda_1$, Proposition \ref{prop:mostly short} implies that $C$ is either $\lambda_1$-short or $(\delta_1, \eta_1)$-terminal for some $\delta_1$ and some $\eta_1$ depending only on $M$. 
	Choosing $\lambda_0\gg_{\delta_1} 0$, if $C$ is $(\delta_1, \eta_1)$-terminal then it follows from Corollary \ref{cor:universal terminal} that $A$ is $(\delta, \eta)$-terminal for some $\delta$ and universal $\eta$. 
	As the crest of $C$ contains $F_*^NI$, if $C$ is $\lambda_1$-short, then for any $\lambda_2$ it follows from Corollary \ref{cor:short pushforward} that either there is a short wave $D$ over $F_*^{\bound+2}F_*^NI= F_*^{q-j}I$ with $$\WW(D)\geq \lambda_2\WW(C)\geq \lambda_0\lambda_1\lambda_2\WW(A)$$ or $C$ is $(\delta_2, \eta_2)$-terminal for some $\delta_2$ and universal $\eta_2$. 
	Choosing $\lambda_0\lambda_1\lambda_2\geq \lambda$, as  $F_*^{q-j}I$ intersects $\BL_0$ it follows that $D$ is a critical $\lambda$-visit for $A$.
	Choosing $\lambda_0\lambda_1\gg_{\delta_2} 0$, if $C$ is $(\delta_2, \eta_2)$-terminal then $A$ is $(\delta, \eta_2)$-terminal for some $\delta$.
\end{proof}

\section{Rigidity}\label{sec:rigidity}

We now sketch the framework for upgrading \emph{a priori} bounds to rigidity; for more details we refer the reader to \cite{lyubich1997dynamics}.

Let $M\subset\MM$ be a prime and primitive small Mandelbrot set, and let $f$ be a quadratic polynomial corresponding to a parameter in $M$. 
Let $I^0(f)$ denote the closure of the union of the external rays of $f$ landing at $\alpha_f$ together with the equipotential of $f$ at potential $1$. The set $I^0(f)$ cuts $\C$ into the \emph{Yoccoz puzzle pieces}  of level $0$. 
For all $n>1$, denoting $I^n(f):= f^{-n}(I^0(f))$, the Yoccoz puzzle pieces of level $n$ are defined to be the components of $\C\sm I^n(f)$. 
Let $Y^n(f)$ denote the level $n$ Yoccoz puzzle piece containing the critical point of $f$.
For $p$ the associated period of $M$, there is some minimal $n\geq 0$ depending only on $M$ so that $f^n:Y^{n+p}(f)\to Y^n(f)$ restricts to a quadratic-like map with connected filled Julia set; $n$ is called the \emph{renormalization level} of $M$. 

For $n$ the renormalization level of $M$ and $f, g$ two quadratic polynomials corresponding to parameters in $M$, there is a canonical homeomorphism $I^n(f)\to I^n(g)$ induced by the canonical markings of external rays and equipotentials. 
Let $d_M(f, g)$ denote the \emph{Teichm\"uller distance}  between $I^n(f)$ and $I^n(g)$, that is half of the $\log$ of the smallest dilatation of a quasiconformal homeomorphism $h: \C\to \C$ that agrees with the canonical map $I^n(f)\to I^n(g)$. We define the \emph{puzzle Teichm\"uller diameter}  $\text{diam}(M)$ to 
be the max of $d_M(f, g)$ over all such $f$ and $g$.
Let us say that set $\CC$ of prime and primitive types has \emph{bounded puzzle Teichm\"uller diameter} \index{family $\CC$ of types!bounded puzzle Teichm\"uller diameter} when 
$$\sup_{M\in \CC}\text{diam}(M)< \infty.$$

A family of types $\CC$ is \emph{combinatorially rigid} \index{family $\CC$ of types!combinatorially rigid} if any two infinitely renormalizable quadratic-like maps with identical combinatorics in $\CC$ are hybrid equivalent; if $\CC$ is combinatorially rigid then it is automatically rigid in the sense from the introduction.
The following theorem is proved in \cite{lyubich1997dynamics}, although it is not stated in this generality there:

\begin{thmB}
	Every set of prime primitive types with beau bounds and bounded puzzle Teichm\"uller diameter is rigid.
\end{thmB}

\begin{proof}
	The entire construction from \cite[Section 10]{lyubich1997dynamics} can be applied in this setting without change. The entire purpose of \cite[Section 11]{lyubich1997dynamics}, and more specifically \cite[Main Lemma]{lyubich1997dynamics}, is to show that the secondary limbs condition implies bounded puzzle Teichm\"uller diameter.
\end{proof}

\appendix
\section{Ideal markings and laminations}\label{app:markings}
Here we introduce the concept of an ideal marking and then define the canonical lamination and waves in terms of these markings. 

\subsection{Ideal markings}\label{app:ideal markings and waves}

Let $\VV$ be a hyperbolic Riemann surface without cusps. 
Let $\pi: \D\to \VV$ be the universal covering, and let $\Lambda$ be the limit set of the group $\Delta$ of deck transformations. The \textit{ideal boundary}  of $\VV$ is 
$$\ddi \VV:= (\dD\sm \Lambda)/ \Delta.$$ 
The \emph{ideal closure}  $\VV$ is $\overline{\VV}:= \VV\cup \ddi \VV$.

We define an \textit{ideal marking} \index{ideal marking $(\VV, \II)$} of $\VV$ to be a countable set 
$\II= \{\BI_n\}_n$  of pair-wise disjoint 
open connected subsets of the $\ddi\VV$ such that $\ddi \VV = \overline {\bigcup_n\BI_n}.$ We will call the elements of $\II$ \textit{intervals}, and the elements of the complement $\ddi \VV\sm {\bigcup_n {\BI_n}}$ \textit{endpoints}. The \textit{cardinality} of $\II$ is the number of intervals it contains. 
Note that we allow intervals to be entire components of the ideal boundary, in which case the interval is homeomorphic to $\dD$.

Fixing an ideal marking $\II$ of $\VV$, 
we will call a connected open subset $I$ of the ideal boundary of $\VV$ that is the union of intervals and endpoints a \emph{multi-interval} in $\II$.
We define the \textit{combinatorial length} $I$ to be the number $\ell(I)$ of intervals in $\II$ it contains. 
We will say that a multi-interval $J$ is \emph{well-inside} $I$, and write $J\Subset I$, if $\overline J\subset I$.
We will say that $\BI\in \II$ is an \textit{end-interval} of $I$ if $\BI\subset I$ and  $\BI\not \Subset I$; if $I$ is an entire component of $\ddi \VV$ then $I$ has no end-intervals.

We will say that another ideal marking $\JJ$ of $\VV$ is a \emph{refinement} \index{ideal marking $(\VV, \II)$!refinement} of $\II$ if every interval in $\II$ is a multi-interval in $\JJ$. We will say that the refinement is \emph{minor} \index{ideal marking $(\VV, \II)$!refinement!minor} if the combinatorial length in $\JJ$ of every interval in $\II$ is one; that is every interval in $\JJ$ can be obtained from an interval in $\II$ by removing at most one point. 
The ideal markings of $\VV$ are partially ordered by refinement, and there is a unique minimal ideal marking in this order given by the set of all components of $\ddi \VV$.

We will say that a \textit{partial ideal marking} \index{ideal marking $(\VV, \II)$!partial marking} of $\VV$ is a nonempty countable set $\JJ$ of pairwise disjoint connected open sets and points in $\ddi\VV$.  From the partial marking we can form the \textit{induced marking}
\index{ideal marking $(\VV, \II)$!induced marking} by taking the smallest tiling that contains each interval in $\JJ$ and has each point in $\JJ$ as an endpoint. 

For any subset $X$ of $\ddi \VV$, 
we will say that a proper path $\gamma$ in $\overline \VV$ is \textit{subordinate to $X$} \index{ideal marking $(\VV, \II)$!lamination!subordinate} if, up to homotopy rel its endpoints, $\gamma$ belongs to every neighborhood of $X$ in $\overline\VV$. 
We will say that a proper path $\gamma$ in $\overline\VV$ is a \textit{path in $(\VV, \II)$} if both its endpoints belong to intervals in $\II$ and it is not subordinate to a multi-interval with combinatorial length at most two.
An \textit{arc} \index{ideal marking $(\VV, \II)$!arc} in $(\VV, \II)$ is a homotopy class of simple paths.
A \textit{lamination} \index{ideal marking $(\VV, \II)$!lamination} in $(\VV, \II)$ is a lamination in $\VV$ consisting  of paths in $(\VV, \II)$. We will say that a path 
in $(\VV, \II)$ \textit{ends} \index{ideal marking $(\VV, \II)$!lamination!ends} at a multi-interval $I$ if it has an endpoint in $I$. We will say that a lamination $\FF$ in $(\VV, \II)$ ends at $I$ if every leaf ends at $I$; we will say that $\FF$ is \textit{localized} \index{ideal marking $(\VV, \II)$!lamination!localized} \index{waves!localized} if it ends at an interval in $\II$. 

Let $S$ be a lamination in $(\VV, \II)$ subordinate to a multi-interval $B$ in $\II$, that is every leaf of $S$ is subordinate to $B$; we will call $S$ a \textit{wave subordinate to $B$}. \index{waves} We will call the smallest such multi-interval $B$ the \textit{base} \index{waves!base}  of $S$. We define the \textit{ends} 
\index{waves!ends} of $S$ to be the smallest multi-intervals $I$ and $J$ such that each leaf of $S$ has one endpoint in $I$ and the other in $J$. We define the \textit{crest} \index{waves!crest} of $S$ to be $C= B\sm (I\cup J)$; we will say that $S$ \textit{crests over} $X\subset \ddi \VV$, or equivalently that $S$ is a \textit{wave over} $X$, if $X\subset C$.

\subsection{The canonical weighted arc diagram}

\subsubsection{Rectangles}
Let $\overline{\BPi}$ be a standard rectangle with horizontal sides $\BI$ and
$\BJ$, and let $\ol \FF= \overline{\FF}(\BI, \BJ)$ be the vertical foliation on $\overline{\BPi}$. We let 
$$
\overline{\WW}(\BI, \BJ)  \equiv {\WW}(\overline{\FF}) 
$$
be its conformal {\it width}, which is also equal to the width of the 
{full path family} in $\overline{\BPi}$ connecting $\BI$ to $\BJ$. 
When the width of $\overline{\FF}$ is greater than 2, we can remove square buffers from $\overline{\BPi}$ to obtain a quadrilateral $\BPi$ supporting a foliation $\FF = \FF(\BI, \BJ)$ with width
$$
\WW(\BI, \BJ) \equiv 
\WW(\FF)  = \WW(\ol \FF) -2. 
$$
If instead the width of $\ol \FF$ is at most two, then we take $\FF= \FF(\BI, \BJ)$ to be the empty foliation, so $\WW(\BI, \BJ) \equiv \WW(\FF) = 0$.
Denoting the hyperbolic distance and length in $\overline{\BPi}$ by $d_\hyp$ and $\l_\hyp$ respectively, the following lemma relates the conformal width to hyperbolic geometry:
\begin{lem}\label{lem:hyp geo}
	Let $\bg$ and $\Bde$ be the hyperbolic geodesics in $\overline{\BPi}$ that share endpoints with $\BI$ and $\BJ$ respectively. For any $\eps>0$ and $\bg^\eps= \setofst{z\in\bg}{d_\hyp(z, \Bde)< \eps}$, we have:
	$$\l_\hyp(\bg^\eps)= \WW(\BI, \BJ)+O_\eps(1).$$
\end{lem}

\subsubsection{Weighted arc diagrams}\label{app:WAD}
Now let us fix a Riemann surface $\VV$ with ideal marking $\II$
as in the previous subsection. 
A  \textit{weighted arc diagram} \index{weighted arc diagram (WAD)} (WAD) on $(\VV, \II)$ is a function $W: \AA\to \R_{\geq 0}$ such that any two arcs in the support of $W$ do not intersect.
The \textit{norm} of a weighted arc diagram is defined as 
$$\norm{W} = \sum_{\alpha\in \AA}W(\alpha).$$
Denoting the Euler characteristic of $\VV$ by $\chi(\VV)$, 
the number of arcs in the support of any WAD on $(\VV, \II)$ is at most $3|\chi(\VV)|+ |\II|$.

For any lamination $\FF$ and arc $\alpha$ in $(\VV, \II)$, let $\FF(\alpha)\subset \FF$ denote the sublamination consisting of leaves represented by $\alpha$. 
The lamination $\FF$ determines a WAD $W_\FF$ on $(\VV, \II)$ defined by
$$W_\FF(\alpha) = \WW(\FF(\alpha)).$$
We will call $W_\FF$ the WAD \textit{induced by} $\FF$; any WAD induced by a lamination is called \textit{valid}.

\subsubsection{The canonical lamination and arc diagram}\label{app:wcan}

Let $\pi: \tilde \VV\to \VV$ be the universal cover of $\VV$, and let $\Lambda$ be the limit set of the group $\Delta$ of deck transformations. 
The ideal marking $\II$ lifts to an ideal marking $\tilde \II$ of $\tilde \VV$. Every interval $\tilde \BI\in \tilde \II$ is mapped to an interval $\BI\in \II$ as either a homeomorphism or a universal cover; in either case we will call $\tilde \BI$ a \textit{lift} of $\BI$. 

If $(\VV, \II)$ is an annulus with the minimal ideal marking, then we can identify $\tilde \VV$ with an infinite strip $S= \{z: |\text{Im}z|< R\}$ for some $R>0$ so that $\Delta$ is the group generated by $z\mapsto z+1$ and $\tilde \II = \{\R\pm iR\}.$
In this case, 
the foliation of $S$ by vertical lines descends to a lamination on $\VV$, 
which we 
denote by $\fcan(\alpha) = \fcan(\VV, \II)(\alpha)$ for $\alpha$ the unique arc in $(\VV, \II)$.  We also define the \textit{flux} through $\alpha$ to be 
$$\flux(\alpha)= \flux(\VV, \II)(\alpha):= 1/\WW(\fcan(\alpha))= \mmod(\VV).$$
\index{ideal marking $(\VV, \II)$!arc!flux}

We suppose now that $(\VV, \II)$ is not an annulus with the minimal ideal marking. 
Any arc $\alpha$ in $(\VV, \II)$ connecting intervals $\BI$ and $\BJ$ lifts to an arc $\tilde \alpha$ in $(\tilde\VV, \tilde \II)$ connecting some intervals $\tilde \BI$ and $\tilde \BJ$ that cover $\BI$ and $\BJ$ respectively. These lifted intervals determine an ideal quadrilateral with ``horizontal" sides $\tilde\BI$, $\tilde\BJ$, and corresponding foliations $\FF(\tilde \BI, \tilde \BJ)$ as above. 
This foliation descends to a lamination $\fcan(\alpha) = \fcan(\VV, \II) (\alpha)$ on $(\VV, \II)$ whose width does not depend on the choice of lift $\tilde \alpha$. 
For an arc $\alpha$ as above, we also define the \textit{flux} through $\alpha$ to be
$$\flux(\alpha)= \flux(\VV, \II)(\alpha):= 1/{\overline{\WW}(\tilde \BI, \tilde\BJ)}.$$

Denoting the set of all arcs in $(\VV, \II)$ by $\AA = \AA(\VV, \II)$, the union $$\fcan = \fcan(\VV, \II):=\bigcup_{\alpha\in \AA}\fcan(\alpha)$$
is a lamination, called the \textit{canonical lamination} of $(\VV, \II).$
 The weighted arc diagram 
$\wcan = \wcan(\VV, \II)$ induced by $\fcan$ is called the \textit{canonical WAD} \index{weighted arc diagram (WAD)!canonical} of $(\VV, \II)$. 
The canonical WAD is maximal in the sense that if $W$ is another valid WAD on $(\VV, \II)$ then 
$$W(\alpha)\leq \wcan(\alpha)+2$$
for all arcs $\alpha$. Note that the function $\alpha\mapsto \flux(\alpha)$ is  usually not a weighted arc diagram as every arc has positive flux.

For any interval $\BI\in \II$, we denote by $\AA(\BI)= \AA(\VV, \II)( \BI)$ the set of oriented arcs in $(\VV, \II)$ that start at $\BI$. 
We will call  
$$
\wcan(\BI)= \wcan(\VV, \II)(\BI):=\sum_{\alpha\in \AA(\BI)}\wcan(\alpha)
$$
the \textit{local canonical weight} of $\BI$ in $(\VV, \II)$.
We use oriented arcs so that the weight of any arc connecting $\BI$ to itself is counted twice. In particular, we have 
$$2\|\wcan\| = \sum_{\BI\in \II}\wcan(\BI).$$

\subsubsection{Local widths}

If  $\BI\in \II$ is simply connected and $\tilde \BI\in \II$ is a lift of $\BI$, then there exists some minimal multi-interval $\tilde I$ in $\tilde \II$ that contains the closure of $\tilde \BI$; the \textit{local width} of $\BI$ is defined to be $$\WW(\BI)= \WW(\VV, \II, \BI):= \overline{\WW}(\tilde \BI, \ddi \tilde \VV\sm \tilde I).$$
This width does not depend on the choice of lift $\tilde \BI$. 
If instead $\BI\in \II$ is not simply connected and $\tilde \BI\in \II$ is a lift of $\BI$, then the stabilizer of $\tilde \BI$ in $\Delta$ is a cyclic group generated by a hyperbolic M\"obius transformation $\tau$. The Riemann surface $\tilde \VV/\langle \tau \rangle$ is an annulus; the \textit{local width} of $\BI$ in this case is defined as 
$$\WW(\BI)= \WW(\VV, \II, \BI):= \WW(\tilde \VV/\langle \tau \rangle).$$
We define the \textit{total width} of $(\VV, \II)$ to be 
$$\WW(\VV, \II)= \sum_{\BI\in \II}\WW(\BI).$$

\subsubsection{Thick-thin decomposition}
It follows immediately from the definitions that we can relate the local weight and local width of an interval by
\begin{equation}\label{eq:total width vs total weight single}
	\WW(\BI) - O(|\chi(\VV)|+ |\II|) \leq \wcan(\BI)\leq \WW(\BI).
\end{equation}
In fact, the following theorem, first proved in \cite{KL:eyes}, compares the local weight and width of $(\VV, \II)$:
\index{Thick-thin decomposition}
\begin{theorem}[Thick-thin decomposition]\label{main estimate} 
	If $\VV$ is a Riemann surface with ideal marking 
	$\II$, then:
	\begin{equation}\label{eq:thick-thin}
		\WW(\VV, \II)- O(|\chi(\VV)|+ |\II|) \leq 2\|\wcan(\VV, \II)\|\leq \WW(\VV, \II).
	\end{equation}
\end{theorem}  

\begin{proof}
	The second inequality follows immediately from \eqref{eq:total width vs total weight single}, so we only need to prove the first inequality. We  set $p = |\chi(\VV)|+ |\II|.$ 
	
	We consider first the case where every interval in $\II$ is simply connected.
	For every $\BI\in \II$, let $\gamma_\BI$ be the hyperbolic geodesic in $\VV$ with the same endpoints as $\BI$.
	We set $\Gamma = \bigcup \gamma_\BI$.
	
	Let $\pi: \tilde \VV\to \VV$ be a universal covering, and let $\tilde \II$ be the corresponding lift of $\II$.  For any multi-interval $\tilde I$ in $\tilde \II$, we let $\gamma_{\tilde I}$ be the hyperbolic geodesic in $\tilde \VV$ with the same endpoints as $\tilde I$. Note that if $\tilde \BI$ is a lift of  $\tilde \BI\in \tilde \II$, then $\pi$ maps $\gamma_{\tilde \BI}$ isometrically to $\gamma_\BI$.

	Let $d_\hyp$ and $\l_\hyp$ denote hyperbolic distance and length respectively  in both $\VV$ and $\tilde \VV$. 
	We now fix some $\eps>0$ small enough so that for any disjoint hyperbolic geodesics $\gamma_0, \gamma_1, \gamma_2$ in $\tilde \VV \cong \D$ with none of the three separating the other two, if $x\in \gamma_0$ then the $\eps$-ball around $x$ intersects at most one of $\eps_1$ and $\eps_2$.

	For any oriented arc $\alpha$ in $(\VV, \II)$ connecting intervals $\BI$ and $\BJ$ with corresponding lift $\tilde\alpha$ in $(\tilde \VV, \tilde \II)$ connecting interval $\tilde \BI$ and $\tilde \BJ$, we set 
	\begin{equation*}
		\gamma_{\alpha}= \{\pi(z): z\in\gamma_{\tilde \BI}, d_\hyp(z, \gamma_{\tilde{\BJ}})< \eps\}.
	\end{equation*}
	For any interval $\BI\in \II$, we define the \textit{thin part} of $\gamma_\BI$ to be $\gamma_\BI^\thin:= \bigcup_{\alpha\in \AA(\BI)} \gamma_\alpha$.	
	We define the \textit{thin part} of $\Gamma$ to be $\Gamma^\thin= \bigcup\gamma_\BI^\thin.$
	It follows from Lemma \ref{lem:hyp geo} that 
	$$\wcan(\alpha) = \l_\hyp(\gamma_\alpha)+O(1),$$
	hence 
	\begin{equation}\label{eq:thin part}
		2\|\wcan\|= \l_\hyp(\Gamma^\thin)+O(p).
	\end{equation}
	
	For any interval $\BI\in \II$ and lift $\tilde \BI$ adjacent to some intervals $\tilde \BI_1, \tilde\BI_2$, we similarly define the \textit{cuspidal part} of $\gamma_\BI$ to 
	\begin{equation*}
		\gamma_{\BI}^\cusp = \{\pi(z): z\in\gamma_{\tilde \BI}, d_\hyp(z, \gamma_{\tilde{\BI}_1}\cup \gamma_{\tilde \BI_2})< \eps\},
	\end{equation*}
	which does not depend on the choice of lift. 
	We define the \textit{cuspidal part} of $\Gamma$ to be $\Gamma^\cusp= \bigcup \gamma_\BI^\cusp.$
	
	For any interval $\BI\in \II$, we define the \textit{thick part} of $\gamma_\BI$ to be $\gamma_{\BI}\sm (\gamma_\BI^\thin\cup \gamma_\BI^\cusp);$ we similarly define 
	$\Gamma^\thick = \Gamma \sm (\Gamma^\thin\cup \Gamma^\cusp).$
	Let $D\subset \VV$ be the region bounded by $\Gamma$, it follows from Gauss-Bonnet formula that the hyperbolic area of $D$ is $O(p)$.  
	We consider a maximal net of $\eps$-separated points in $\Gamma^\thick$; it contains approximately $\l_\hyp(\Gamma^\thick)/\eps$ many points. 
	Our choice of $\eps$ ensures that the hyperbolic disks of radius $\eps/2$ centered at these points do not overlap, so the total area of their intersection with $D$ is bounded by $O(p)$. Hence
	$$\frac{\l_\hyp(\Gamma^\thick)}{\eps}\cdot \eps^2 =O(p),$$
	so \begin{equation}\label{eq:thick part}
		\l_\hyp(\Gamma^\thick) = O(p/\eps).
	\end{equation}

	For any lift $\tilde \BI$ of some $\BI\in \II$ adjacent to intervals $\tilde \BI_1$ and $\tilde \BI_2$, let $\tilde I$ be the minimal multi-interval in $\tilde \II$ containing $\tilde \BI, \tilde \BI_1$, and $\tilde \BI_2$.  We set  $$\gamma_\BI^\eps= \setofst{\pi(z)}{z\in \gamma_{\tilde \BI}, d_\hyp(z, \gamma_{\tilde I})< \eps},$$
	which does not depend on the choice of lift.
	It follows from Lemma \ref{lem:hyp geo} that $\l_\hyp(\gamma_\BI^\eps) = \WW(\BI)+ O(1)$.
	Setting $\Gamma^\eps = \bigcup \gamma_\BI^\eps$, our assumption on $\eps$ ensures that $\Gamma^\eps\subset \Gamma\sm \Gamma^\cusp$, so 
	\begin{equation}\label{eq:cusp part}
		\WW(\VV, \II)\leq \l_\hyp(\Gamma^\eps)+O(|\II|) \leq \l_\hyp(\Gamma\sm \Gamma^\cusp)+O(|\II|).
	\end{equation}
	Combining \eqref{eq:thin part}, \eqref{eq:thick part}, and \eqref{eq:cusp part} yields 
	\begin{align*}
		\WW(\VV, \II)&\leq \l_\hyp(\Gamma^\thin)+ \l_\hyp(\Gamma^\thick)+ O(|\II|)\\
		&\leq 2\|\wcan\|+O(p)
	\end{align*}
	as desired.
	
	Now let us return to the general case where some intervals in $\II$ may not be simply connected. We  consider a minor refinement $\JJ$ of $\II$. 
	Thus for any interval $\BJ\in \JJ$ there is a unique $\BI\in \II$ with $\BJ\subset \BI$, moreover either $\BI = \BJ$ or $\BI = \overline\BJ$. 
	If $\BI = \BJ$, then it follows immediately from the definitions that $\WW(\BI) = \WW(\BJ)$. If $\BI = \overline \BJ$, then instead we have
	$\WW(\BI)\leq \WW(\BJ)+O(1).$
	Hence 
	\begin{equation}\label{eq:minor widths}
		\WW(\VV, \II)\leq \WW(\VV, \JJ)+O(p).
	\end{equation}
	
	Let $\alpha$ be an arc in $(\VV, \II)$, and let $B(\alpha)$ be the set of all arcs in $(\VV, \JJ)$ that are equivalent to $\alpha$ in $(\VV, \II)$. As the union $\bigcup_{\beta\in B(\alpha)}\fcan(\beta)$ is a lamination consisting of paths homotopic to $\alpha$, it follows from the maximality of the canonical weighted arc diagram that
	$$\sum_{\beta\in B(\alpha)}\wcan(\beta)\leq \wcan(\alpha)+O(1).$$
	Hence 
	\begin{equation}\label{eq:minor weights}
		\|\wcan(\VV, \JJ)\|\leq \|\wcan(\VV, \II)\|+O(p).
	\end{equation}
	We can choose $\JJ$ so that every interval in $\JJ$ is simply connected, so \eqref{eq:thick-thin} holds for $\JJ$. Together with \eqref{eq:minor widths} and \eqref{eq:minor weights}, we can conclude
	$$\WW(\VV, \II)\leq \WW(\VV, \JJ)+O(p)\leq 2\|\wcan(\VV, \JJ)\|+O(p)\leq 2\|\wcan(\VV, \II)\|+O(p)$$
	as desired.
\end{proof}

\subsection{Pullbacks and pushforwards}\label{app:dynamics}

Let us fix some hyperbolic Riemann surfaces $\VV, \VV'$ and ideal marking $\II$ of $\VV$.

\subsubsection{Coverings} Let $f: \VV'\to \VV$ be a holomorphic covering map,
so $f$ extends continuously to a covering map from an open and dense subset $f^{-1}(\ddi \VV)$ of $\ddi \VV'$ to $\ddi \VV$. The ideal marking $\II$ therefore lifts to an ideal marking $f^*\II$ of $\VV'$. 
For any multi-interval $I'$ in $f^*\II$, we define its \emph{pushforward} to be $f_*I:= f(I),$ the pushforward is automatically a multi-interval in $\II$.

If $\gamma'$ is a proper path in $(\VV', \II')$, then $f(\gamma)$ is a proper path in $(\VV, \II)$; we will call $f_*\gamma:=f(\gamma)$ the \emph{pushforward} of $\gamma$. 
We  similarly define the pushforward of arcs, path families, and laminations, however the pushforward of a lamination is not always a lamination. 
The following lemma controls allows us to extract laminations from the pushforward of laminations:

\begin{lemma}\label{lem:push forward laminations}
	Let $A'$ be a lamination in $\VV'$ ending at some $X'\subset \ddi \VV$.   If $f$ has degree $\leq d$ on $X'$, then there is a lamination $A\subset f_*A'$ such that 
	$$\WW(A')\leq d\WW(A)+2d.$$
\end{lemma}

\subsubsection{Immersions}
Let $\iota: \VV'\to \VV$ be a holomorphic immersion. 
Let us say that $\iota$ is \emph{proper} on some $X'\subset \ddi\VV'$ if $\iota$ extends continuously to a proper map on $X'$.
For a proper path $\gamma$ in $\VV$, we define its \emph{pullback} $\iota^*\gamma$ to be the set of proper paths $\gamma'$ in $\VV'$ contained in $\iota^{-1}(\gamma)$. 
We similarly define the restriction of path families  and laminations. To avoid confusion with pullbacks under cover, we will usually call a pullback by $\iota$ a \emph{restriction}.
The following lemma lets us relate the widths of a lamination and its restriction:
\begin{lemma}\label{lem:domination}
	Let $S$ be a lamination on $\VV$ and let $S_1', \dots, S_N'$ be laminations on $\VV'$. If $\iota^*\alpha$ contains a leaf of each  $S_n'$ for every $\alpha\in S$, then 
	$\WW(S)\leq \WW(S_n')$ for all $n$. If additionally $\iota$ is injective on $\bigcup S_n'$, then 
	$$\WW(S) \leq \bigoplus_{n=1}^N\WW(S_n').$$
\end{lemma}

For any connected subset $X\subset \ddi\VV$,  we define its \emph{pullback} by $\iota$ to be the maximal subset $\iota^*X$ of $\ddi\VV'$ such that: if $\gamma$ is a proper path in $\VV$, there is a unique path $\gamma'\in \iota^*\gamma$, and $\gamma'$ is subordinate to $\iota^*X$, then $\gamma$ is subordinate to $X$.
Note that the pullbacks of intervals in $\II$ do not necessarily form an ideal marking, or even a partial marking, of $\VV'$: some pullbacks may not exist or may overlap. If the set $\setofst{\iota^*\BI}{\BI\in \II}$
forms a partial ideal marking of $\VV'$, then we denote the induced ideal marking by $\iota^*\II$. 

We assume now that $\iota^*\II$ is defined, and let $\II'$ be a refinement of $\iota^*\II$. Note that if $\gamma$ is a path in $(\VV, \II)$, then some of the paths in $\iota^*\gamma$ may not be paths in $(\VV',\II')$. 
In this context, we will ignore all paths in $\iota^*\gamma$ that are not paths in $(\VV', \II')$ or are subordinate to a single interval in $\iota^*\II$. 
The following proposition controls the combinatorics of the restrictions of paths:
\begin{prop}\label{prop:itineraries}
    Let $\II'$ be a refinement of $\iota^*\II$. 
	If $\gamma$ is a path in $(\VV, \II)$ connecting two intervals $\BI, \BJ\in \II$, there there is a finite sequence $(\BI_n')_{n=0}^{N+1}$ of intervals in $\II'$ satisfying:
	\begin{enumerate}
		\item $\BI_0'\subset \iota^*\BI$ and $\BI_{N+1}'\subset \iota^*\BJ$.
        \item $\iota$ is not proper on $\BI_n'$ for $1\leq n \leq N$.
		\item For all $0\leq n \leq N$, either $\iota^*\gamma$ contains a path $\gamma_n'$ connecting $\BI_n'$ to $\BI_{n+1}'$, $\BI_{n}'$ and $\BI_{n+1}'$ are adjacent, or $\BI_n'$ and $\BI_{n+1}'$ both belong to $\iota^*\BI_n$
		for some $\BI_n\in \II$. 
	\end{enumerate}
\end{prop}

\subsubsection{Domination}\label{sec:domination}
Let us observe that Lemma \ref{lem:domination} induces a ``domination" relation on the WADs introduced in \cite{K06}. We suppose that $\iota^*\II$ is defined and $\II'$ is a refinement of $\iota^*\II$. 
We will say that a collection of arcs $(\alpha_i)_i$ in $\VV'$ \emph{arrows} and arc $\beta$ in $\VV$, and write $(\alpha_i)_i\to \beta,$ if there are paths $a_i$ representing the $\alpha_i$ and a path $b$ representing $\beta$ such that $(a_i)_i= \iota^* b$.
A WAD $X$ on $(\VV', \II')$ \emph{dominates} \index{weighted arc diagram (WAD)!domination $\dominates$} a WAD $Y$ on $(\VV, \II)$, denoted $X\dominates Y$, if we can write 
$$X\geq \sum_i \sum_jw_{i,j}\alpha_{i, j} \text{ and }Y= \sum_i v_i\beta_i,$$
where $$(\alpha_{i, j})_j\to \beta_i \text{ and }\bigoplus_jw_{i, j}\geq v_i$$
for all $i$. 
If $\iota$ is a proper embedding, $Y$ is a WAD induced by a lamination $\FF$ on $(\VV, \II)$, and $X$ is the WAD on $(\VV', \II')$ induced by $\iota^*\FF$, then Lemma \ref{lem:domination} immediately implies that $X\dominates Y$.
Let us say that $\iota$ \emph{factors through an embedding} if there is a Riemann surface $\UU$ and holomorphic immersions $\jmath: \UU\to \VV$ and $\jmath': \VV'\to \UU$ such that $\iota = \jmath\circ \jmath'$ and one of $\jmath, \jmath'$ is a proper embedding. 
In this case, if $Y$ is a WAD induced by a lamination $\FF$ on $(\VV, \II)$, and $X$ is the WAD on $(\VV', \II')$ induced by $\iota^*\FF$, then 
applying 
Lemma \ref{lem:domination} twice yields $X\dominates Y$.

\subsubsection{Dynamics}\label{sec:dynamics}

Let $f:\VV'\to \VV$ and $\iota:\VV'\to \VV$ be a holomorphic cover and holomorphic immersion respectively as above. 
Denoting the pair $F= (f, \iota)$ and fixing an ideal marking $\II$ of $\VV$, we will say that $F$ \textit{acts on $(\VV, \II)$} if $f^*\II$ is a refinement of $\iota^*\II$. In this case, $F_*:= f_*\iota^*$ acts on multi-intervals in $\II$. We can naturally extend this action to the closures of multi-intervals; note that we are not guaranteed that $F_*\overline{I} = \overline{F_*I}$, indeed the pullback of an endpoint of $I$ may itself be a multi-interval in $\II'$. 

\begin{figure}
	\begin{center}
		\def\svgwidth{4.5in}
		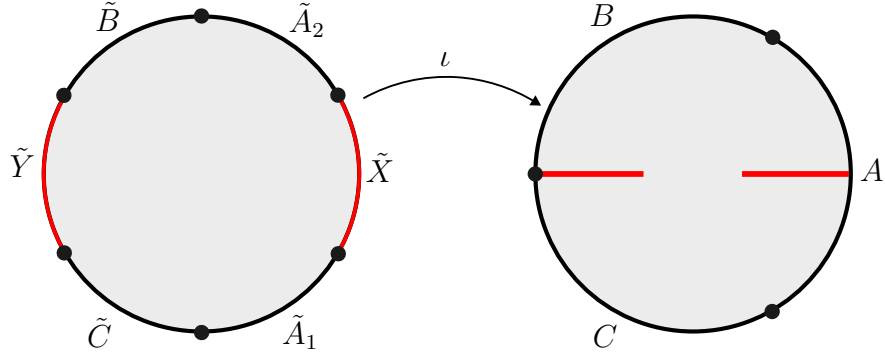
		\caption{
        An example of lifting intervals under an inclusion $\iota$. Here, $\iota^*A$ contains $\tilde A_1\cup \tilde X\cup \tilde A_2$, $\iota^*B = \tilde B$, and $\iota^* C = \tilde C$. In this example,  $\iota$ is end-proper on $\iota^*A$ and proper on $\iota^*B$ and $\iota^*C$.}
		\label{fig:lift}
	\end{center}
\end{figure}

The pushforward $F_* = f_*\iota^*$ also acts on path families and arcs; recall from the previous subsection that for a path $\gamma$ in $(\VV, \II)$ we only include in $\iota^*\gamma$ paths in $(\VV', \II')$ that are not subordinate to an interval in $\iota^*\II$.
Proposition \ref{prop:itineraries} together with Lemmas \ref{lem:push forward laminations} and \ref{lem:domination} allows us to study the action of $F_*$ on laminations. 
We end this section by examining some special cases where we have extra control of this action.

For waves, we have the following improvement to 
Lemma \ref{lem:push forward laminations} which was first introduced in \cite[Lemma 5.2]{dudko2023mlcfeigenbaumpoints} as the part of their proof of the \textit{Wave Lemma} there:
\begin{lemma}\label{lem:wave}
	Let $A'$ be a wave in $(\VV', f^*\II)$ ending at multi-interval $I'$.  If $f$ has degree $\leq d$ on $I'$, then there is a wave $A\subset f_*A'$  with 
	$$\WW(A')\leq \WW(A)+ O(1+\log d).$$
\end{lemma}

For any multi-interval $I'$ in $f^*\II$, we will say that $\iota$ is \textit{end-proper} on $I'$ if it is proper on the end-intervals of $I'$ (see Figure \ref{fig:lift}).
We will say that $F_*$ is proper (resp. end-proper) on an interval $I$ in $\II$ if $\iota$ is proper (resp. end-proper) on $\iota^*I$.
For any connected subset $X$ of $\ddi \VV$, we will say that $F$ acts \textit{combinatorially} on $X$ if $F_*$ is end-proper on every interval in $\II$ contained in $X$. 

Let us say that a path $\gamma$ in $(\VV, \II)$ \textit{breaks} under $F_*$ if $F_*\gamma$ is not a single path. 
For a path family $\Gamma$, we define the \textit{breakable} and \textit{unbreakable parts} of $\Gamma$ to be the subfamilies $\Gamma^\br$ and $\Gamma^\nbr$ consisting  of paths  that do or do not break under $F_*$ respectively. 
When $F$ acts combinatorially, we can  simplify Proposition \ref{prop:itineraries} for unbreakable laminations:
\begin{prop}\label{prop:nonbreaking}
    Assume that $F_*$ acts combinatorially on a multi-interval $I$ in $\II$, and
    let $A$ and $B\subset F_*A^\nbr$ be laminations on $(\VV, \II).$
    If $E\Subset I$ is one end of $A$, then $F_*\overline I$ contains an end of  $B$. If $A$ is a wave over  $I$, then $B$ is a wave over $F_*I$.
\end{prop}

Let us say that a multi-interval $I\in \II$ \textit{shifts} under $F$ if $F_*$ acts combinatorially on $I$ and for any 
interval $\BI\subset I$, the intersection $\BI\cap F_*\overline\BI$ is empty. The following lemma, an analogue of the \textit{Shift Argument} introduced in \cite[A.3]{DLneutral}
\cite{dudko2023mlcfeigenbaumpoints}, bounds the width of some unbreakable waves with shifting bases:

\begin{lemma}\label{lem:shift}
    Assume that $B$ is a multi-interval in $\II$ that shifts under $F_*$, and let $S$ be a wave subordinate to $B$ with crest $C$ and end $E\Subset I$. 
     If $ F_*\overline E\subset C$ or $E\subset F_*C$, then $\WW(S^\nbr)\leq 2$.
\end{lemma}

\section{Translation markings} \label{app:transl-mark}

Fixing some positive integers $n, p$ and setting $N = 2(n+1)p$, let $\II = (\BI_j)_{j\in \Z/N\Z}$ be a cardinality $N$ ideal marking of $\dD$ with the intervals labeled according to their counter-clockwise circular ordering. 
Let $f: \D\to \D$ be a holomorphic isomorphism that extends continuously to a homeomorphism from $\BI_j$ to $\BI_{j+2p}$ for all $0 < j < N-2p$. Let $\iota: \D\to \D$ be a holomorphic immersion that extends continuously to a  homeomorphism from $\BI_j$ to $\BI_j$ for all odd $1\leq j < N$. 
Denoting the pair $(f, \iota)$ by $F$, we will call the pair $(F, \II)$ an \textit{$(n, p)$-translation marking}. \index{translation marking}

We define the  \textit{standard width} \index{translation marking!standard width} of the translation marking $(F, \II)$ to be the quantity
$$\WW(F, \II): =\sum_{|j|\leq 2p}\WW(\BI_j).$$
Let us say that an arc in $(\D, \II)$ is \textit{long} \index{translation marking!long arcs} if it is represented by a wave over $\bigcup_{j=1}^{2p}\BI_{m+j}$ for some $0< m < N-2p$; we denote by $\wcan^\lo(\D, \II)$ the restriction of $\wcan(\D, \II)$ to long arcs. 
In this section we bound the total weight of long arcs by the standard width:
\begin{theorem}\label{thm:bounded short}
	For any $(n, p)$-translation marking $(F, \II)$, we have:
	$$\|\wcan^\lo(\D, \II)\|= O(\WW(F, \II)+np).$$
\end{theorem}

\begin{proof}
	Setting $I = \bigcup_{|s|\leq 2p}\BI_s$, let $X\subset \fcan(\D, \II)$ denote the sublamination consisting of leaves subordinate to $\dD\sm I$ and let $X^\lo\subset X$ be the sublamination consisting of long leaves.
	Any leaf of $\fcan^\lo(\D, \II)$ not in $X^\lo$ must have an endpoint in $I$ automatically,
	hence 
	\begin{equation}\label{eq:translation marking mostly long}
		\|\wcan^\lo(\D, \II)\|\leq \WW(X^\lo) + \WW(F, \II).
	\end{equation}
	
	Let $A$, $B$ and $C$ be the sublaminations of $X$ consisting of all the leaves connecting some $\BI_j$ to $\BI_k$ with 
	$i, j$ both odd, one of $i, j$ odd, and $i, j$ both even respectively.
	Fixing some $\BI_j$ and $\BI_k\subset \dD\sm I$ with $2p< j < k< N-2p$, let $Y\subset X$ be the sublamination consisting of leaves connecting $\BI_j$ to $\BI_k$.
	Note that $X$ can be written as the sum of at most $N$ many such sublaminations $Y$.

    Let us recall the definition of the breakable and unbreakable parts under $F$ of laminations from Appendix \ref{app:dynamics}.
	If $Y\subset X^\lo$, so $k > j+2p$, then every leaf of $f_*Y^\nbr$ intersects every leaf of $Y^\nbr$; hence $\WW(Y^\nbr)\leq 2$.
	If both $j$ and $k$ are odd, so $F_*$ is proper on $\BI_j$, $\BI_{j+2p}$, and $\BI_k$, then it 
	follows from Proposition \ref{prop:itineraries} that the pushforward by $F_*$ of any leaf of $Y$ must contain a path that intersects every leaf of $Y$; hence $\WW(Y)\leq 2$.
	Thus 
	\begin{equation}\label{eq:translation marking arc}
		\WW(X^\lo\cap X^\nbr)+ \WW(A^\lo) \leq 4N.
	\end{equation}
	If both $j$ and $k$ are even, then Proposition \ref{prop:itineraries} implies that the pushforward 
	by $F_*$ of any leaf in $Y^\nbr$ or $Y^\br$ must contain one or two paths respectively either connecting even intervals in  $\II$ or with one endpoint in $I$. Similarly, if exactly one of $j$ and $k$ is odd, then the pushforward by $F_*$ of any leaf of $Y^\br$ must contain a path 
	connecting two even intervals in $\II$ or a path with an endpoint in $I$. 
	It therefore follows from the maximality of the canonical lamination that 
	$$\WW(B^\br)+\WW(C^\nbr)+ 2\WW(C^\br)\leq \WW(C)+\WW(F, \II)+2N,$$
	so 
	\begin{equation}\label{eq:translation marking segments}
		\WW(B^\br)+\WW(C^\br)\leq \WW(F, \II)+2N.
	\end{equation}
	Combining  \eqref{eq:translation marking mostly long}, \eqref{eq:translation marking arc}, and \eqref{eq:translation marking segments} yields
	\begin{align*}
		\|\wcan^\lo(\D, \II)\|&\leq \WW(X^\lo\cap X^\nbr)+ \WW(A^\lo)+ \WW(B^\br)+\WW(C^\br)+\WW(F, \II)\\
		&\leq 2\WW(F, \II)+6N
	\end{align*}
	as desired.
\end{proof}

\section{Cross sections and cloning} \label{app:clones}

In this appendix we introduce the concept of cloning a cross section of a directed graph, and estimate the leading eigenvalue and associated eigenvector of the result. 

Let $A$ be a finite directed graph with vertex set $V$ and let $E \subset V$. 
We say that $E$ is a \emph{cross section} of $A$
 if there is some $N \in \N$ such that every directed path in $A$ of length $N-1$ intersects $E$, 
 and no edge goes from $E$ to itself. 
We fix $(A, E, N)$ with these properties for this appendix.
Given such an $E$, we can \emph{clone} $E$ by replacing $E$ with $n+1$ copies $E_0, \ldots, E_n$ of $E$: the edges into each vertex in $E$ becomes an edge into $E_0$ (with the same source), each edge out of $E$ becomes an edge out of $E_n$ (with the same target), and we add an edge from each vertex in $E_k$ with $k < n$ to the copy in $E_{k+1}$ of the same vertex. 
We let $A_n$ be the result of this cloning, so $A_0 = A$, and let $V_n$ be the vertex set of $A_n$.

Let $T\from \R^{V} \to \R^{V}$ be the pushforward matrix for $A$. 
Let $T_n$ be the matrix for $A_n$.
Suppose that $A$ is primitive (i.e. $T_0 = T$ is primitive) and further suppose that $T_n$ is primitive for $n \in \N$. 
Let $\lambda_n$ and $v_n$ denote the \PF\ eigenvalue and (normalized) eigenvector for $T_n$, 
so $v_n \in \R^{V_n}$. 
We observe that, when $x \in E_0$, and $y$ is the clone of $x$ in $E_k$ (for $0 \le k \le n$), then $v_n(y) = \lambda^k v_n(y)$. 
We also observe that, when $n > N$, every directed path of length $n$ in $A_n$ that \emph{ends} at a vertex in $E_0$ must begin in some $E_k$, with $0 \le k \le N$. 

Let $\hat A_n$ be the labeled directed graph  on $E_0 \subset A_n$ where, for each path of length $n$ from $x \in E_k$ to $y \in E_0$ in $A_n$, we add an edge from $x'$ to  $y$, where $x$ is the clone of $x' \in E_0$, and we label the edge with $k$. So $\hat A_n$ may have multiple edges between two given vertices in $E_0$. We observe (and define) the following:
\begin{enumerate}
\item
The graph $\hat A_n$ is independent of $n$, for $n > N$ (after identifying the various $E_0 \subset A_n$). So we can let $\hat A = \hat A_n$ for any $n > N$. 
\item
For $\lambda>0$, we let $\hat T_\lambda$ be the pushforward matrix of $\hat A$ where we multiply by $\lambda^k$ when pushing forward along an edge with label $k$. 
So 
$v_n|_{E_0}$ is an eigenvector of $\hat T_{\lambda_n}$ with eigenvalue $\lambda_n^n$.
\item
For $\lambda >0$, the matrix $\hat T_\lambda$ is primitive. 
We then let $h(\lambda)$ be the leading eigenvalue of $\hat T_\lambda$, and $\hat v_\lambda$ denote the associated (normalized) eigenvector. 
\end{enumerate}

\begin{theorem} \label{thm:clone-eigen}
For $n$ large, 
$\lambda_n$ is the unique solution for $\lambda \in \R^+$ to
\begin{equation} \label{eq:h-lambda}
h(\lambda) = \lambda^n,
\end{equation}
and $v_n|_{E_0} = \hat v_{\lambda_n}$. 
Moreover, $\lambda_n \to 1$ and $\lambda_n^n \to h(1)$ as $n\to \infty$ (and also $v_n|_{E_0} \to \hat v_1$ by continuity). 
\end{theorem}

\begin{proof}
We observe that 
\begin{equation} \label{eq:lambda-ge-1}
\hat T_\lambda \le \lambda^N \hat T_1 \text{ for $\lambda \ge 1$}
\end{equation}
and 
\begin{equation} \label{eq:lambda-le-1}
\hat T_\lambda \ge \lambda^N \hat T_1 \text{ for $\lambda \le 1$}.
\end{equation}
Moreover
\begin{equation} \label{eq:h-1}
h(1) > 1
\end{equation}
because $\hat T_1$ is primitive and has integral coefficients.
By \eqref{eq:lambda-le-1} and \eqref{eq:h-1},
$h(\lambda) > \lambda^N \ge \lambda^n$ when $\lambda \in (0, 1]$ and $n > N$.
By \eqref{eq:lambda-ge-1}, 
$h(\lambda) \le h(1) \lambda^N < \lambda^n$ when $\lambda \in [2, \infty)$ and $n$ is large.
Moreover,
when $n$ is large,
$\frac{d}{d\lambda}\lambda^n$ is uniformly large (and positive) on $[1, 2]$, while $h'(\lambda)$ is bounded on that interval.
We conclude that there is a unique solution to \eqref{eq:h-lambda} in $\R^+$, and that it lies in $(1, 2)$. 

Moreover, 
letting $\alpha = 1/n$, 
we can recast \eqref{eq:h-lambda} as 
\begin{equation} \label{eq:h-alpha}
\alpha \log h(\lambda) = \log \lambda.
\end{equation}
Letting $g(\alpha, \lambda) = \alpha \log h(\lambda) -\log \lambda$, we see that $g(0, 1) = 0$, and $\dd_\alpha g(0, 1) = \log h(1)$ and $\dd_\lambda g(0, 1) = -1$.
Therefore, 
by the Implicit Function Theorem,
we can solve \eqref{eq:h-alpha} for $\lambda \equiv\lambda(\alpha)$ for all $\alpha$ near 0,
with $\lambda(0) = 1$ and $\lambda'(0) = \log h(1)$. 
Hence
$$
\lambda_n - 1 \sim \frac{\log h(1)}n
$$
and the Theorem follows.
\end{proof}
%
%
%
%


\section{The Wanderers Theorem}\label{app:wanderers}
Here we state and prove the Wanderers Theorem, 
which is similar in spirit to the Covering Lemma and the Quasi-Additivity Law from \cite{KL:QA}.
We state the theorem in a somewhat simplified form that will be sufficient for its application in this paper. 
\index{Wanderers Theorem}
\begin{theorem} \label{thm:wanderers}
Suppose $X$ and $Y$ are Riemann surfaces and $g\from X \to Y$ is proper and holomorphic, of degree $D$.
Suppose that $\BB$ is a finite set of disjoint compact connected subsets of $Y$ whose union contains the critical values of $g$.
Let $\AA$ be the set of components of $g\invp{\bigcup \BB}$,
and suppose $\hat \AA \subset \AA$. 
Suppose that $g\from \bigcup \hat \AA \to \bigcup \BB$ is at most $d$ to 1. 
Let $\FF$ be an oriented partial foliation of $X$ such that every leaf of $\FF$
\begin{enumerate}
\item
begins at an element of $\hat\AA$, 
\item
ends at an element of $\AA$,
\item
visits at least $n-1$ elements of $\AA$ in between, and
\item
does not cross through the component of $\hat \AA$ where it begins. 
\end{enumerate}
Then 
\begin{equation} \label{eq:wanderers}
\norm{\WW(\FF)} \le \frac{2d}n \norm{\wcan^{h+v}(Y \sm \bigcup \BB)} + C(|\chi(Y)| , |\BB|, D).
\end{equation}
\end{theorem}
Before proving this theorem,
we should recall the definition of the \emph{route} of a \emph{good path} from \cite{KL:QA}; 
we then have the following lemma, which is left to the reader:
\begin{lemma} \label{lem:work-backwards}
Suppose that
\begin{enumerate}
\item
 $C$,$C_1$, and $C_2$ are good paths in $X$ rel $\AA$,
\item
 $\rt{C_1} = \rt{C_2}$, 
\item
 $|\rt{C}| = |\rt{C_1}|$, 
 and 
\item
 $\delta_k(C) \subset Q_k(C_1, C_2)$. 
\end{enumerate}
Then $\rt C  = \rt{C_1}$. 
\end{lemma}

\begin{proof}
First suppose that $g$ is a (branched) Galois cover.
We can trim the leaves of $\FF$ at their terminal ends so that every resulting leaf is good (in the sense of \cite{KL:QA}) and has a route of length exactly $n$. We will still call the resulting foliation $\FF$; its width is at least as great as the original $\FF$.
We partition $\FF$ by route; the leaves with the same route form a parallel family, with a canonical linear order up to reversal. For each route $p$, we denote the corresponding subset of $\FF$ by $\FF_p$. Each $\FF_p$ (and indeed the whole of $\FF$) has a measure given by extremal width. We can form buffers of width 4 at both ends of each route class; we let $\tilde \FF_p$ be $\FF_p$ with these buffers removed, and denote the corresponding little rectangles, as in \cite{KL:QA}, by $\tilde Q_k(\FF_p)$. 

Let $p$ and $q$ be two routes, and suppose that, for some $\sigma \in \Deck(X/Y)$,  $\sigma Q_k(\FF_p) \cap Q_k(\FF_q) \neq \emptyset$ as oriented homotopy classes in $X \sm \bigcup \AA$.  As in \cite{KL:QA}, we can choose leaves in each of the two buffers of $\sigma\FF_p$ and $\FF_q$ (for four leaves in all), such that the four leaves are disjoint; we can then label three of these leaves as $C$, $C_1$, and $C_2$ such that they satisfy the hypotheses of Lemma \ref{lem:work-backwards}. By that Lemma, we can conclude that $\sigma p = q$. 

Because we have already cut out buffers,
each $\tilde Q_k(\FF_p)$ is mapped injectively by $g$ into the corresponding canonical rectangle for $Y \sm \bigcup \BB$. 
We pull back the canonical metric for the latter to each $Q_k(\FF_p)$; then each leaf of each $\tilde \FF_p$ has length at least $n$. 
If two of these $\tilde Q_k$'s map by $g$ to intersect \emph{and have the same orientation},
their associated routes are related by an element $\sigma$ of the deck group.
This latter relation is an equivalence relation,
and each equivalence class has at most $d$ elements, 
because each leaf of $\FF$ begins at an element of $\hat \AA$.
Thus $g$ is at most $2d$ to 1 on $\bigcup_p \tilde Q_k(\FF_p)$, and the pullback metric has area at most $2d \norm{\wcan(Y \sm \bigcup \BB}$. 
The Theorem then follows.

Now suppose that $g$ is not Galois. As in \cite{KL:QA}, we can form a branched cover $h\from \gc X \to X$ such that $\gc g := g \circ h$ is a Galois branched cover of degree at most $D!$. We let $\gc A = h^{-1}(A)$,  $\gc{\hat A} = h^{-1} \hat A$, and $\gc \FF = h^{-1} \hat \FF$. Then $W(\gc \FF) = (\deg h) W(\FF)$, and $\gc g\from \gc {\hat A} \to \BB$ is at most $d(\deg h)$ to 1.  Applying \eqref{eq:wanderers} to the Galois cover $\gc g$, we obtain
$$
\norm{\WW(\gc \FF)} \le \frac{2d(\deg h)}n \norm{\wcan^{h+v}(Y \sm \bigcup \BB)} + C(|\chi(Y)| , |\BB|, D!). \qedhere
$$
\end{proof}

\printindex

\printbibliography

\end{document}

%% file: macros/basic_macros.tex
\input{macros/paper_macros.tex}

%% file: macros/paper_macros.tex
\usepackage[margin=1.5in]{geometry}
\input{macros/packages}
\numberwithin{equation}{section}
\input{macros/theorem_types}

\input{macros/formula_macros}

%% file: macros/packages.tex
\usepackage{amsmath,amssymb}
\usepackage{amsthm}
\usepackage{graphicx}
\usepackage{hyperref}
\usepackage{mathtools}
\usepackage{verbatim}

\usepackage{tikz-cd}

\usepackage{enumitem}
\setlist{
leftmargin=2em}

%% file: macros/theorem_types.tex
\newtheorem{cor}[equation]{Corollary}
\newtheorem{lem}[equation]{Lemma}
\newtheorem{prop}[equation]{Proposition}
\newtheorem{theorem}[equation]{Theorem}
\newtheorem{lemma}[equation]{Lemma}
\newtheorem{proposition}[equation]{Proposition}

\theoremstyle{remark}
\newtheorem{rem}[equation]{Remark}

\theoremstyle{definition}

%% file: macros/formula_macros.tex
\newcommand{\inv}{^{-1}}
\newcommand{\invp}[1]{\inv(#1)}

\newcommand{\norm}[1]{\left\Vert #1 \right\Vert}
\newcommand{\pair}[2]{\left< #1, #2 \right>}

\newcommand{\setofst}[2]{\left\{ #1 \mid #2 \right\}}
\newcommand{\ceil}[1]{\lceil #1 \rceil}
\newcommand{\floor}[1]{\lfloor #1 \rfloor}

\newcommand{\largergiven}[1]{\ggg_{#1}}
\newcommand{\smallergiven}[1]{\lll_{#1}}

\newcommand{\largegiven}[1]{\largergiven{#1} 1}
\newcommand{\smallgiven}[1]{\smallergiven{#1} 1}

\newcommand{\from}{\colon}
\newcommand{\sm}{\setminus}
\newcommand{\thh}{^{\text{th}}}

\def\<#1>{\left< #1 \right>}
\newcommand{\dd}{\partial}

\newcommand{\B}{\mathbb B}
\newcommand{\C}{\mathbb C}
\newcommand{\D}{\mathbb D}
\newcommand{\E}{\mathbb E}

\newcommand{\N}{\mathbb N}

\newcommand{\Q}{\mathbb Q}
\newcommand{\R}{\mathbb R}

\newcommand{\Z}{\mathbb Z}

\renewcommand{\AA}{\mathcal A}
\newcommand{\BB}{\mathcal B}
\newcommand{\CC}{\mathcal C}

\newcommand{\FF}{\mathcal F}
\newcommand{\GG}{\mathcal G}
\newcommand{\HH}{\mathcal H}
\newcommand{\II}{\mathcal I}
\newcommand{\JJ}{\mathcal J}
\newcommand{\KK}{\mathcal K}
\newcommand{\LL}{\mathcal L}
\newcommand{\MM}{\mathcal M}
\newcommand{\NN}{\mathcal N}

\newcommand{\RR}{\mathcal R}

\newcommand{\TT}{\mathcal T}
\newcommand{\UU}{\mathcal U}
\newcommand{\VV}{\mathcal V}
\newcommand{\WW}{\mathcal W}
\newcommand{\XX}{\mathcal X}

\newcommand{\dD}{\dd \D}

\DeclareMathOperator{\mmod}{mod}

\DeclareMathOperator{\Deck}{Deck}

\DeclareMathOperator{\Int}{Int}

\DeclareMathOperator{\supp}{supp}

%% file: macros/annotation_macros.tex
\usepackage[textsize=small]{todonotes}
\newcommand{\jktodo}[2][inline]{\todo[color=red!40,#1]{#2}}

\newcommand{\jkheader}[1]{\smallskip\jktodo{#1}\kern-1em}

\newcommand{\ann}[1]
{\jktodo[]{#1}}

%% file: local_macros.tex
\usepackage{bm}
\usepackage{tikz}

\theoremstyle{theorem}
\newtheorem*{thmA}{Theorem A}
\newtheorem*{thmB}{Theorem B}
\newtheorem*{thmC}{Theorem C}
\newtheorem*{thmD}{Theorem D}

\newtheorem*{conjecture*}{Conjecture}
\newtheorem*{lemmaA}{Lemma}

\newcommand{\shrink}[2]{#1_{\left>#2\right<}}
\newcommand{\dcara}{\dd_C}
\newcommand{\pql}{$\psi$-ql }
\newcommand{\toto}{\rightrightarrows}

\newcommand{\FL}{\mathfrak{L}}

\newcommand{\biv}{{\mathrm{biv}}}
\newcommand{\br}{{\mathrm{br}}}
\newcommand{\nbr}{{\mathrm{ubr}}}
\newcommand{\ver}{{\mathrm{ver}}}
\newcommand{\hor}{{\mathrm{hor}}}

\newcommand{\str}{{\mathrm{str}}}

\newcommand{\out}{\operatorname{out}}
\newcommand{\bgg}{\mathbf G}
\newcommand{\inbr}{{\mathrm{in}}}

\newcommand{\outbr}{{\mathrm{out}}}
\newcommand{\amp}{{\mathrm{amp}}}

\newcommand{\etf}{{E(T_f)}}
\newcommand{\rt}[1]{\operatorname{rt}(#1)}

\newcommand{\wcan}{W_{\operatorname{can}}}
\newcommand{\fcan}{\FF_{\operatorname{can}}}

\newcommand{\dominates}{\multimap}

\newcommand{\hperp}{\HH^\biv}
\newcommand{\hv}{\HH^\ver}

\newcommand{\bound}{\mathbf b}

\newcommand{\wstd}{\WW(\RR F)}

\theoremstyle{definition}
\numberwithin{equation}{section}
\numberwithin{figure}{section}

\newcommand{\com} {{ \mathrm{com}}}
\newcommand{\thin} {{\mathrm{thin}}}
\newcommand{\thick} {{ \mathrm{thick}}}

\DeclareMathOperator{\flux}{flux}

\newcommand{\ddi}{\dd^i}

\DeclareMathOperator{\val}{Val}

\newcommand{\PF}{Perron-Fr\"obenius}

\newcommand{\FK}{\mathfrak{K}}

\newcommand{\per}{\mathbf{n}}

\makeatletter
\newtheorem{repeatlem@}{Lemma}

\makeatother

\makeatletter
\newtheorem{repeatprop@}{Proposition}

\makeatother

\newcommand{\good}{^g}
\newcommand{\bad}{^b}
\newcommand{\vg}{^{g\dominates}}

\newcommand{\gb}{^{gb\dominates}}
\newcommand{\bv}{^{b}}

\newcommand{\xhv}{X}

\newcommand{\jj}{\jmath}

\newcommand{\Card}{Card}

\renewcommand{\Card}{\cardioidsymbol}

\renewcommand{\Card}{\cardioidsymbolRightShifted}

\renewcommand{\Card}{\operatorname{Card}}

\newcommand{\gc}[1]{\breve{#1}}

\newcommand{\mathhyphen}{\text{-}}

%% file: misha_needed_macros
\newcommand{\hyp}{{\mathrm{hyp}}}
\newcommand{\cusp}{{\mathrm{cusp}}}

\def\BJ{{\mathbf{J}}}

\def\BJ{{\mathbf{J}}}

\newcommand{\BI}{{\mathbf{I}}}
\newcommand{\BL}{\mathbf{L}}
\def\B0{{\mathbf{0}}}

\newcommand{\BPi}{{\boldsymbol{\Pi}}}

\newcommand{\Bde}{{\boldsymbol{\delta}}}

\newcommand{\eps}{{\epsilon}}

\newcommand{\ra}{\rightarrow}

\newcommand{\bg}{{\boldsymbol{\gamma}}}

\newcommand{\ol}{\overline}

\newcommand{\lo}{{\mathrm{lo}}}

\newcommand{\sh}{{\mathrm{sh}}}

\newcommand{\can}{\operatorname{can}}

%% file: figs/tails.eps_tex
\begingroup%
  \makeatletter%
  \providecommand\color[2][]{%
    \errmessage{(Inkscape) Color is used for the text in Inkscape, but the package 'color.sty' is not loaded}%
    \renewcommand\color[2][]{}%
  }%
  \providecommand\transparent[1]{%
    \errmessage{(Inkscape) Transparency is used (non-zero) for the text in Inkscape, but the package 'transparent.sty' is not loaded}%
    \renewcommand\transparent[1]{}%
  }%
  \providecommand\rotatebox[2]{#2}%
  \newcommand*\fsize{\dimexpr\f@size pt\relax}%
  \newcommand*\lineheight[1]{\fontsize{\fsize}{#1\fsize}\selectfont}%
  \ifx\svgwidth\undefined%
    \setlength{\unitlength}{2033.82147217bp}%
    \ifx\svgscale\undefined%
      \relax%
    \else%
      \setlength{\unitlength}{\unitlength * \real{\svgscale}}%
    \fi%
  \else%
    \setlength{\unitlength}{\svgwidth}%
  \fi%
  \global\let\svgwidth\undefined%
  \global\let\svgscale\undefined%
  \makeatother%
  \begin{picture}(1,0.78056899)%
    \lineheight{1}%
    \setlength\tabcolsep{0pt}%
    \put(0,0){\includegraphics[width=\unitlength]{tails.eps}}%
  \end{picture}%
\endgroup%

%% file: figs/H_AD.eps_tex
\begingroup%
  \makeatletter%
  \providecommand\color[2][]{%
    \errmessage{(Inkscape) Color is used for the text in Inkscape, but the package 'color.sty' is not loaded}%
    \renewcommand\color[2][]{}%
  }%
  \providecommand\transparent[1]{%
    \errmessage{(Inkscape) Transparency is used (non-zero) for the text in Inkscape, but the package 'transparent.sty' is not loaded}%
    \renewcommand\transparent[1]{}%
  }%
  \providecommand\rotatebox[2]{#2}%
  \newcommand*\fsize{\dimexpr\f@size pt\relax}%
  \newcommand*\lineheight[1]{\fontsize{\fsize}{#1\fsize}\selectfont}%
  \ifx\svgwidth\undefined%
    \setlength{\unitlength}{494.97867614bp}%
    \ifx\svgscale\undefined%
      \relax%
    \else%
      \setlength{\unitlength}{\unitlength * \real{\svgscale}}%
    \fi%
  \else%
    \setlength{\unitlength}{\svgwidth}%
  \fi%
  \global\let\svgwidth\undefined%
  \global\let\svgscale\undefined%
  \makeatother%
  \begin{picture}(1,0.66053768)%
    \lineheight{1}%
    \setlength\tabcolsep{0pt}%
    \put(0,0){\includegraphics[width=\unitlength]{H_AD.eps}}%
  \end{picture}%
\endgroup%

%% file: value1.tex
Let  $F$ be a \pql map as in the previous subsection. The results in this section are purely combinatorial, so for simplicity may imagine that $F$ is a quadratic polynomial, so in particular $\iota$ is an inclusion (see \cite[Section 4.1]{K06} for details of transferring combinatorial data from \pql maps to quadratic polynomials).

We recall the definition of \emph{domination} from Appendix \ref{sec:domination}. Let us say that a WAD diagram on $\UU\sm \KK$ is \emph{horizontal} if it is supported on horizontal arcs. 
Let $\val$ be a function that assigns a non-negative real number to every WAD on 
$\UU^n\sm \KK^n$ with $n\geq 0$. If 
\begin{enumerate}
    \item $\val(X)>0$ for any horizontal WAD $X$;
    \item for any WADs $X$ on $\UU^n\sm \KK^n$ and $Y$ on $\UU\sm \KK$, if $X\dominates Y$ then $\val(X)\geq \val(Y)$; 
    \item there is some $\lambda>1$ so that $\val((f^*)^nX)\leq \lambda^{-1}\val(X)$ for every WAD $X$ on $\UU\sm \KK$ and every $n\geq 0$;
\end{enumerate}
then we will call $\val$ an \emph{$\lambda$-expanding valuation} \index{weighted arc diagram (WAD)!expanding valuation} for $F$. 
Recalling the core eigenvalue $\lambda_F$ from Section \ref{sec:hubbard trees}, we observe that a $\lambda_F$-expanding valuation always exists:

\begin{theorem}\label{thm:value}
    If $F$ is a primitively renormalizable and periodically repelling \pql-map with core eigenvalue $\lambda_F$, then there is a $\lambda_F$-expanding valuation for $F$.
\end{theorem}

Before proving Theorem \ref{thm:value}, we need the following two arithmetic facts:
\begin{lemma}
	Suppose that $a_i, b_i, c_i \ge 0$ for $i = 1 \ldots n$,
	and,
	for each $i$,
	\begin{equation} \label{eq:upper}
		a_i^2 + b_i^2 \le c_i^2.
	\end{equation}
	Then 
	\begin{equation} \label{eq:combined}
		\left(\sum_{i = 1}^n a_i\right)^2 + \left(\sum_{i = 1}^n b_i\right)^2 \le \left(\sum_{i = 1}^n c_i\right)^2.
	\end{equation}
\end{lemma}
\begin{proof}
	By Cauchy-Schwarz and \eqref{eq:upper},
	for all $i$ and $j$,
	\begin{equation} \label{eq:cs}
		a_i a_j + b_i b_j \le \sqrt{(a_i^2 + b_i^2)(a_j^2 + b_j^2)} \le c_i c_j.
	\end{equation}
	Adding this up over all $i$ and $j$, we obtain \eqref{eq:combined}.
\end{proof}

\begin{lemma}\label{lem:harmonic and geometric sum}
    For any sequences $(a_j)$ and $(w_j)$ of nonnegative real numbers, 
    \begin{equation}\label{eq:harmonic sum bound}
        \left(\bigoplus w_j\right)\left(\sum a_j\right)^2\leq \sum w_ja_j^2. 
    \end{equation}
\end{lemma}
\begin{proof}
    See \cite[Lemma 10.1]{K06}.
\end{proof}

\begin{proof}[Proof of Theorem \ref{thm:value}]
    We recall the linear  transformations $T_F: \R^{E(\TT)}\to \R^{E(\TT)}$ with eigenvalue $\lambda_F$ from Section \ref{sec:hubbard trees}. We have a natural bijection between $E(\TT)$ and $\HH^\biv$, so we can view any non-negative element of $\R^{E(\TT)}$ as a WAD with support in $\HH^\biv$. From this viewpoint, $T_F$ and its transpose $T_F^\top$ are exactly the linear transformation generated by 
    $$\beta\mapsto \sum_{\alpha\in \iota^*\beta} f_*\alpha \quad \text{ and } \quad \beta\mapsto \sum_{\alpha \in f^*\beta}\iota_*\alpha$$
    respectively.  {(Informally speaking, the former summation is taken over all $f$-push-forwards of $\beta$ (rel $\TT$),
    while the latter is taken over all $f$-pullbacks (rel $\TT$).}
    Additionally, 
    as $f^n\circ \iota_n^{-1}$ and $\iota_n\circ f^{-n}$ induce the linear transformations $T_F^n$ and $(T_F^\top)^n$ respectively, $f^n_*\iota_n^*\beta = (f_*\iota^*)^n\beta$ for all $\beta\in \HH^\biv$ and $n\geq 0$.

    Let $v$ and $w$ be WADs corresponding to eigenvectors with eigenvalue $\lambda_F$ of $T_F$ and $T_F^\top$ respectively and fix some $n\geq 0$. Thus 
    \begin{align*}
        \lambda_F^nv(\beta) =\sum_{\alpha\in \iota_n^*\beta}v(f^n_*\alpha) \text{ and }
        \lambda_F^nw(\beta) =\sum_{\alpha\in (f^n)^*\beta}w((\iota_n)_*\alpha)
    \end{align*}
    for all $\beta\in \HH^\biv$.
   {For any $\beta\in \HH^\biv$ and {$\alpha\in \iota_n^*(\HH^\biv)$} 
    we set 
    \begin{equation*}
    \l(\beta)= v(\beta)\sqrt{w(\beta)},\quad 
    \l(\alpha) = \lambda_F^{-n} \, v(f_*^n\alpha)\sqrt{w((i_n)_*\alpha)}.
    \end{equation*}} 
    Thus 
\begin{align} \label{eq:basic-length}
	\sum_{\alpha \in \iota_n^*\beta} l(\alpha) = l(\beta)
\end{align}
and 
\begin{equation} \label{eq:basic-value}
	\sum_{{\alpha \in f_*^n \beta}} l(\alpha)^2 = \lambda_F^{-n} l(\beta)^2.
\end{equation}

For any $n\geq 0$ and arc $\gamma$ in $\UU^n\sm \KK^n$, we set $\l(\gamma) = \sum_{\alpha\in \iota_n^*\HH^\biv} \pair{\alpha}{\gamma}\l(\alpha)$.
For any arc $\beta$ in $\UU\sm \KK$, we observe that
\begin{equation} \label{eq:length}
	 \sum_{\alpha\in \iota_n^*\beta} l(\alpha) \ge l(\beta)
\end{equation}
and
\begin{equation} \label{eq:value}
	\sum_{\alpha\in (f^n)^*\beta} l(\alpha)^2 \le \lambda_F^{-n} l(\beta)^2. 
\end{equation}
Indeed, the inequality \eqref{eq:length} follows easily from \eqref{eq:basic-length},
and \eqref{eq:value} follows from \eqref{eq:basic-value} and the following:

Now, we define the {valuation} $\val$ by $$\val(X) = \sum_\alpha X(\alpha) l(\alpha)^2.$$
We then observe that,
because of \eqref{eq:length} and \eqref{eq:harmonic sum bound}, 
\begin{equation}
	X \dominates Y \implies \val(X) \ge \val(Y).
\end{equation}
Moreover, 
by \eqref{eq:value},
$\val((f^n)^*X) \le \lambda_F^{-n} \val(X)$.
\end{proof}

Moreover, our proof of Theorem \ref{thm:value} allows us to estimate the value of a WAD in terms of its size:
\begin{prop}\label{prop:value estimate}
    Let $v$ and $w$ be eigenweights of $T_F$ and $T_F^\top$ respectively.
    We can choose the valuation in Theorem \ref{thm:value} so that if $X$ is a WAD on $\UU\sm \KK$ and $\langle \alpha, \HH^\biv\rangle\leq C$ for all $\alpha\in \supp X,$
    then
    $$u_{\min}^2\|X\|\leq V(X)\leq C^2u_{\max}^2\|X\|,$$
    where $u_{\min}$ and $u_{\max}$ are the maximum and minimum values of $v(\alpha)\sqrt{w(\alpha)}$ over all arcs $\alpha\in \HH^\biv$ that intersect the support of $X$.
\end{prop}

%% file: not-all-good2.tex
\subsection{Geometry, Combinatorics, and a Contradiction} \label{align:ggc}

In this subsection we will state a geometric proposition and a combinatorial one, and immediately conclude Lemma \ref{lem:hammer}. We let $G = (g, \jj) = (f^\per, \iota_\per)$ be the $\per\thh$ iterate of $F$. 

We say that $X = \wcan^{h+v}(\UU \sm \KK)$ is \emph{$(L, \epsilon)$-well-controlled} \index{weighted arc diagram (WAD)!well-controlled} if there is a good set of arcs in the translation region, each of combinatorial length at most $L$, that have $1 - \epsilon$ of the total weight of $X$.  In this case we say that an arc in $\supp X$ is good if belongs to the good set of arcs, and that it is \emph{bad} if it does not.  We let $X^g$ denote restriction of $X$ to the good arcs. 

For any interval $I\subset \R$ and  WAD $Z$ on $\UU \sm \KK$, we let $Z|_I$ be the restriction of $Z$ to the arcs that lie entirely in the subsurface $U_I$. We will say that $I$ lies in the translation region when $I\subset (\bound, q-\bound)$.

Here is our geometric proposition. It says that when $X$ is well-controlled, there is an interval $I$ in the translation region such that $X^g|I$ is nearly self-dominating. 
\begin{prop}  \label{prop:interval-self-dominating}
For all $\delta> 0$:
Suppose that $X$ is $(L, \epsilon)$-well-controlled
for $\epsilon \smallgiven{\bound, \delta}$ and $q, \WW(\RR F) \largegiven{\bound, \delta, L}$.
Then there is an interval $I$ in the translation region and a WAD $Y \le X^g|_I$ such that 
\begin{equation} \label{eq:interval-self-dominating-dominated}
g^*(X^g|_I) \dominates Y
\end{equation}
and 
\begin{equation} \label{eq:interval-self-dominating-control}
\norm{X^g|_I - Y} \le \delta \norm{X^g|_I}.
\end{equation}
\end{prop}

We will prove Proposition \ref{prop:interval-self-dominating} in Sections \ref{align:interval} through \ref{align:immersion}.

Here is our complementary combinatorial proposition, which asserts that there can be no nearly self-dominating horizontal WAD with combinatorially bounded support. 
\begin{prop} \label{prop:no-near-self-dominating}
For $\delta \smallgiven{L, \bound}$,
there are no WAD's $Z$ and $Y$ with $Y \le Z$ such that 
$g^*Z \dominates Y$ and $\norm{Z - Y} \le \delta \norm{Z}$,
and for all $\alpha \in \supp Z$,
\begin{itemize}
\item
$\alpha$ is horizontal,
\item
$\alpha$ lies in translation region, and
\item
the combinatorial length of $\alpha$ is at most $L$.
\end{itemize}
\end{prop}

We prove	Proposition \ref{prop:no-near-self-dominating} in Section \ref{align:value}, 
after preparation in Section \ref{align:Hubbard-weights}.

Assuming Propositions \ref{prop:interval-self-dominating} and \ref{prop:no-near-self-dominating},
we can immediately prove Lemma \ref{lem:hammer}.

\begin{proof}[Proof of Lemma \ref{lem:hammer}]
Given $L$, 
we choose a $\delta \smallgiven{L, \bound}$ as in Proposition \ref{prop:no-near-self-dominating},
and $\epsilon \smallgiven{\bound, \delta}$ as in Proposition \ref{prop:interval-self-dominating}.
Given $f$ that is $(\epsilon, L)$-well-controlled,
we  let $I$ and $Y$ be then as given in Proposition \ref{prop:interval-self-dominating}.
Then $X^g|_I$ and $Y$ satisfy the hypotheses for $Z$ and $Y$, 
respectively,
in Proposition \ref{prop:no-near-self-dominating}.
This is a contradiction.
\end{proof}

\subsection{The eigenvalues and eigenweights of our Hubbard trees} \label{align:Hubbard-weights}
We now recall the the concepts of cross sections and cloning from  Appendix \ref{app:clones}
and the linear transformation $T_F$ from Section \ref{sec:hubbard trees}; we will use the results there to bound the eigenvalue $\lambda_F^\per$ of $T_F^\per$ and corresponding eigenweights of $\TT$.
We will use the bound on the eigenweights in Section \ref{align:value} to prove Proposition \ref{prop:no-near-self-dominating}.
Along the way we prove some combinatorial statements that will be used in Section \ref{align:interval}. 

\begin{lemma} \label{lem:get-around}
There is some $N = N(\bound)$ such that, when $q \largegiven{\bound}$,  $$\LL_0 \subset \bigcup_{i=1}^{N-1} f^{-i}\left(\overline{\LL_{[1, \bound]}}\right).$$ 
\end{lemma}
\begin{proof}
Recall that $K_0$ denotes the central small Julia set, and $K_i = f^i(K_0)$.
Let $\alpha_F'$ denote the second pre-image of $\alpha_F$, and let $\LL_{(0, q)}'$ denote the component of $F^{-1}(\LL_{(0, q)})$ containing $\alpha_F'$. We set $C=F^{-1}(\LL_1)$.
For simplicity we assume that $\FK$ is locally connected, the general case can be handled similarly.

For $K_i \subset \LL_{(0, q)}$, let $\gamma_i$ be the path in $\TT$ from $K_i$ to $\alpha_F$.
For $K_i \subset \LL_{(0, q)}'$, let $\gamma_i$ be the path in $\TT$ from $K_i$ to $\alpha'$.
For $K_i\subset C$, let $\gamma_i = \emptyset$. 
Thus $f(\gamma_i) \subset C \cup \gamma_{i+1}$ for all $i$.
Therefore, by induction,
for any $z \in \gamma_i$,
we have 
\begin{equation} \label{eq:get-around-1}
\gamma_i \subset F^{-k}(\gamma_{i+k}) \cup \bigcup_{j=1}^{k} F^{-j}(C)\subset F^{-k}(\gamma_{i+k}) \cup \bigcup_{j=2}^{k+1} F^{-j}(\LL_1)
\end{equation}
for all $k \ge 1$. 

For some set $I_0\subset \Z/\per\Z$ of cardinality at most $\bound$, we can write 
$$\LL_0 \subset C\cup \bigcup_{i \in I_0} K_i \cup \gamma_i.$$
We leave it as an exercise for the reader to show that, 
for each $i \in I_0$,
there is an $n_i \le \bound^2$ such that $F^{n_i}(K_i) \subset L_1 \cup \ldots \cup L_{\bound}$.
The lemma then follows from \eqref{eq:get-around-1},  taking $N = \bound^2 + 1$.
\end{proof}

\begin{cor} \label{cor:cross-section}
Suppose that $F$ is parabolically $\bound$-bounded, 
with rotation number $1/(2 \bound + 2)$. 
Then $E(\TT) \cap \LL_{b+1}$ is a cross section for the push-forward $T_F$, for paths of length $N = N(\bound)$. 
\end{cor}
\begin{proof}
We take $N = \bound^2 + 2 \bound + 3$. 
For any point $z \in \TT$, 
$w:=F^k(z) \in \LL_0$ for some $k \le 2 \bound + 1$,
and,
by Lemma \ref{lem:get-around},
$F^l(w) \in \LL_1 \cup \ldots \cup \LL_{\bound}$ for some $l \le \bound^2 + 1$. 
Therefore $F^m(z) \in L_{\bound + 1}$ for some $m \le \bound^2  + 2\bound + 2 < N$. 
\end{proof}

\begin{cor} \label{cor:bounded-gamma}
For suitable $N = N(\bound)$, 
and each $m \in  (N+b + 1, q-N-b-1)$, 
we have $f^\per(\LL_m) \subset \LL_{[m-N, m+N]}$. 
\end{cor}

\begin{prop} \label{prop:eigenweights}
{Suppose that $F$ is parabolically $\bound$-bounded. 
Let $v$ and $w$ be eigenweights for $T_F$ and $T_F^\top$ respectively with eigenvalue $\lambda_F$.
For any $\alpha, \beta \in E(\TT)$ that lie in the translation region, 
$v_f(\alpha)/v_f(\beta) \le C(\bound)$ and $w_f(\alpha)/w_f(\beta) \le C(\bound)$. Moreover, $2 \le \lambda_f^\per \le C(\bound)$.}
\end{prop}

\begin{proof}
There are finitely many parabolically $\bound$-bounded renormalization types with $q = 2 \bound + 2$;
by Lemma \ref{lem:get-around},
$\TT \cap L_{\bound + 1}$ is a cross section for the push-forward map on $E(\TT)$;
the dynamics on $\TT$ for any $q > 2\bound +2$ is then the result of cloning of $\TT \cap L_{\bound + 1}$. 
Therefore,
by Theorem \ref{thm:clone-eigen},
$\lambda_F^{q - (2 \bound + 2)}$ converges, and hence is bounded, for each series of clones,
and therefore $\lambda_F^\per$ is bounded as well, since $\per \le \bound q$. 
On the other hand,
$\lambda_F^\per \ge 2$ for any prime and primitive renormalization type,
because each edge $e$ of $\TT$ is mapped by $F^p$ to a sequence of edges in $\TT$ with the same edges as $e$, and the sequence of edges cannot be just $e$, so therefore it has length at least 2.

Finally,
for $\alpha, \beta \in \TT \cap L_{\bound+1}$, 
we have $v(\alpha)/v(\beta)$ converges as $q \to \infty$ in each sequence of clones, and hence it is bounded.
On the other hand if $\alpha, \beta$ lie in the translation region, and are clones of the same edge, 
then $v(\alpha)/v(\beta) = \lambda_F^k$ for some $|k| \le q$, and is hence bounded.
The same applies to $w(\alpha)/w(\beta)$,
because the transpose of the clones are the clones of the transpose.
\end{proof}

\subsection{No near self-domination} \label{align:value}

We can now use the results of Section \ref{align:Hubbard-weights} and the theory of values from Section \ref{sec:value} to prove Proposition \ref{prop:no-near-self-dominating}.
\begin{proof}[Proof of Proposition \ref{prop:no-near-self-dominating}]

By Theorem \ref{thm:value}, there is an $\lambda_F$-expanding value $\val$ for $F$.

Suppose that $Z$ and $Y$ satisfy the hypotheses of the Proposition.
As $\lambda_F^\per \ge 2$, we have
\begin{equation} \label{eq:vals}
\frac12 \val(Z) \ge \val(g^*Z)\ge \val(Y).
\end{equation}

On the other hand,
our definition of combinatorial length and Lemma \ref{lem:compare} imply

$\pair{\alpha}{\hperp}
 \leq 2 \bound L$ 
for all $\alpha \in \supp Z$. 
Hence,  by Proposition \ref{prop:value estimate}, we can choose $\val$ so that
\begin{align*}
\val(Z - Y) 
&\le 4 \bound^2 L^2 u_{\max}^2 \norm{Z - Y} \\
&\le 4 \delta  \bound^2 L^2  u_{\max}^2 \norm{Z} \\
& \le 4\delta  \bound^2 L^2 \frac{u_{\max}^2}{u_{\min}^2} \val(Z).
\end{align*}
where  $v$ and $w$ are eigenweights for $T_F$ and $T_F^\top$ with eigenvalue $\lambda_F$ and $u_{\min}$ and $u_{\max}$ are minimum and maximum respectively of $v(\alpha)\sqrt{w(\alpha)}$ over all $\alpha$ in the translation region.
By Proposition \ref{prop:eigenweights},
\begin{equation}
 \frac{u_{\max}^2}{u_{\min}^2} \le \max_{\alpha, \beta} \frac{v_f(\alpha)^2 w_f(\alpha)}{v_f(\beta)^2 w_f(\beta)} \le C(\bound).
\end{equation}
When
\begin{equation} \label{eq:delta-control}
4 \delta \bound^2 L^2 \frac{u_{\max}^2}{u_{\min}^2} < \frac12, 
\end{equation}
we obtain $\val(Z - Y) < \frac12 \val(Z)$, a contradiction to \eqref{eq:vals}.
\end{proof}

\subsection{The subsurface and WAD for an interval} \label{align:interval}
We now turn to the proof of Proposition \ref{prop:interval-self-dominating}, which will occupy the rest of this section.
Because the full return map $g$ has uncontrolled degree, 
we need to lift $g$ to a subsurface cover (defined later in this subsection) where we can then control the degree. 
In this subsection, 
we describe how to effectively restrict $g$ to an interval and then how to lift that restriction to the subsurface cover. 
We begin this subsection with a combinatorial discussion where we can think of $f$ (and hence $g$) as a quadratic-like map, or even a quadratic polynomial, as in Section \ref{align:Hubbard-weights}. We then move to a more geometric discussion, where we are compelled to consider our given $\psi$-quadratic-like map $F$ and its $\per\thh$ iterate $G = (g, \jj)$. 

For any interval $I\subset \R$, we recall from Section \ref{sec:limbs}
that $U_I$ is a neighborhood of the limbs $\FL_I$ up to isotopy rel $\KK$. 
We observe that,
for any subsurface $U$ of $\UU$ defined rel $\KK$ (and whose boundary is disjoint from $\KK$), 
any component of $g\invp U$ is defined rel $g\invp \KK$, and hence defined rel $\KK$. 
This then determines the combinatorics of the lifts to the associated covers. 

If $I\subset \R$ is an interval of  length at least $2r+1$, let $\shrink I r$ denote the interval with the intervals of length $r$ on each side of $I$ removed; 
if the length of $I$ is smaller we let $\shrink I r$ be the empty interval.
If $K_i$ is a small Julia set, and $K_i \subset L_m$, we let $\NN_{r, s}(K_i) = \bigcup_{j=m-r}^{m+s} L_j$. 
For any subset $Z$ of a Riemann surface (or any Hausdorff topological space),
we let $Z|_a$ be the component of $a \in Z$ that contains $a$;
we have the same definition for $Z|_A$, where $A\subset Z$ is connected. 

We begin with the following two combinatorial lemmas.
\begin{lemma}
There is a component $\hat U_I$ of $g\invp {U_I}$ that contains $U_{\shrink I r}$, for $r = r(\bound )$. 
We also have $\hat U_I \subset U_I$ (in the sense of subsurfaces rel $\KK$). 
\end{lemma}

\begin{proof}
Let $r$ be the $N=N(\bound)$ defined as in Corollary \ref{cor:bounded-gamma}.
Then by that Corollary,
$g(\TT \cap U_{\shrink I r}) \subset U_I$. 
Hence there is a component $\hat U_I$ of $g\invp{U_I}$ that contains $\TT \cap U_{\shrink I r}$; 
up to isotopy rel $\KK$, this component contains $U_{\shrink I r}$. 

On the other hand, we observe that, for any $K_i$,
$$ f\invp{\NN_{r, s}(K_i)}|_{K_{i-1}} \subset \NN_{r, s}(K_{i-1}).$$
It follows by induction that $f^{-j}(\NN_{r, s}(K_i))|_{K_{i-j}} \subset \NN_{r, s}(K_{i-j})$, 
and hence that $f^{-\per}(\NN_{r, s}(K_i))|_{K_i} \subset \NN_{r, s}(K_i)$. 
This implies that $\hat U_I \subset U_I$. 
\end{proof}

\begin{lemma} \label{lem:pullback-degree}
Suppose $f$ is an $\per$-renormalizable quadratic polynomial (or ($\psi$-)quadratic-like) map. Suppose that $U$ is a simply connected subsurface of the range of $f$ that is simply connected such that $\dd U$ is disjoint from $\KK$, and $V$ is any component of $f^{-\per}(U)$. Then the degree of $f^\per|_V$ is at most $2^m$, where $m$ is the number of components of $\KK \cap U$.
\end{lemma}
\begin{proof}
We let $V_0 = V$ and inductively let $V_{i+1} = f(V_i)$, so $V_\per = U$. 
Then each $V_i$ is connected, so $f\from V_i \to V_{i+1}$ has degree 1 or 2, with the latter occurring if and only if the critical point $c$ of $f$ lies in $V_i \cap \KK$. 
If it does,
then $f^{\per - i}(c)$ lies in $\KK \cap U$,
and each such $i$
requires a different component of $\KK \cap U$. 
\end{proof}

If $Q$ is a surface and $H \subset Q$ is a subsurface, 
we say the \emph{subsurface cover} of $Q$ for $H$ is the cover corresponding to $i_*\pi_1(H)$, 
where $i\from H \to Q$ is the inclusion. 
We can lift $i$ to a map $\hat \imath$ from $H$ to the subsurface cover,
and every end of $H$ that is proper for $i$ will be proper for $\hat \imath$. 
In particular every arc in $H$ has a natural lift to the subsurface cover.
Any continuous map $h\from (\hat Q, \hat H) \to (Q, H)$ then lifts to a map,
from the subsurface cover of $\hat Q$ for $\hat H$,
to the subsurface cover for $Q$ for $H$,
and immersions lift to immersions.

Recalling the definition of localization from Section \ref{sec:localization}, we let $V_I$ be the localization for $U_I$ in $\UU\sm \KK$ with all the components of $\KK\cap U_I$ filled in; we also denote by $\KK_I$ the copy of $\KK_I$ in $V_I$. 
We likewise define $\hat V_I$ 
to be the localization of $\hat U_I$ in $g^{-1}(\UU\sm \KK)$ with all the components of $g^{-1}(\KK)\cap \hat U_I$ filled in; we also denote by $g^{-1}(\KK_I)$ the copy of $g^{-1}(\KK_I)$ inside $\hat V_I$. We let $\hat\KK_I$ denote the components of $g\inv(\KK_I) \cap \hat V_I$ 
that correspond to components of $\KK$.
We let $X_I = \wcan^{h+v}(V_I \sm \KK_I)$ and $\hat X_I = \wcan^{h+v}(\hat V_I \sm \hat\KK_I)$.

We then have 
\begin{itemize}
\item
a branched cover $g\from \hat V_I \to V_I$ with degree bounded in terms of $\bound$ and $|I|$,
\item
an inclusion
\begin{equation} \label{eq:inclusion}
\hat V_I \sm g\inv\KK_I \subset \hat V_I \sm \hat \KK_I,
\end{equation}
and 
\item
the immersion $\jj\from \UU^\per \to \UU$ 
restricts to an immersion $\jj\from \UU^\per \sm g\inv\KK \to \UU\sm \KK$, 
which lifts, 
via the inclusion $ \hat U_I \subset U_I$,
to an immersion 
\begin{equation} \label{eq:immersion}
i \from \hat V_I \sm \hat \KK_I \to V_I \sm \KK_I
\end{equation}
that is proper on the ends of $\hat V_I \sm \hat \KK_I$ corresponding to  $\hat\KK_I$. 
\end{itemize}
We will study the effect of \eqref{eq:inclusion} in Section \ref{align:good-by-good} and the effect of \eqref{eq:immersion} in Section \ref{align:immersion}, where
we then connect the relations in the covers $\hat V_I$ and $V_I$ to those in $\UU \sm \KK$ and complete the proof of Proposition \ref{prop:interval-self-dominating}.

\subsection{Typical intervals in the translation region} \label{align:typical}
\emph{We assume for the rest of this section that $X$ is $(L, \epsilon)$-well controlled}.

In this subsection we define a ``typical'' interval in the translation region, and show that they exist. As the name suggests, these are intervals $I$ where anomalies that are rare in the translation region are rare in $I$.

We fix $I$ and let $X\good_I \le X_I$ denote the portion of $X_I$ that is supported on the lift of good arcs, of combinatorial length at most $L$,
\emph{that lie entirely in $U_I$},
and let $X\bad_I := X_I - X\good_I$. We observe that every good arc is horizontal, so $X\bad_I$ contains all the vertical arcs in $X_I$. 
\begin{lemma} \label{lem:I-typical}
For  $L_1 \largegiven{\epsilon, L}$ and $\WW(\RR F) \largegiven \epsilon$, we can choose $I$ in the translation region with $|I| = L_1$,
such that 
\begin{equation} \label{eq:I-typical}
\norm{X\bad_I} <  3\epsilon L_1 \WW(\RR F).
\end{equation}
\end{lemma}
\begin{proof}
For a given interval $I$ of length $L_1$, there are three kinds of contributions to $X\bad_I$:
\begin{enumerate}
\item \label{it:bad}
Weighted arcs in $X^b$ that have at least one endpoint in $U_I$,
\item \label{it:outside}
arcs of $X^g$ that go from $U_I$ to outside $U_I$, and
\item \label{it:ghost}
``ghost vertical'', in the sense of weight of $X^v_I$ that does not appear in $X$. 
This has an upper bound of 
\begin{equation}
\sum_{K_i \subset U_I} W(K_i) - X|_{K_i}
\end{equation}
where $X|_{K_i}$ is the (weighted) intersection number of $X$ with a curve going around $K_i$ (as in \cite{K06}).	
\end{enumerate}
The contribution from \ref{it:bad} is, \emph{on average}, at most $\epsilon L_1\wstd$.
Case \ref{it:outside} contributes at most $4L\wstd$, 
because now our good arcs have combinatorial length at most $L$, 
so only those that start outside of $\shrink I L$ can go outside of $I$.
Finally,
the terms $W(K_i) - X|_{K_i}$ are each universally bounded, \emph{on average}, so the contribution from \ref{it:ghost} is $O(L_1)$ for a typical interval. 

For $q > L_1$, we can find a set of disjoint intervals of length $L_1$ that together cover at least half of the translation interval. 
Then one of these intervals will be typical in the sense that the total from Cases \ref{it:bad} and \ref{it:ghost} will be no more than twice the  average amounts stated above. 

Taking $L_1 \largergiven{} L/\epsilon$ and $\wstd \largergiven{} 1/\epsilon$, we obtain \eqref{eq:I-typical}.
\end{proof}

We will call an interval $I$ that satisfies \eqref{eq:I-typical} a \emph{typical} interval.
We observe that for any interval, $\norm{X_I} \asymp |I| \wstd$, and $\norm{X^g_I} \asymp \norm{X_I}$ when $X$ is $\epsilon, L$-controlled.

\subsection{Good arcs are mostly dominated by good arcs} \label{align:good-by-good}
In this subsection,
which is the heart of the proof of Proposition \ref{prop:interval-self-dominating},
 we study the inclusion \eqref{eq:inclusion} and show that most of $X_I$ is dominated by $g^*X^g_I$. 
\begin{prop} \label{prop:most-well-dominated}
For $\epsilon \smallgiven{\bound, \delta}$,
when $I$ is a typical interval of length $L_1$,
we can find $\hat Y_I \le \hat X_I$ with 
\begin{equation} \label{eq:most-well-dominated}
\norm{\hat X_I - \hat Y_I} \le \delta\norm{X_I}
\end{equation}
such that 
\begin{equation} \label{eq:most-well-dominated-dominated}
g^*X^g_I \dominates \hat Y_I
\end{equation}
in the inclusion \eqref{eq:inclusion}.
\end{prop}
\begin{proof}
From \cite{K06},
we have $g^*\wcan(V_I \sm \KK_I)\dominates\wcan(\hat V_I \sm \hat\KK_I) - C(L_1)$.
Moreover horizontal arcs can only be realized as a concatenation of horizontal arcs,
and vertical as a concatenation of horizontal and vertical. 
So we have
$g^*X_I \dominates \hat X_I - C(L_1)$. 
For the sake of the next paragraph we orient the arcs of $X^h_I$ arbitrarily,
and we orient the arcs of $X^v_I$ from $\KK_I$ to $\dd V_I$. 

As a witness to this domination,
we can find sequences $(Y_i)$ and $(Z_i)$ of WAD's (with the same index set), such that
\begin{equation} \label{eq:witness}
g^*X_I \ge \sum_i Y_i \text{ and } \hat X_I = \sum_i Z_i
\end{equation}
and $Y_i = \sum_j w_{ij} \alpha_{ij}$ and $Z_i = v_i \beta_i$,
and $Y_i \dominates Z_i$ in the elementary sense that $\sum_j v_{ij}^{-1} \le w_i^{-1}$ and $(\alpha_{ij})_j \to \beta$. 
It is important in this setting to require that the $\alpha_{ij}$'s appear 
\emph{in order along the associated representative of $\beta_i$} in $(\alpha_{ij})_j$.

We let $\hat X\vg_I$ be the portion of $\hat X_I$ that is dominated by $g^*X\good_I$,
so $X\vg_I = \sum_{Y_i \le g^*{X\good_I}} Z_i$. 
We let $\hat X\gb_I = \hat X_I - \hat X\vg_I$. 
(The superscript $gb\kern-3pt\dominates$ here suggests an arc that is dominated by a sequence of good and bad arcs (with at least one bad arc). 
We observe that $\hat X\gb_I$ includes all weight on the vertical arcs).  
Our goal is to show that 
\begin{equation} \label{eq:gb+v}
\norm{\hat X\gb_I} \le \delta \norm{X_I}
\end{equation}
for $\epsilon \smallgiven{\bound, \delta}$ and $\norm{X_I} \largegiven{\bound, \delta, L_1}$;
we can then let $\hat Y_I = \hat X\vg_I$ and conclude \eqref{eq:most-well-dominated}.

To see \eqref{eq:gb+v},
we let $\hat X\gb_k$ be the sum of all $Z_i$'s where the first $j$ for which $\alpha_{ij}$ is not a pullback by $g$ of a good arc is equal to $k$. 

We observe that 
\begin{equation} \label{eq:kis1}
\norm{\hat X\gb_1} \le 4 \norm{X_I\bv} + C(\bound , L_1),
\end{equation} 
because each oriented bad (or vertical) arc from a given small Julia set has at most 2 preimages by $g$ that start at the same small Julia set (and there are two ways of orienting each arc). 

Let $E$ be a component $E$ of $\KK$ in $V_I$.
Then there are at most $6\bound$ (oriented) good arcs that end at $E$.
Each such arc starts at another component $E'$ of $\KK$,
and has two pullbacks by $g$ that start at $E'$; 
these are actually the only pullbacks by $g$ that start at any component of $\KK$.
We conclude that there are at most $12\bound$ components of $g\invp E$ in $\hat V_I$
that are connected to a component of $\KK$ by an arc in $\supp g^*X_I^g$. 

So we obtain 
\begin{equation}
\norm{\hat X\gb_2} \le 24\bound \norm{X_I\bv} + C(\bound , L_1).
\end{equation}
Similarly,
for each $E$ as before,
there are at most $(6\bound)^k$ ways of walking backwards along a sequence of $k$ oriented arcs,
starting the backward walk from $E$.
We therefore obtain
\begin{equation} \label{eq:genk}
\norm{\hat X\gb_{k+1}} \le 4(6\bound )^k \norm{X_I\bv} + C(\bound, k, L_1)
\end{equation}
 for each $k\ge 1$. 

Finally,
letting $X\gb_{\ge m} = \sum_{k = m}^\infty X\gb_k$, we have
\begin{equation} \label{eq:applywand}
\norm{\hat X\gb_{\ge m}} \le \frac{4}{m} \norm{X_I}
\end{equation}
by the Wanderers Theorem (\ref{thm:wanderers}).
Taking $m = \ceil{12/\delta}$ and $\epsilon < \frac19\delta / (12\bound)^m$ and $\wstd\largegiven{m, L_1, \delta}$, and applying \eqref{eq:kis1}, \eqref{eq:genk}, and \eqref{eq:applywand},
we obtain \eqref{eq:gb+v}.
\end{proof}


\subsection{From $\hat Y_I$ to $Y_I$ (and $Y_I$ to $Y$)} \label{align:immersion}
In this brief and final subsection we study the effect of the inclusion \eqref{eq:immersion},
and combine that with Proposition \ref{prop:most-well-dominated} to prove Proposition \ref{prop:interval-self-dominating}.

First, we have the following fact for relating weighted arc diagrams under by immersions:
\begin{theorem} \label{thm:immer-decomp}
Suppose that $i\from V' \to V$ is a holomorphic immersion of finite topology open Riemann surfaces, and $E'$ is a set of ends of $V'$ such that $i$ is proper, and has degree 1, at each end in $E'$.
Suppose that $Y$ is a valid WAD on $V$, 
such that all arcs in $Y$ have (at least) one endpoint in $E := i(E')$. 
Then there exists $Y'$ valid on $V'$ such that we can write 
\begin{equation} \label{eq:immer-decomp}
Y' \ge \sum_i v_i \beta'_i \qquad Y = \sum v_i \beta_i,
\end{equation}
and,
for each $i$,
either
\begin{itemize}
\item
both endponts of $\beta'_i$ lie in $E'$ and $i_*\beta'_i = \beta_i$,
\item
one endpoint of $\beta'_i$ lies in $E'$ and the other endpoint does not. 
\end{itemize}
\end{theorem}

\begin{proof}
Since $Y$ is valid, is the WAD for a partial proper foliation $F$ on $V$. 
We can orient the leaves of $F$ such that each leaf $\ell$ begins on an end in $E$.
We can then lift $\ell$, 
starting from an end in $E'$, 
until it leaves $V'$; we thereby obtain a proper leaf $\ell'$ in $V'$.
If the other endpoint of $\ell'$ lies in $E'$, 
then $i(\ell') = \ell$. 

We assemble the leaves $\ell'$ above into a partial proper foliation $F'$ on $V'$, and let $Y' = \WW(F')$.
We can also divide up $F$ into the leaves $\ell$ of a given homotopy class, and then further subdivide by the homotopy class of the resulting $\ell'$. In this way we can form the divisions of $Y$ and $Y'$ given in \eqref{eq:immer-decomp}, and, by the previous paragraph, it has the desired properties. 
\end{proof}

We call \eqref{eq:immer-decomp} a witness for Theorem \ref{thm:immer-decomp}.


%


\begin{proof}[Proof of Proposition \ref{prop:interval-self-dominating}]
We choose $\epsilon$ and $L_1$ as in Proposition \ref{prop:most-well-dominated},
with $\delta$ replaced by $\delta/2$;
we can also require that $L_1$ satisfy the hypothesis of Proposition \ref{lem:I-typical}
and  $L_1 > 8r$.
We can then choose a typical interval $I$ of length $L_1$ and then apply Proposition \ref{prop:most-well-dominated} to obtain a WAD $\hat Y_I$ on $\hat V_I \sm \hat \KK_I$ that satisfies \eqref{eq:most-well-dominated} (with $\delta$ replaced by $\delta/2$) and \eqref{eq:most-well-dominated-dominated}.

We let $X_I'$ be the portion of $X_I$ that is supported on arcs with at least one endpoint of $i(\hat\KK_I)$.
We can then apply Theorem \ref{thm:immer-decomp} to the immersion $i\from \hat V_I \sm \hat\KK_I \to V_I \sm \KK_I$,
letting $Y$ in that theorem be $X_I'$. 
We can then write $X'_I = X''_I+ Y_I$,
where 
\begin{itemize}
\item
$Y_I = i_*\hat Y_I$, and
\item
in the witness of Theorem \ref{thm:immer-decomp},
$X''_I$ corresponds to something less than or equal to $\hat X_I - \hat Y_I$. 
\end{itemize}

Then $\norm{X''_I} \le \norm{\hat X_I - \hat Y_I}$.
Moreover $\norm{X_I - X_I''} \le 4r \wstd$.
Combining these two inequalities with \eqref{eq:most-well-dominated},
we obtain
\begin{equation} \label{eq:what-we-wanted}
\norm{X_I - Y_I} \le \delta \norm{X_I}.
\end{equation}

Finally,
we observe that $Y_I$ is a lift of a (unique) WAD $Y$ on $\UU \sm \KK$, 
and $g^*(X^g|_I) \dominates Y$,
and $\norm{X|_I - Y} \le \norm{X_I - Y_I}$, so,
using the observation at the end of Section \ref{align:typical},
we obtain \eqref{eq:interval-self-dominating-control}.
\end{proof}

%% file: figs/dynamic_intervals.eps_tex
\begingroup%
  \makeatletter%
  \providecommand\color[2][]{%
    \errmessage{(Inkscape) Color is used for the text in Inkscape, but the package 'color.sty' is not loaded}%
    \renewcommand\color[2][]{}%
  }%
  \providecommand\transparent[1]{%
    \errmessage{(Inkscape) Transparency is used (non-zero) for the text in Inkscape, but the package 'transparent.sty' is not loaded}%
    \renewcommand\transparent[1]{}%
  }%
  \providecommand\rotatebox[2]{#2}%
  \newcommand*\fsize{\dimexpr\f@size pt\relax}%
  \newcommand*\lineheight[1]{\fontsize{\fsize}{#1\fsize}\selectfont}%
  \ifx\svgwidth\undefined%
    \setlength{\unitlength}{1041.05874369bp}%
    \ifx\svgscale\undefined%
      \relax%
    \else%
      \setlength{\unitlength}{\unitlength * \real{\svgscale}}%
    \fi%
  \else%
    \setlength{\unitlength}{\svgwidth}%
  \fi%
  \global\let\svgwidth\undefined%
  \global\let\svgscale\undefined%
  \makeatother%
  \begin{picture}(1,0.66072362)%
    \lineheight{1}%
    \setlength\tabcolsep{0pt}%
    \put(0,0){\includegraphics[width=\unitlength]{dynamic_intervals.eps}}%
    \put(0.49508958,0.59284245){\makebox(0,0)[lt]{\lineheight{1.25}\smash{\begin{tabular}[t]{l}$f$\end{tabular}}}}%
    \put(0.49300066,0.41700863){\makebox(0,0)[lt]{\lineheight{1.25}\smash{\begin{tabular}[t]{l}$\iota$\end{tabular}}}}%
    \put(0.49508958,0.27513135){\makebox(0,0)[lt]{\lineheight{1.25}\smash{\begin{tabular}[t]{l}$f_*$\end{tabular}}}}%
    \put(0.49300066,0.08565948){\makebox(0,0)[lt]{\lineheight{1.25}\smash{\begin{tabular}[t]{l}$\iota^*$\end{tabular}}}}%
  \end{picture}%
\endgroup%

%% file: figs/breaking.eps_tex
\begingroup%
  \makeatletter%
  \providecommand\color[2][]{%
    \errmessage{(Inkscape) Color is used for the text in Inkscape, but the package 'color.sty' is not loaded}%
    \renewcommand\color[2][]{}%
  }%
  \providecommand\transparent[1]{%
    \errmessage{(Inkscape) Transparency is used (non-zero) for the text in Inkscape, but the package 'transparent.sty' is not loaded}%
    \renewcommand\transparent[1]{}%
  }%
  \providecommand\rotatebox[2]{#2}%
  \newcommand*\fsize{\dimexpr\f@size pt\relax}%
  \newcommand*\lineheight[1]{\fontsize{\fsize}{#1\fsize}\selectfont}%
  \ifx\svgwidth\undefined%
    \setlength{\unitlength}{492.73465374bp}%
    \ifx\svgscale\undefined%
      \relax%
    \else%
      \setlength{\unitlength}{\unitlength * \real{\svgscale}}%
    \fi%
  \else%
    \setlength{\unitlength}{\svgwidth}%
  \fi%
  \global\let\svgwidth\undefined%
  \global\let\svgscale\undefined%
  \makeatother%
  \begin{picture}(1,0.47329661)%
    \lineheight{1}%
    \setlength\tabcolsep{0pt}%
    \put(0,0){\includegraphics[width=\unitlength]{breaking.eps}}%
    \put(0.15644294,0.02242361){\makebox(0,0)[lt]{\lineheight{1.25}\smash{\begin{tabular}[t]{l}$J_1$\end{tabular}}}}%
    \put(0.4897755,0.01971825){\makebox(0,0)[lt]{\lineheight{1.25}\smash{\begin{tabular}[t]{l}$I$\end{tabular}}}}%
    \put(0.81575407,0.02242361){\makebox(0,0)[lt]{\lineheight{1.25}\smash{\begin{tabular}[t]{l}$J_2$\end{tabular}}}}%
    \put(0.17542455,0.38810318){\makebox(0,0)[lt]{\lineheight{1.25}\smash{\begin{tabular}[t]{l}$\gamma_1$\end{tabular}}}}%
    \put(0.21806379,0.33373808){\makebox(0,0)[lt]{\lineheight{1.25}\smash{\begin{tabular}[t]{l}$\gamma_2$\end{tabular}}}}%
    \put(0.26603307,0.27724094){\makebox(0,0)[lt]{\lineheight{1.25}\smash{\begin{tabular}[t]{l}$\gamma_3$\end{tabular}}}}%
    \put(0.30867236,0.22287579){\makebox(0,0)[lt]{\lineheight{1.25}\smash{\begin{tabular}[t]{l}$\gamma_4$\end{tabular}}}}%
  \end{picture}%
\endgroup%

%% file: figs/lift.eps_tex
\begingroup%
  \makeatletter%
  \providecommand\color[2][]{%
    \errmessage{(Inkscape) Color is used for the text in Inkscape, but the package 'color.sty' is not loaded}%
    \renewcommand\color[2][]{}%
  }%
  \providecommand\transparent[1]{%
    \errmessage{(Inkscape) Transparency is used (non-zero) for the text in Inkscape, but the package 'transparent.sty' is not loaded}%
    \renewcommand\transparent[1]{}%
  }%
  \providecommand\rotatebox[2]{#2}%
  \newcommand*\fsize{\dimexpr\f@size pt\relax}%
  \newcommand*\lineheight[1]{\fontsize{\fsize}{#1\fsize}\selectfont}%
  \ifx\svgwidth\undefined%
    \setlength{\unitlength}{319.69124862bp}%
    \ifx\svgscale\undefined%
      \relax%
    \else%
      \setlength{\unitlength}{\unitlength * \real{\svgscale}}%
    \fi%
  \else%
    \setlength{\unitlength}{\svgwidth}%
  \fi%
  \global\let\svgwidth\undefined%
  \global\let\svgscale\undefined%
  \makeatother%
  \begin{picture}(1,0.41427898)%
    \lineheight{1}%
    \setlength\tabcolsep{0pt}%
    \put(0,0){\includegraphics[width=\unitlength]{lift.eps}}%
    \put(0.32651859,0.38000408){\makebox(0,0)[lt]{\lineheight{1.25}\smash{\begin{tabular}[t]{l}$\tilde A_2$\end{tabular}}}}%
    \put(0.41531833,0.20635916){\makebox(0,0)[lt]{\lineheight{1.25}\smash{\begin{tabular}[t]{l}$\tilde X$\end{tabular}}}}%
    \put(0.00402521,0.21183692){\makebox(0,0)[lt]{\lineheight{1.25}\smash{\begin{tabular}[t]{l}$\tilde Y$\end{tabular}}}}%
    \put(0.09322369,0.01539218){\makebox(0,0)[lt]{\lineheight{1.25}\smash{\begin{tabular}[t]{l}$\tilde C$\end{tabular}}}}%
    \put(0.50059559,0.34082978){\makebox(0,0)[lt]{\lineheight{1.25}\smash{\begin{tabular}[t]{l}$\iota$\end{tabular}}}}%
    \put(0.1026123,0.38164554){\makebox(0,0)[lt]{\lineheight{1.25}\smash{\begin{tabular}[t]{l}$\tilde B$\end{tabular}}}}%
    \put(0.98686433,0.20793766){\makebox(0,0)[lt]{\lineheight{1.25}\smash{\begin{tabular}[t]{l}$A$\end{tabular}}}}%
    \put(0.67400399,0.38668144){\makebox(0,0)[lt]{\lineheight{1.25}\smash{\begin{tabular}[t]{l}$B$\end{tabular}}}}%
    \put(0.67785583,0.01505758){\makebox(0,0)[lt]{\lineheight{1.25}\smash{\begin{tabular}[t]{l}$C$\end{tabular}}}}%
    \put(0.31884564,0.01894996){\makebox(0,0)[lt]{\lineheight{1.25}\smash{\begin{tabular}[t]{l}$\tilde A_1$\end{tabular}}}}%
  \end{picture}%
\endgroup%

%% file: macros/biblio.bib
@book {McMullen94,
	AUTHOR = {McMullen, Curtis T.},
	TITLE = {Complex dynamics and renormalization},
	SERIES = {Annals of Mathematics Studies},
	VOLUME = {135},
	PUBLISHER = {Princeton University Press, Princeton, NJ},
	YEAR = {1994},
	PAGES = {x+214},
	ISBN = {0-691-02982-2; 0-691-02981-4},
	MRCLASS = {58F23 (30D05)},
	MRNUMBER = {1312365},
	MRREVIEWER = {Gregery\ T.\ Buzzard},
}

@book {Orsay,
	AUTHOR = {Douady, A. and Hubbard, J. H.},
	TITLE = {\'{E}tude dynamique des polyn\^{o}mes complexes.},
	SERIES = {Publications Math\'{e}matiques d'Orsay [Mathematical Publications
	of Orsay]},
	VOLUME = {85},
	NOTE = {With the collaboration of P. Lavaurs, Tan Lei and P. Sentenac},
	PUBLISHER = {Universit\'{e} de Paris-Sud, D\'{e}partement de Math\'{e}matiques, Orsay},
	YEAR = {1985},
	PAGES = {v+154},
	MRCLASS = {58F08 (30D05 39B10)},
	MRNUMBER = {812271},
	MRREVIEWER = {M. Rees},
}

@incollection {Hubbard-Yoccoz,
	AUTHOR = {Hubbard, J. H.},
	TITLE = {Local connectivity of {J}ulia sets and bifurcation loci: three
	theorems of {J}.-{C}. {Y}occoz},
	BOOKTITLE = {Topological methods in modern mathematics ({S}tony {B}rook,
	{NY}, 1991)},
	PAGES = {467--511},
	PUBLISHER = {Publish or Perish, Houston, TX},
	YEAR = {1993},
	MRCLASS = {58F23 (28A80 30C10 58F12)},
	MRNUMBER = {1215974},
	MRREVIEWER = {Christoph\ Bandt},
}

@incollection {BD86,
	AUTHOR = {Branner, Bodil and Douady, Adrien},
	TITLE = {Surgery on complex polynomials},
	BOOKTITLE = {Holomorphic dynamics ({M}exico, 1986)},
	SERIES = {Lecture Notes in Math.},
	VOLUME = {1345},
	PAGES = {11--72},
	PUBLISHER = {Springer, Berlin},
	YEAR = {1988},
	ISBN = {3-540-50226-2},
	MRCLASS = {58F14 (30C60)},
	MRNUMBER = {980952},
	MRREVIEWER = {M.\ Lyubich},
	DOI = {10.1007/BFb0081395},
	URL = {https://doi.org/10.1007/BFb0081395},
}

@article {DL_pacmen,
	AUTHOR = {Dudko, Dzmitry and Lyubich, Mikhail},
	TITLE = {Local connectivity of the {M}andelbrot set at some satellite
	parameters of bounded type},
	JOURNAL = {Geom. Funct. Anal.},
	FJOURNAL = {Geometric and Functional Analysis},
	VOLUME = {33},
	YEAR = {2023},
	NUMBER = {4},
	PAGES = {912--1047},
	ISSN = {1016-443X,1420-8970},
	MRCLASS = {37F10},
	MRNUMBER = {4616693},
	DOI = {10.1007/s00039-023-00637-8},
	URL = {https://doi.org/10.1007/s00039-023-00637-8},
}

@ARTICLE{CS_satellite,
	author = {Cheraghi, Davoud and Shishikura, Mitshurio},
	title = "{Satellite renormalization of quadratic polynomials}",
	journal = {arXiv e-prints},
	year = 2015,
	month = sep,
	eid = {arXiv.1509.07843},
	pages = {arXiv.1509.07843},
	archivePrefix = {arXiv},
	eprint = {1509.07843},
	primaryClass = {math.DS},
	adsurl = {https://arxiv.org/abs/1509.07843v1},
}

@article{IS,
	title = {The renormalization for parabolic fixed points and their
	perturbation},
	author = {Inou, H. and Shishikura, M.},
	Year = {2008},
	Journal = {Manuscript},
}

@article {DLS_pacman,
	AUTHOR = {Dudko, Dzmitry and Lyubich, Mikhail and Selinger, Nikita},
	TITLE = {Pacman renormalization and self-similarity of the {M}andelbrot
	set near {S}iegel parameters},
	JOURNAL = {J. Amer. Math. Soc.},
	FJOURNAL = {Journal of the American Mathematical Society},
	VOLUME = {33},
	YEAR = {2020},
	NUMBER = {3},
	PAGES = {653--733},
	ISSN = {0894-0347,1088-6834},
	MRCLASS = {37E20 (37F25 37F46)},
	MRNUMBER = {4127901},
	MRREVIEWER = {Matthieu\ Astorg},
	DOI = {10.1090/jams/942},
	URL = {https://doi.org/10.1090/jams/942},
}

@misc{DLneutral,
	title={Uniform a priori bounds for neutral renormalization}, 
	author={Dzmitry Dudko and Mikhail Lyubich},
	year={2024},
	eprint={2210.09280},
	archivePrefix={arXiv},
	primaryClass={math.DS},
	url={https://arxiv.org/abs/2210.09280}, 
}

@unpublished{DKLP:elephants,
	author = {Dudko, Dzmitry and Kapiamba, Alex and  Lyubich, Mikhail and Peterson, Carsten},
	note = {In preparation},
	title = {Elephant puzzles}
}


%% file: macros/everything.bib
@preamble{"\def\cprime{$'$} "}

@unpublished{DKL:virtual,
	author = {Dzmitry Dudko and Jeremy Kahn and Mikhail Lyubich},
	note = {In preparation},
	title = {Virtual renormalization: interpolating between primitive and satelite regimes}}

@article{KL:QA,
	author = {Kahn, Jeremy and Lyubich, Mikahil},
	coden = {ANMAAH},
	doi = {10.4007/annals.2009.169.561},
	fjournal = {Annals of Mathematics. Second Series},
	issn = {0003-486X},
	journal = {Ann. of Math. (2)},
	mrclass = {37F10 (30F20 37F35)},
	mrnumber = {2480612 (2010a:37091)},
	mrreviewer = {Peter Ha{\"{\i}}ssinsky},
	number = {2},
	pages = {561--593},
	title = {The quasi-additivity law in conformal geometry},
	url = {http://dx.doi.org/10.4007/annals.2009.169.561},
	volume = {169},
	year = {2009}}

@article{KL:decorations,
	author = {Kahn, Jeremy and Lyubich, Mikhail},
	fjournal = {Annales Scientifiques de l'\'Ecole Normale Sup\'erieure. Quatri\`eme S\'erie},
	issn = {0012-9593},
	journal = {Ann. Sci. \'Ec. Norm. Sup\'er. (4)},
	mrclass = {37F25 (37F10 37F45 37F50)},
	mrnumber = {2423310 (2009k:37106)},
	mrreviewer = {Volker Mayer},
	number = {1},
	pages = {57--84},
	title = {A priori bounds for some infinitely renormalizable quadratics. {II}. {D}ecorations},
	volume = {41},
	year = {2008}}

@incollection{KL:molecules,
	address = {Wellesley, MA},
	author = {Kahn, Jeremy and Lyubich, Mikhail},
	booktitle = {Complex dynamics},
	mrclass = {37F25 (37F10 37F45 37F50)},
	mrnumber = {2508259 (2010f:37078)},
	mrreviewer = {Volker Mayer},
	pages = {229--254},
	publisher = {A K Peters},
	title = {A priori bounds for some infinitely renormalizable quadratics. {III}. {M}olecules},
	year = {2009}}


%% file: macros/mlc.bib
@article{lyubich1997dynamics,
	author = {Lyubich, Mikhail},
	journal = {Acta Mathematica},
	number = {2},
	pages = {185--297},
	publisher = {International Press of Boston},
	title = {{ Dynamics of quadratic polynomials, I--II}},
	volume = {178},
	year = {1997}}

@misc{dudko2023mlcfeigenbaumpoints,
	archiveprefix = {arXiv},
	author = {Dzmitry Dudko and Mikhail Lyubich},
	eprint = {2309.02107},
	primaryclass = {math.DS},
	title = {{MLC at Feigenbaum points}},
	url = {https://arxiv.org/abs/2309.02107},
	year = {2023}}


%% file: papers.bib
@misc{KL:eyes,
      title={A priori bounds for some infinitely renormalizable quadratic: IV. Elephant Eyes}, 
      author={Jeremy Kahn and Misha Lyubich},
      year={2026},
      eprint={2601.21905},
      archivePrefix={arXiv},
      primaryClass={math.DS},
      url={https://arxiv.org/abs/2601.21905}, 
}

@misc{K06,
      title={A priori bounds for some infinitely renormalizable quadratics: I. Bounded primitive combinatorics}, 
      author={Jeremy Kahn},
      year={2006},
      eprint={math/0609045},
      archivePrefix={arXiv},
      primaryClass={math.DS},
      url={https://arxiv.org/abs/math/0609045}, 
}

@inbook{Milnor,
place={Cambridge}, series={London Mathematical Society Lecture Note Series}, title={Local connectivity of Julia sets: expository lectures}, booktitle={The Mandelbrot Set, Theme and Variations}, publisher={Cambridge University Press}, author={Milnor, John}, editor={Lei, TanEditor}, year={2000}, pages={67–116}, collection={London Mathematical Society Lecture Note Series}
}

@misc{dudkoMLC,
      title={On the MLC Conjecture and the Renormalization Theory in Complex Dynamics}, 
      author={Dzmitry Dudko},
      year={2025},
      eprint={2512.24171},
      archivePrefix={arXiv},
      primaryClass={math.DS},
      url={https://arxiv.org/abs/2512.24171}, 
}

@article {AL22,
    AUTHOR = {Avila, Artur and Lyubich, Mikhail},
     TITLE = {Lebesgue measure of {F}eigenbaum {J}ulia sets},
   JOURNAL = {Ann. of Math. (2)},
  FJOURNAL = {Annals of Mathematics. Second Series},
    VOLUME = {195},
      YEAR = {2022},
    NUMBER = {1},
     PAGES = {1--88},
      ISSN = {0003-486X,1939-8980},
   MRCLASS = {37F10 (37F25)},
  MRNUMBER = {4358413},
MRREVIEWER = {Kevin\ M.\ Pilgrim},
       DOI = {10.4007/annals.2022.195.1.1},
       URL = {https://doi-org.ezp-prod1.hul.harvard.edu/10.4007/annals.2022.195.1.1},
}

@article {BC12,
    AUTHOR = {Buff, Xavier and Ch\'eritat, Arnaud},
     TITLE = {Quadratic {J}ulia sets with positive area},
   JOURNAL = {Ann. of Math. (2)},
  FJOURNAL = {Annals of Mathematics. Second Series},
    VOLUME = {176},
      YEAR = {2012},
    NUMBER = {2},
     PAGES = {673--746},
      ISSN = {0003-486X,1939-8980},
   MRCLASS = {37F50},
  MRNUMBER = {2950763},
MRREVIEWER = {Peter\ Ha\"issinsky},
       DOI = {10.4007/annals.2012.176.2.1},
       URL = {https://doi-org.ezp-prod1.hul.harvard.edu/10.4007/annals.2012.176.2.1},
}
